\documentclass[11pt]{article}

\usepackage{grffile}

\usepackage{amsmath,amsfonts,amscd,a4wide}
\usepackage{amsmath}
\usepackage{amssymb}
\usepackage{graphicx}
\usepackage{amsthm}
\usepackage{color}

\usepackage{chngcntr}
\usepackage{apptools}
\AtAppendix{\counterwithin{Lemma}{section}}

\usepackage[letterpaper,textheight=9in,textwidth=7in]{geometry}
\usepackage{color}
\usepackage{transparent}
\usepackage[skip=10pt,font=footnotesize]{caption}
\usepackage{subcaption}
\usepackage{enumitem}

\graphicspath{{Figures/}}

\newtheorem{Lemma}{Lemma}[section]
\newtheorem{Theorem}{Theorem}
\newtheorem{Proposition}[Lemma]{Proposition}
\newtheorem{Corollary}[Lemma]{Corollary}
\newtheorem{Remark}[Lemma]{Remark}

 {\begin{trivlist} \item[]{\bf Proof. }}%
 {\hspace*{\fill}$\rule{.4\baselineskip}{.4\baselineskip}$\end{trivlist}}

\newenvironment{Acknowledgment}%
 {\begin{trivlist}\item[]\textbf{Acknowledgments.}}{\end{trivlist}}

\makeatletter\@addtoreset{figure}{section}\makeatother

\makeatletter \@addtoreset{equation}{section} \makeatother

\newcommand{\R}{\mathbb{R}}
\newcommand{\C}{\mathbb{C}}
\newcommand{\N}{\mathbb{N}}
\newcommand{\Z}{\mathbb{Z}}
\newcommand{\mD}{\mathbb{D}}

        \newcommand{\mc}[1]{\mathcal{#1}}
        \newcommand{\mb}[1]{\mathbb{#1}}
         
        \newcommand{\tl}[1]{\tilde{#1}}
        \newcommand{\ul}[1]{\underline{#1}}
        \newcommand{\lp}{\left}
        \newcommand{\rp}{\right}
                \newcommand{\la}{\lp\langle}
        \newcommand{\ra}{\rp\rangle}
        \newcommand{\beq}{\begin{equation}}
        \newcommand{\eeq}{\end{equation}}
        \newcommand{\ba}{\begin{align}}
        \newcommand{\ea}{\end{align}}
        \newcommand{\fr}[2]{\frac{#1}{#2}}
        \newcommand{\p}{\partial}
        \newcommand{\ri}{\mathrm{i}}

        \newcommand{\re}{\mathrm{e}}

        \renewcommand{\Re}{\mathrm{Re}}
        
        \newcommand{\rre}{\mathrm{Re}}

         \newcommand{\bth}{\mathbf{\theta}}

        \newcommand{\mfu}{\mathfrak{u}}

\title{Vortices in stably-stratified rapidly rotating Boussinesq convection}
\author{Ryan Goh and C. Eugene Wayne}

\begin{document}

\maketitle
\abstract{

 We study the Boussinesq approximation for rapidly rotating stably-stratified fluids in a three dimensional infinite layer with either stress-free or periodic boundary conditions in the vertical direction. For initial conditions satisfying a certain  quasi-geostrophic smallness condition, we use dispersive estimates and the large rotation limit to prove global-in-time existence of solutions.  We then use self-similar variable techniques to show that the barotropic vorticity converges to an Oseen vortex, while other components decay to zero.  We finally use algebraically weighted spaces to determine leading order asymptotics. In particular we show that the barotropic vorticity approaches the Oseen vortex with algebraic rate while the barotropic vertical velocity and thermal fluctuations go to zero as Gaussians whose amplitudes oscillate in opposite phase of each other while decaying with an algebraic rate.

\section{Introduction}

\subsection{Background}
The rotating Boussinesq equations have been widely used to study the effects of rotation and density stratification on flow dynamics in the Earth's atmosphere and oceans, as well as in many other geophysical settings. These equations are an approximation of the compressible, rotating Navier-Stokes equations where one assumes weak density variations in only the equation of state and the buoyancy term. Furthermore, one assumes that the density varies linearly with temperature, and, if one is studying ocean dynamics, salinity. 

After such approximations and nondimensionalization, these equations take the form of the incompressible Navier-Stokes equation posed in a non-inertial rotating frame coupled with an equation for temperature fluctuations in the stratified fluid,
\begin{align}\label{e:000}
\p_t u&= \nu \Delta u - u\cdot \nabla u - \Omega e_3\times u -\nabla p + g \theta e_3,\\
\p_t \theta &= \nu' \Delta \theta  - u\cdot\nabla \theta + (\fr{\p\bar\rho}{\partial x_3}) u_3,\\
 0&=\mathrm{div}\, u,\quad u = (u_1,u_2,u_3)\in \R^3, \, x\in \R^3, \, t>0.
\end{align}
Here $u$ describes the velocity field of the fluid, $g>0$ is a gravitational constant, $\nu>0$ represents the kinematic viscosity, and $\nu'>0$ the thermal viscosity. The term $\Omega e_3\times u$ arises from the effect of the Coriolis force due to the rotating frame. The angular frequency $\Omega$ is also proportional to the inverse of the Rossby number which measures the rate of rotation relative to the characteristic length scale of the fluid.  While $\Omega$ is in general $x$-dependent (say for spherical geometry in the case of the atmosphere of the Earth), it suffices in many situations to use the ``f-plane" approximation where $\Omega$ is constant. We will assume this for the rest of the work.
  Next $\theta$ and $p$ denote thermal and pressure fluctuations about a horizontally homogeneous mean state.  We assume that the mean velocity is zero, the mean density varies linearity with $x_3$, and that mean density and mean pressure are in \emph{hydrostatic} balance, where the pressure gradient is balanced by the force of gravity on the fluid,
$$
-\fr{\p}{\p x_3} \bar p(x_3) =  g\bar \rho(x_3).
$$
This approximation arises from the fact than in typical applications of this model, such as the Earth's oceans, the horizontal scale is much larger than than the vertical scale; see \cite{pedlosky2013geophysical} for more explanation. If $\fr{d\bar\rho}{dx_3}>0$, so that the fluid is convective with less dense fluid lying below more dense fluid (say the fluid is heated from below for example),  the system is said to be unstably-stratified.  See more Section \ref{ss:unst} for more discussion. 
 
In this work we focus on the case $\fr{d\bar\rho}{dx_3}<0$, where less dense fluid is above more dense fluid,  known as stable stratification.  The Brunt-V\"{a}is\"{a}l\"{a} buoyancy frequency $N:= \sqrt{-d\bar\rho/dx_3}$ denotes the frequency of the neutral oscillation of a small ``parcel" of vertically unstable fluid feeling the effects of buoyancy and gravity inside a uniform background of stably-stratified fluid; see \cite[\S 1.3]{majda2003introduction} for more discussion on this.  Substituting these quantities into \eqref{e:000} we obtain
   \begin{align}\label{e:00}
\p_t u&= \nu \Delta u - u\cdot \nabla u - \Omega e_3\times u -\nabla p + g \theta e_3,\\
\p_t \theta &= \nu' \Delta \theta  - u\cdot\nabla \theta - N^2 u_3,\\
0&=\mathrm{div}\, u.
\end{align}
For a full derivation of this system see \cite{gill1982atmosphere}. For general discussion on these equations, the many interesting physical phenomenon they model, and its various approximating limits, see \cite{pedlosky2013geophysical,riley2000fluid, gill1982atmosphere}.  For a more mathematically focused introduction and review of the subject see \cite{majda2003introduction, lions1992equations,chemin2006mathematical, li2007some,rademacher2017}.

\subsection{Overview of mathematical literature}
Due to the wealth of physical applications and the historical inaccessibility of various asymptotic parameter limits in numerical simulations, a sizeable body of mathematical research has been performed on this system and its many variants, with many interesting analytical tools brought to bear on the problem.  In addition, the fact that dynamical systems theory has often proven a valuable tool in analyzing the behavior of these equations means that they have featured in Nonlinearity's pages since its earliest days.   For instance, papers \cite{lions1992new} and \cite{lions1992equations} both treat models very similar to that treated here, and discuss their applicability to questions of ocean and atmospheric dynamics.

In order to discuss some of the key issues in the analysis of this system we set $\nu = \nu' = 1$; see Remark \ref{r:nunu} for discussion on the $\nu\neq\nu'$ and $\nu=\nu'\neq1$ cases.   Following \cite{herring_metais_1989,babin1999regularity,iwabuchi2016global}, to obtain a skew-symmetric operator  we scale $\theta \mapsto \fr{\sqrt{g}}{N}\theta,$ and denote $\Gamma = N\sqrt{g}$.   We also combine all the dependent variables into a single-four vector $v = (u_1,u_2.u_3,\theta)$.  Finally, applying the Helmholtz projection $\mb{P}$ defined in \eqref{e:helm} below,  we obtain from \eqref{e:00} the formulation
\beq\label{e:00d}
\p_t v + \mb{P}(u\cdot\tl\nabla)v - \mathrm{diag}(\Delta)v + \mb{P}J_{\Omega,\Gamma}\mb{P}v = 0,\quad \mathrm{div} u = 0
\eeq 
with $\tl\nabla = (\p_1,\p_2,\p_3,0)^T,$ and $\mb{P}J_{\Omega,\Gamma}\mb{P}$ is the operator formed by the Coriolis, buoyancy, and gravity effects with,
$$
J_{\Omega,\Gamma} = \left(\begin{array}{cc}\Omega J & 0 \\0 & \Gamma J\end{array}\right),\quad J=\left(\begin{array}{cc}0 & -1 \\1 & 0\end{array}\right).
$$

As discussed in \cite{iwabuchi2016global} (and reviewed below), there is one physically relevant eigendirection,
which we call the ``quasi-geostrophic mode'', which 
undergoes no dispersive smoothing.  However, the other eigendirections correspond to rapidly oscillating waves known
as ``inertial'' or ``Poincar\'e waves''.  While these do not decay in $L^2$ they do decay in $L^p$ spaces with $p>2$.
This can be quantified with the aid of Strichartz estimates.  One then attempts to control such oscillatory behavior and determine
the long-time, quasi-geostrophic dynamics of the system.  Such dynamics are briefly discussed below.  See the book by Chemin et. al  \cite{chemin2006mathematical} for more discussion on the role of inertial wave phenomenon in the context of a rotating incompressible Navier-Stokes system without density effects. There they also illustrate interesting connections of the $\Omega\rightarrow\infty$ limit with the incompressible limit (low mach number) of compressible Navier-Stokes.  Also see the text of Majda which discusses these topics in the context of invisid shallow-water wave equations \cite{majda2003introduction}.

Global well-posedness and regularity has been one of the main focuses of mathematical research on this system.  Here results date back to the works of Lions, Temam, and Wang \cite{lions1992new,lions1992equations} which proved global existence for weak solutions and estimates on the dimensions of global attractors, for a related model, known as the ``primitive" equations, which focused on ocean dynamics. This model includes an equation for water salinity and considered the system in a thin layer on a sphere, $S^2\times (-H,0)$ where the depth, $H$ was dependent on the spherical variables.  The primitive equations are obtained from \eqref{e:00} under the hydostatic approximation, where one uses the disparity between the vertical and horizontal length scales in the Earth's ocean, to replace the equation for $u_3$ with the leading order balance $\fr{\p p}{\p x_3} = - \theta g$.   Subsequent works (see \cite{ cao2007global, temam2005some,hu2003primitive,guillen2001anisotropic} and references therein) generalized these results, extending them to strong solutions (both small- and global-time existence) and to more general domains.

The works of Babin, Mahalov, and Nicolenko \cite{babin1998nonlinear, babin1999regularity,babin2002fast} considered systems of the form \eqref{e:00} with an additional forcing term, posed on fully periodic domains $x\in [0,2\pi a_1]\times[0,2\pi a_2]\times[0,2\pi a_3]$ or horizontally periodic domains with ``stress-free" boundary conditions in the vertical direction:
\beq
u_3 = \theta = 0, \quad \p_{x_3}u_1 =\p_{x_3} u_2 = 0,\quad x = 0,2\pi a_3. \label{e:00stf}
\eeq
There, inertial waves cannot ``escape to infinity" due to the bounded domain, possibly inducing resonant three dimensional wave interactions.  Small divisor techniques are required to prove global existence and regularity of strong solutions, along with the existence of attractors, for certain non-resonant domains in this setting. In these works, they also studied how, in the joint limit $\Omega = 1/\epsilon,N^2 = \tl N/\epsilon, \, \epsilon\rightarrow0$ solutions of \eqref{e:00} can be uniformly approximated in time by solutions of the quasi-geostrophic equations,  a classic and often-used model first derived by Charney \cite{charney1948scale} for the slow dynamics of the atmosphere. 

 There also has been an intense study of the full Boussinesq system \eqref{e:00} in the whole space $x\in\R^3.$  In a series of works \cite{charve2008global,charve2011global,charve2006asymptotics,charve2005global}, Charve has shown convergence to the quasi-geostrophic system in the whole space $x\in \R^3$ for both weak and strong solutions of various regularity and initial condition size, in the same limit $\epsilon\rightarrow0$ mentioned above.   The idea underlying all these
approaches is to decompose the solution as a part governed by the quasi-geostrophic system and a ``remainder'' and to show that the oscillatory effects arising from the fast rotation allow one to prove existence and uniqueness of solutions in three-dimensions, even for large initial data.

Subsequent results such as \cite{koba2012global} extend the use of dispersive estimates in $\R^3$ to other  $L^p$ spaces.  The most recent result \cite{iwabuchi2016global} investigates more deeply the connections between harmonic analysis and the principal curvatures of the linear dispersion relation $p_\eta( k)$; see \cite{greenleaf1981principal} for example. In particular, degeneracies of the Hessian, $(\p^2_{ k_i k_j}p_\eta)_{ij}$, of the dispersion relation are used to study more general initial conditions, only requiring smallness in the quasi-geostrophic component. We remark that because of the smallness of the quasi-geostrophic initial condition these last two results omit study of the quasi-geostrophic limit as in Charve's work. See \cite{iwabuchi2016global} for a more detailed discussion of these topics and review of the related literature.

\subsection{The role of dynamical systems}
Dynamical systems has also been integral to the study geophysical fluid dynamics.  In fact the study of such fluid models has been one of the main motivations for the development of many mathematical tools such as global attractors, inertial manifolds, and slow manifolds.   Again, Nonlinearity was a leader in this work and already, in the second volume, dynamical systems methods were exploited to study the stability of stratified fluids \cite{holm1989lyapunov}.

\paragraph{Slow manifolds}
During the 1970's and 1980's many researchers, including Lorenz \cite{lorenz1986existence}, posited the existence of a slow invariant manifold in the Boussinesq system which controlled and organized the dynamics. In an effort to understand the interplay between fast, oscillatory gravity-waves and the slow geostrophic dynamics for $\Omega,\Gamma\sim \mc{O}(1/\epsilon)$ with small $\epsilon$, researchers worked to discover if there was a manifold in the (infinite-dimensional) phase space, which was invariant under the time evolution, devoid of fast waves, and contained the slowly evolving geostrophic dynamics which governs much of the long term behavior of the system. Such manifolds would then be of use as they could give information on the global attractor of the system.  While the current consensus is that no such exact manifold exists \cite{lorenz1987nonexistence}, there has been rigorous work characterizing approximately invariant, or ``fuzzy", manifolds when the Boussinesq system undergoes certain types of forcing \cite{temam2011slow, temam2010stability, petcu2005renormalization}.  These works have shown that in finite-time the fast dynamics decay and stay $\mc{O}(\epsilon)$-small, so that the dynamics stay near geostrophic balance. For a nice review of this topic, the important work in the area, and some of the most pressing open questions, see \cite{temam2011slow}.

\paragraph{Scaling variables, invariant manifolds, and omega-limit sets}  
In work that is the closest to our own in both technique and aims, Gallay and Roussier-Michon have used dynamical systems techniques to study long-time asymptotics of the rotating incompressible Navier-Stokes system
\beq
\p_t u + (u\cdot \nabla )u + \Omega e_3\times u -\Delta u = 0,\quad \mathrm{div}\, u = 0.
\eeq
In \cite{gallay2009global}, they considered the above system posed on $\R^2\times[0,1]$ with periodic boundary conditions in the vertical direction.  They showed, for any size initial condition and a correspondingly large $\Omega$, the existence of global, infinite energy solutions which converge to the $x_3$-independent diffusively decaying Oseen vortex solution.  In other words, they showed for high rotation rates, solutions asymptotically behave like the non-rotating two-dimensional Navier-Stokes equations.  Using the barotropic-baroclinic decomposition described in \cite{chemin2006mathematical}, which splits the vector field into an $x_3$-independent component, $\bar u$, and an $x_3$-mean-zero component, $\tl u$, one obtains a system of three equations, one each for the horizontal vorticity $\bar \omega_3   = \p_1\bar u_2 - \p_2 \bar u_1 = \int_0^1 (\nabla \times u)_3 dx_3$, the vertical barotropic velocity $\bar u_3$, and the baroclinic velocity $\tl u.$ They obtain exponential decay of the baroclinic component using Poincare's inequality (due to vertical mean-zero and boundedness of vertical domain) and Strichartz estimates.  After using linear convection-diffusion estimates of Carlen and Loss \cite{carlen1996optimal} to show the algebraic decay of $\nabla \bar u_3$, they then study the barotropic vorticity which satisfies an equation very similar to the vorticity formulation of two-dimensional Navier Stokes.  Taking the approach of Gallay and Wayne \cite{gallay2005global}, they can employ scaling variables and compactness arguments to determine the omega-limit set of the system, showing that $\bar \omega$, and hence $u$ itself, asymptotically converge to the Oseen vortex solution as $t\rightarrow\infty$.    However, they do not consider the effects of the coupling between the velocity field and the temperature, nor do they derive detailed asymptotics for a solution near a vortex, like those in Theorem \ref{t:3}.

\subsection{Un-stably stratified fluids}\label{ss:unst}
The un-stably stratified system 
\begin{align}\label{e:00un}
\p_t u&= \Delta u - u\cdot \nabla u - \Omega e_3\times u + R \theta e_3,\\
\p_t\theta&= \Delta \theta - u\cdot \nabla \theta +  u_3,
\end{align}
has also been the focus of much recent study.  Here the Prandtl number $\nu/\nu'$ is assumed to be one for simplicity, and $R$ is the Rayleigh number, measuring the effects of gravity, thermal expansion, and the density difference in the background linear density profile.  Physical boundary conditions once again take the form \eqref{e:00stf}. Such models typically arise when studying more extreme systems, such as rapidly-rotating convective atmospheric layers and are contained in the more complicated magneto-hydrodynamic models used to study stellar plasmas.  See \cite{kraichnan1980two, boffetta2012two} for a review of these applications.

The convective nature of the unstable stratification leads to a rich family of structures and dynamics for various parameter ranges.  Experimental and numerical studies of \eqref{e:00un}, with various geometries and boundary conditions, have exhibited a wealth of interesting behaviors, such as convection cells, localized plumes, and large-scale turbulent vortices as the Rayleigh number is varied.  See the introduction of \cite{favier2014inverse} for a nice review of the numerical and experimental literature and more discussion on this topic. 

 In the case of rapid rotation, $\Omega\gg1$, the Coriolis term suppresses behaviors characteristic of three-dimensional incompressible fluids, such as direct turbulence cascades.  This causes the system to behave similarly to a forced two-dimensional fluid, where an inverse cascade dominates turbulent dynamics; see once again \cite{kraichnan1980two, boffetta2012two}.  Assuming large enough Rayleigh number, $R$, the thermal forcing creates small-scale eddies which become vertically ``aligned" and coalesce into large domain-scale turbulent vortices.

Mathematically, much less is known about these equations compared with the stably-stratified case.  While there are some results on existence and bifurcation of finite dimensional attractors in fully periodic domains $\mb{T}^3$ \cite{li2007some,hsia2007stratified}, little work has been done to rigorously characterize coherent structures which arise in these systems.  It would be interesting to determine if an invariant or approximately invariant slow-manifold existed in the system.  Given the previous literature and discussion above, a reasonable candidate for a slow variable is the the barotropic component of the full three-dimensional velocity,
$$
\bar v = \int_0^1 v(x_h,x_3) d x_3.
$$  
The boundary conditions on $\theta$ and $u_3$ immediately give that $\bar u_3 = \bar \theta = 0$. Then denoting the projection associated with this decomposition as  $Q$, the equation for the barotropic component $\bar v = (\bar u_h,0,0)^T$ can readily be found to be
\begin{align}
\p_t \bar u_h + (\bar u_h\cdot \nabla_h)\bar u_h  + Q\nabla_h p = \Delta \bar u_h - Q\lp[(\tl u\cdot \tl\nabla) \tl u_h\rp],
\end{align}
where $\tl v = (1 - Q) v$ is the baroclinic part of the system.  Thus $\bar u_h$ satisfies a forced 2-dimensional Navier-Stokes equation with forcing term $Q\lp[(\tl u_h\cdot\nabla_h )\tl u_h\rp]$, coming from the baroclinic system (which we have not written down here).  One would hope to show that in the high rotation limit, $\Omega = \epsilon^{-1},\,\,0<\epsilon\ll1$ (possibly also scaling the Rayleigh number $R = \tl R/\epsilon$ as well), the Coriolis force overpowers the unstable stratification term causing the baroclinic variables $\tl v$ to decay to $\mc{O}(\epsilon)$ sizes after a finite-time and remain small for all subsequent times, as in \cite{temam2011slow}.  Hence, the barotropic subspace $\{\tl v = 0\}$ would form an approximately invariant manifold for the full dynamics. One would then hope to characterize the dynamics of $\bar u_h$, either explicitly characterizing large-scale vortices or at the very least showing the existence of, and characterizing the attractor.  As forced two-dimensional turbulence is in general only understood statistically (see \cite{boffetta2012two}) this approach seems to be an interesting line of research, which we shall pursue in the future.

The system \eqref{e:00un} is in general difficult to simulate, with direct numerical simulations for large Rayleigh number only being performed in the last few years; see for example \cite{favier2014inverse}. Julien and Knobloch with various collaborators have derived a formal asymptotic quasi-geostrophic model which eliminates fast inertial waves and certain boundary effects, known as Ekman layers, and is thus much more tractable numerically while still retaining much of the three-dimensional dynamics (as opposed to a 2-D Navier-Stokes system with arbitrary forcing).  In a series of works \cite{sprague2006numerical,julien2012statistical,rubio2014upscale}, they have shown this model exhibits various coherent structures and similar statistical behaviors observed in experiments and recent simulations of the full system \eqref{e:00un}.   It would be of great interest to study such coherent structures of in this asymptotic model rigorously and consequently to rigorously study the asymptotic convergence of \eqref{e:00un} to this formal model in a way similar to that of the quasi-geostrophic approximation of the stably-stratified system discussed above.

\subsection{Our results}\label{s:results}
 This paper focuses on determining leading order asymptotics of \eqref{e:00d}.  We consider the stably-stratified Boussinesq system \eqref{e:00d} above, posed on $\mb{D}=\R^2\times[0,1]$ with either stress-free boundary conditions
 \begin{align}\label{e:sfbc}
\p_{x_3} u_1 = \p_{x_3} u_2 = u_3 = 0, \quad\theta= 0, \quad x_3 = 0,1,
\end{align}
or periodic boundary conditions
\beq\label{e:pbc}
u(x_h,x_3) = u(x_h,x_3+1), \quad \theta(x_h,x_3) = \theta(x_h,x_3+1).
\eeq
While the latter boundary conditions are for the most part non-physical, they have been often studied as an idealized version of the system. 

\paragraph{Local dynamics near a barotropic Oseen vortex}
 Our main result of this paper (proved in Section \ref{s:alg}), is to determine the leading order asymptotics in the case of periodic vertical boundary conditions.  As in many previous works, we begin by splitting the evolution into its barotropic (i.e. vertically averaged) and baroclinic parts, $\bar v = Qv:=\int_0^1 v(x_h,x_3) dx_3$, and $\tl v = (1-Q)\tl v$.  The purely barotropic evolution is obviously an invariant manifold within
the phase space of the full system.  We then identify a family of explicit vortex solutions in the barotropic system.  These vortices correspond to an Oseen vortex for the two-dimensional Navier-Stokes equation in the two horizontal components of the velocity, and a coupled pair of vortices in the vertical velocity and temperature fields which oscillate with the Brunt-V\"{a}is\"{a}l\"{a} frequency.  We then show that regardless of the vortex amplitude, these solutions are stable with respect to the full Boussinesq evolution.  From a dynamical systems perspective, this shows that these solutions are in the barotropic manifold and are at least locally attractive.
We also note, that for this local analysis, we need not assume the rotation rate is large, even to obtain the stability of large vortex solutions.

As we explain below, the detailed analysis of the convergence towards these vortices requires somewhat localized initial
data.
Technically,  we adapt and extend the aforementioned techniques of \cite{gallay2005global} to prove global existence in weighted spaces and derive leading order asymptotics by decomposing solutions using the leading order eigenspaces of the linear system.   
We enforce algebraic spatial decay with the weighted spaces
\begin{align}
L^2(m)^4 = \{ v(\xi,x_3)\in L^2(\mb{D})^4 \,:\, ||v||_{L^2(m)^4}:= ||b^m v||_{L^2(\mb{D})}<\infty\}\notag\\
L^2_{2D}(m) = \{ f\in L^2(\R^2)\,:\, || f||_{L^2_{2D}(m)}:=||b^m f||_{L^2(\R^2)}<\infty\},\notag
\end{align}
where $b^m = (1+|\xi|^2)^{m/2}.$   (The reason for this precise choice of weighted space is discussed in Section \ref{s:alg}.)
We also denote by 
\begin{equation}
\varphi_0(\xi) = \fr{1}{4\pi}\re^{-|\xi|^2/4}\ ,\qquad \bar{\mfu}^0_h (\xi)=\fr{1-\re^{-|\xi|^2/4}}{2\pi\xi^2}\left(\begin{array}{c}-\xi_2 \\\xi_1\\\end{array}\right),\label{e:oss}
\end{equation}
where $\bar{\mfu}^0_h$ is the incompressible, two-dimensional velocity field with Gaussian vorticity distribution $\varphi_0$. (This is the Oseen vortex, which governs the
long-time evolution of the two-dimensional Navier-Stokes equations \cite{gallay2005global}.) Then, setting $\overline{\omega}:= \nabla\times \bar u = (\p_2 \bar u_3, -\p_1 \bar u_3, \p_2 \bar u_1 - \p_1 \bar u_2)^T$, $\omega_h = (\omega_1,\omega_2)^T$, and $\bar\Theta:= \nabla^\perp_h \bar \theta = (\p_2 \bar\theta,-\p_1\bar\theta)$ our precise result is as follows
\begin{Theorem}\label{t:3}
Fix $\mu\in (0,1/2)$. Then for all $A,B_1,B_2\in \R$, there exists a constant $r>0$ such that for all initial data $(\bar \omega_{0}, \bar \Theta_0)\in L^2_{2D}(m)^5$ and $\tl v_0\in L^2(m)^4$, with $m>3$, and
\begin{align}
&||\bar \omega_{3,0} - A \varphi_0 ||_{L^2_{2D}(m)} + \| \bar \omega_{h,0} - B_1 \nabla_h^\perp \varphi_0 \|_{L^2_{2D}(m)^2} + ||\bar \Theta_0 - B_2 \nabla_h^\perp\varphi_0 ||_{L^2_{2D}(m)} + ||\tl v_0||_{L^2(m)^4} < r,\notag\\
&A = \int_{R^2} \bar \omega_{3,0}(x_h)dx_h,\, \, B = (B_1,B_2)^T:=( \int_{\R^2} \bar u_{3,0}(x_h)dx_h,\,  \int_{\R^2} \bar \theta_0(x_h)dx_h)^T,\notag
\end{align}
there exists a global solution of \eqref{e:00d}, for which the quantities $(\bar\omega, \bar \Theta,\tl v)$ satisfy $(\bar\omega,\bar\Theta,\tl v)\in C([0,\infty),L^2_{2D}(m)^3\times L^2_{2D}(m)^2\times L^2(m)^4)$ , and
\begin{align}
||\bar\omega_3(\cdot,t) - \fr{A}{1+t} \varphi_0 \lp(\fr{x}{\sqrt{1+t}}\rp)||_{L^p(\R^2)} &\leq \fr{C_p}{(1+t)^{1+\mu-1/p}}\\
\| \bar \omega_h(\cdot,t) - \fr{1}{(1+t)^{3/2}} \lp( B_1 \cos(\Gamma t) + B_2 \sin(\Gamma t)  \rp) \nabla_h^\perp\varphi_0(\fr{\cdot}{\sqrt{1+t}})\|_{L^p(\R^2)} &\leq \fr{C_p}{(1+t)^{3/2+\mu-1/p}} \\
\| \bar \Theta(\cdot,t) - \fr{1}{(1+t)^{3/2}} \lp( -B_1 \cos(\Gamma t) + B_2 \sin(\Gamma t)  \rp) \nabla_h^\perp\varphi_0(\fr{\cdot}{\sqrt{1+t}})\|_{L^p(\R^2)} &\leq \fr{C_p}{(1+t)^{3/2+\mu-1/p}}\\
\|\tilde{v}(t)\|_{L^p(\R^2)}&\leq \fr{C_p}{(1+t)^{5/4 - 1/p}},
\label{eq:red_asymptotics}
\end{align}
 for $t\geq 0$ and any $p\in [1,2]$. 
 \end{Theorem}
\begin{Corollary}
One can use the Biot-Savart relationships in Proposition \ref{p:bsest} to readily conclude the following decay estimates for $\bar v$, $1/q =1/p - 1/2$ with $p\in [1,2]$, and $t>0,$
\begin{align}
\|\bar{u}_h(\cdot,t) - \frac{A}{\sqrt{1+t}} \bar{\mfu}^0_h(\frac{\cdot}{\sqrt{1+t}}) \|_{L^q(\R^2)}&\leq \fr{C_p}{ (1+t)^{1/2+\mu-1/q}} , \notag \\
\| \bar{u}_3 (\cdot,t) - \frac{1}{1+t} \left( B_1 \cos(\Gamma t) + B_2 \sin(\Gamma t) \right) \phi_0 (\frac{\cdot}{\sqrt{1+t}})\|_{L^q(\R^2)}&\leq \frac{1}{(1+t)^{3/2+\mu-1/q}} , \notag \\
\|\bar{\theta}(\cdot,t) - \frac{1}{1+t }\left( - B_1 \sin(\Gamma t) + B_2 \cos(\Gamma t) \right) \phi_0(\frac{\cdot}{\sqrt{1+t}})\|_{L^q(\R^2)}&\leq \fr{C_p}{(1+t)^{3/2+\mu-1/q}}.\notag
\label{eq:red_asymptotics3}
\end{align}
\end{Corollary}

\paragraph{Global existence and asymptotics}
To study more general initial data, one can then use a high rotation number limit, $|\Omega|\gg1$, in dispersive estimates and energy-type estimates to obtain global existence and asymptotics, albeit with less information on the rate of convergence.  In particular, requiring smallness of only the quasi-geostrophic part of the initial condition, that is initial data $v_0$ satisfying $\int_{\R^3}|\langle \tl v( k,0), a_g( k)\rangle_{\C^4}|^2d k\ll 1$, with $a_g(k)$ defined in \eqref{e:evecs-p}, we prove global existence of mild solutions. We also once again show that at leading order, thermal fluctuations decay to zero and the velocity field converges to the two dimensional vortex solution
$$
u^G(x,t) =(\frac{A}{\sqrt{1+t}} \bar{\mfu}^0_h(\frac{x_h}{\sqrt{1+t}}),0)^T,
$$
with $\overline{\mfu}_h^0$ defined above. Depending on the boundary conditions, we consider solutions in one of the two the Banach spaces
\begin{align}
X_\mathrm{sf}&= \{ v\in H^1_{\mathrm{loc}}(\mb{D})^4\,|\, \mathrm{div} \, \tl u = \mathrm{div} \bar u_h= 0 , \tl v\in H^1(\mb{D})^4, \, \bar \omega\in L^1(\R^2)\cap L^2(\R^2),\text{$v$ satisfies \eqref{e:sfbc}}\},\notag\\
X_\mathrm{p} &=  \{ v\in H^1_{\mathrm{loc}}(\mb{D})^4\,|\, \mathrm{div} \, \tl u = \mathrm{div} \bar u_h= 0 , \tl v\in H^1(\mb{D})^4, \, \bar \omega\in L^1(\R^2)\cap L^2(\R^2),\text{$v$ satisfies \eqref{e:pbc}}\},\notag
\end{align}
with the norms 
\begin{align}
||v ||_{X_\mathrm{sf}} &= ||\tl v||_{H^1(\mb{D})} +||\bar \omega_3||_{L^2(\R^2)}+ ||\bar\omega_3||_{L^1(\R^2)},\notag\\
||v||_{X_\mathrm{p}} &= ||\tl v||_{H^1(\mb{D})} +||\bar \omega_3||_{L^2(\R^2)} + ||\bar\omega_3||_{L^1(\R^2)} +||\bar u_3||_{L^2(\R^2)} + ||\bar \theta||_{L^2(\R^2)}.
\end{align}
where once again $\bar v = Qv:=\int_0^1 v(x_h,x_3) dx_3$, and $\tl v = (1-Q)\tl v$.  Also, we define a projection operator $S$ on $L^2$ as
\begin{align}
S f = \mc{F}^{-1}\lp[  \langle \hat f(\cdot), a_g(\cdot)\rangle_{\C^4}a_g(\cdot) \rp],
\end{align}
where $\mc{F}$ is the Fourier transform, defined on $L^2(\mD)$.
Our results for these two boundary conditions are as follows:

\begin{Theorem}\label{t:0}
For all $\Gamma\in \R\diagdown\{0\}$, and initial conditions $v_0\in X_{i}$ with either $i = \mathrm{p}$ or $i =\mathrm{sf}$, and $\|Sv_0\|_{L^2}$ sufficiently small, there exists $\Omega_0>0$ such that, for all $|\Omega|\geq \Omega_0$, the system \eqref{e:00d} with either periodic boundary conditions \eqref{e:pbc} or stress-free boundary conditions \eqref{e:sfbc} respectively, has a mild solution $v\in C^0([0,\infty),X_\mathrm{i})$ satisfying $v(\cdot,0) = v_0$.  Furthermore, there exists a $C>0$ such that $||v(t)||_{X_\mathrm{i}}\leq C$ for all $t>0$ and given $A = \int_\mb{D}(\mathrm{curl } \, u_0)_3 dx,$  this solution satisfies
$$
||v(\cdot,t) - \frac{A}{\sqrt{1+t}}  V^G(\frac{\cdot}{\sqrt{1+t}},\log(1+t))||_{X_\mathrm{i}}\rightarrow 0,\quad \text{as } t\rightarrow+\infty,
$$
where $V^G(x,t) = (\bar{\mfu}^0_h,0,0)^T$ and $\bar{\mfu}^0_h$ is defined in \eqref{e:oss} above.
\end{Theorem}

\begin{Remark} Note that due to the presence of the $L^1$-norm of $\omega_3$ in the definition of the $X_i$ norm,
and the fact that the third component of the vorticity of $\frac{\alpha}{\sqrt{1+t}}  u^G(\frac{\cdot}{\sqrt{1+t}},\log(1+t))
\sim \frac{\alpha}{(1+t) } G(\frac{\cdot}{\sqrt{1+t}})$, this term has non-zero $X_i$ norm as $t \to \infty$.  However, the remaining
terms in the asymptotics in Theorem \ref{t:3} all vanish in the $X_i$-norm -- they are ``invisible''
in this theorem.  Thus, while Theorem \ref{t:0} has the 
advantage of treating global initial data, it gives far less detailed information about the long-time behavior
of solutions than Theorem \ref{t:3}.
\end{Remark}

The proof of Theorem \ref{t:0} for each type of boundary condition follows the approach of \cite{gallay2009global} with modifications to account for the inclusion of temperature effects.
Namely, we use a barotropic/baroclinic decomposition to split $v$ into a $x_3$ independent part $\bar v$ and a $x_3$-mean zero part $\tl v$.  We shall further decompose the barotropic vector $\bar v$ studying a system of equations for $\bar \omega_3 = (\nabla \times \bar u)_3, \bar u_3,$ and $\bar \theta$. Note stress-free boundary conditions force $\bar u_3\equiv \bar \theta \equiv 0$ identically.
 
We use dispersive estimates to show $\tl v$ decays exponentially fast and then use energy methods and Gronwall's inequality to prove global existence. Diffusive estimates and dynamical systems techniques can then be used to determine the leading order asymptotics of $\bar u_h$, the barotropic part, hence determining the leading order dynamics of the system. Since the proofs of the two cases in Theorem \ref{t:0} are quite similar, we only give the proof in the case of stress-free boundary conditions.  Note we do not obtain as detailed information on the decay rates of different solution components compared to that of Theorem \ref{t:3}.

\begin{Remark}\label{r:nunu}
Qualitatively similar results should hold for the more general situation of differing viscosities $\nu' \neq \nu$. In this situation, while the expressions for the eigenvalues and eigenvectors of the linear system in Fourier space are more complicated and the collection of eigenvectors are not orthogonal, three of these vectors are still orthogonal to the Fourier vector $( k_1, k_2, k_3,0)^T$. Thus, one should be able to push through the required dispersive estimates, as done in \cite[\S 4]{charve2005convergence} for example,  and obtain the same asymptotics for $\tl v$ up to an $\mc{O}(\Omega^{-1})$-sized correction.  The changes we expect if $\nu = \nu' \ne 1$ are also slightly more complicated.  On the whole space, one can change the length scale so that the viscosity is always equal to one.  However, if we make such a rescaling in our situation, it would change the thickness of the fluid layer.  A different value of $\nu$ would probably result in a different decay rate for the baroclinic components of the solution, but we still expect on heuristic grounds that they would decay more rapidly than any inverse power of $t$.  For the barotropic components of the solution, we expect the decay rates would be the same as those in Theorem \ref{t:3}, since these parts of the solution depend only on the unbounded variables $(x_1,x_2)$.
\end{Remark}

While such results seem to have been implicitly known in the literature, (namely that for sufficiently small quasi-geostrophic part the fast-oscillatory component decays as $t\rightarrow+\infty$, or in other words inertial waves escape to infinity, and the system asymptotes to solutions of 2-D Navier-Stokes equations), 
to our knowledge it has never been rigorously stated or proven, nor has the detailed asymptotic behavior obtained in Theorem \ref{t:3} been derived, especially not for perturbations of large, barotropic vortices. Furthermore, we believe the view point taken here could be useful in future studies, especially in the study of unstably-stratified convective rotating systems where a wealth of coherent structures arise; see Section \ref{ss:unst}.

A natural next line of study extending from this work would be to investigate how much one can say about asymptotics for fully arbitrary initial conditions $v_0\in X$, where there is no restriction on the quasi-geostrophic component.  As mentioned above, Charve has shown in $\R^3$ that general initial data evolves towards solutions of the quasi-geostrophic equation.  Thus, a useful question to frame such an inquiry could be: 

\emph{``Can dynamical systems techniques allow for a more refined characterization and classification of the asymptotic dynamics of the stably-stratified system with realistic boundary conditions?"}

Finally, we believe it would of great interest to investigate whether dynamical systems techniques could be used to characterize the physical phenomenon observed in unstably-stratified rotating convection, described in Section \ref{ss:unst} above.  This seems natural as such approaches have been very successful in describing coherent structures in other contexts (see for example \cite{knobloch2015spatial}), so one expects that they could be successful for the cellular and plume like dynamics of the rotating system.  Furthermore, it would be of interest whether the scaling variables/invariant manifold approach of \cite{gallay2002invariant,gallay2005global} could be used to describe the large-scale turbulent vortices which arise and grow out of small-scale turbulent eddies.

\paragraph{Outline of rest of work}
The rest of the work is organized as follows.  In Section \ref{s:prim}, we collect some facts about \eqref{e:00d} with both periodic and stress-free boundary conditions and set-up our framework. In Section \ref{s:alg}, we prove our main result Theorem \ref{t:3}, leaving proofs of several technical propositions to Appendix \ref{a:sg}.  In Section \ref{s:disp} we collect estimates on the linear dispersive equation associated with our system. We then use this information in Section \ref{s:glob} to prove global existence in the stress-free case, while in Section \ref{s:asym} we prove the leading order asymptotics under our assumptions, completing the proof of Theorem \ref{t:0}.

\begin{Acknowledgment} 
The authors were supported in part by the National Science Foundation through grants  NSF-DMS-1603416 (RG) and NSF-DMS-1311553 (CEW).  They would also like to thank Edgar Knobloch for useful discussions while working on this project.
 \end{Acknowledgment}

\section{Preliminaries}\label{s:prim}
To recall, we will study the system
\begin{align}
\p_t u + (u\cdot\nabla )u + \Omega e_3\times u + \nabla p &= \Delta u + \Gamma \theta e_3,\label{e:0a}\\
\p_t \theta + (u\cdot \nabla)\theta &= \Delta \theta  - \Gamma u_3,\label{e:0b}\\
\mathrm{div} u = 0,\label{e:0c}\\
u\in \R^3, (x_h,x_3)\in \mb{D} &:= \R^2 \times [0,1],\notag
\end{align}

We denote $v = (u_1,u_2,u_3,\theta)^T$ to be the combined vector of velocity and thermal fluctuations so that \eqref{e:0a}-\eqref{e:0c}  takes the form
\beq\label{e:0d}
\p_t v + (u\cdot\tl\nabla)v - \mathrm{diag}(\Delta)v + J_{\Omega,\Gamma}v + \tl\nabla p = 0,\quad \mathrm{div} u = 0
\eeq 
with $\tl\nabla = (\p_1,\p_2,\p_3,0)^T,$ and
$$
J_{\Omega,\Gamma} = \left(\begin{array}{cc}\Omega J & 0 \\0 & \Gamma J\end{array}\right),\quad J=\left(\begin{array}{cc}0 & -1 \\1 & 0\end{array}\right).
$$
As it will ease computations, we also set  $J_\eta = \Gamma^{-1}J_{\Omega,\Gamma}$ with $\eta = \Omega/\Gamma$. Applying the Helmholtz projection, denoted as $\mb{P}$, onto divergence free vector fields to the velocity component of \eqref{e:0d} we obtain
\beq\label{e:1}
\p_t v + \mathrm{diag}(-\Delta) v + \Gamma\mb{P}J_{\eta}\mb{P} v + \mb{P}(v\cdot\tl\nabla)v = 0,\quad \mathrm{div} u = 0.
\eeq
For periodic boundary conditions, $\mb{P}$ is defined in Fourier space as 
\beq\label{e:helm}
\widehat{\mb{P} v}(k) = P(k) \hat v(k), \quad P(k) = \left(\begin{array}{c|c}\lp(\delta_{ij} - \fr{k_i k_j}{|k|^2} \rp)_{i,j\in\{1,2,3\}} & 0 \\\hline 0 & 1\end{array}\right),\quad k\in\R^2\times \{2\ri \pi \Z\}.
\eeq
with $\delta_{ij}$ the Kroenecker delta function and $k_h = (k_1,k_2)^T \in \R^2,\,k_3 = 2 n\pi, n\in \mb{Z}$.  For stress-free boundary conditions $\mb{P}$ takes a similar form where one must consider sine and cosine series in the vertical direction.

\subsection{Barotropic/Baroclinic decomposition}
\paragraph{Periodic case}
In order to exhibit the leading order dynamics of the system, we decompose our vector field into ``barotropic" and ``baroclinic" parts.  For periodic boundary conditions \eqref{e:pbc} we set
$$
u = \bar u + \tl u,\quad  \theta = \bar \theta + \tl \theta,\quad \bar u  = Qu := \int_0^1 u(x_h,x_3)dx_3,\quad \bar \theta  = Q\theta := \int_0^1 \theta(x_h,x_3)dx_3,
$$
with $x_3$-independent barotropic variables $\bar u,\bar \theta$ and $x_3$-mean zero baroclinic variables $\tl u,\tl \theta$ which satisfy vertical periodic boundary conditions.  
With this decomposition, the $x_3$-dependent terms have the following Fourier decomposition,
\begin{align}
\tl v(x) &= \int_{\R^2} \sum_{n\in \mb{Z}_*} \widehat{\tilde v}(k) \re^{\ri( k_h\cdot x_h+2\pi n x_3)} dk_h,\,\,\\
\tl\omega_i(x) &= \int_{\R^2} \sum_{n \in \mb{Z}_*} \widehat{\tilde \omega}(k) \re^{\ri( k_h\cdot x_h+2\pi n x_3)} dk_h.
\end{align}
where $k_h = (k_1,k_2)^T$, $x_h = (x_1,x_2)^T$, $\mb{Z}_* = \mb{Z}\diagdown\{0\},$ and $\widehat{\tilde v}$ and $\widehat{\tilde \omega}$ denote the Fourier transform of $\tl v$ and $\tl \omega$ respectively. From the incompressibility condition, the Biot-Savart law relating $\tl u$ and $\tl \omega$ is found to be
\beq\label{e:pbbs}
\tl u_n(k_h) =\fr{1}{|k_h|^2 + 4\pi^2 n^2}  A_n(k_h) \tl \omega_n(k_h):=\fr{1}{|k_h|^2 + 4\pi^2 n^2} \left(\begin{array}{ccc}0 & -2\ri\pi n & \ri k_2 \\2\ri\pi n  & 0 & -\ri k_1 \\-\ri k_2 & \ri k_1 & 0\end{array}\right) \tl \omega_n(k_h), 
\eeq 
while the skew-Hermitian term $\mb{P}J_{\eta}\mb{P}$ takes the form
\beq\label{e:pjp-p}
\widehat{\mb{P}J_{\eta}\mb{P}} = \fr{1}{|k|^2}\left(\begin{array}{cccc}  0		& 4\pi^2 n^2 \eta  & -2\pi n k_2\eta  & -2\pi n k_1  \\ 
-4\pi^2 n^2\eta  & 0   & 2\pi n k_1 \eta  &  -2\pi n k_2  \\
 2\pi n k_2  \eta & -2\pi n k_1 \eta  & 0   & |k_h|^2  \\
 2\pi n k_1    & 2\pi n k_2   & -|k_h|^2   & 0  \end{array}\right),
\eeq
where $k= (k_1,k_2,2\pi n)^T,\,\, |k|^2:= |k_h|^2 + (2\pi n)^2.$ The spectral information of this matrix is computed in \cite[\S 2]{iwabuchi2016global} and one finds it has eigenvalues $0,0,\pm \ri p_\eta(k)$ with corresponding eigenvectors 
\begin{align}
a_g(k)&=\fr{1}{|k_\eta|}\left(\begin{array}{c}  \ri k_2 \\ -\ri k_1 \\ 0   \\ 2\ri\pi n\eta\end{array}\right),\,\,
a_0(k)=\fr{1}{|k|}\left(\begin{array}{c}  \ri k_1 \\  \ri k_2 \\ 2\ri \pi n\\ 0 \end{array}\right),\,\,\\
a_+(k)&=\fr{1}{\sqrt{2}|k_h||k||k_\eta|}\left(\begin{array}{c}2n\pi( k_2  \eta |k| + \ri k_1 |k_\eta| )\\ 2n\pi(-\ri k_1  \eta |k| + k_2 |k_\eta|) \\ -\ri |k_h|^2|k_\eta| \\ |k_h|^2|k| \end{array}\right),\,\,
a_-(k) = \overline{a_+(k)}\label{e:evecs-p}.
\end{align}
where $k_{\eta} = (k_1,k_2,2\pi n\eta)$. We can then decompose \eqref{e:0d} into the following barotropic/baroclinic system,
\begin{align}
\bar u_t &= \Delta \bar u + \Gamma \bar \theta e_3 - Q [\bar u \cdot\nabla \bar u + \tl u \cdot \nabla \tl u] -Q\nabla p ,\notag\\
\bar \theta_t &= \Delta \bar \theta - \Gamma \bar u_3 - Q[\bar u \cdot \nabla \bar \theta+ \tl u \cdot \nabla \tl \theta],\notag\\
\tl v_t &= \Delta \tl v - J_{\Omega,\Gamma} \tl v- (1 - Q) \tl\nabla p - (1 - Q) [( \tl u\cdot \nabla)\tl v + (\bar u \cdot \nabla) \tl v + (\tl u \cdot \nabla) \bar v].\label{e:pbdec}
\end{align}

\paragraph{Stress-free case}
In the stress-free case, the boundary conditions imply that $\bar u_3 = \bar \theta = 0$, so
$$
v = (\bar u_1, \bar u_2, 0, 0)^T + (\tl u_1,\tl u_2, \tl u_3,\tl \theta)^T,\quad u_h = (u_1,u_2)^T,
$$ 
and
$$
\bar u_h(x_h) = Q u_h = \int_0^1 u_h(x_h,x_3) dx_3,\quad \tl u_h = (1-Q) u_h.
$$
 Hence $\bar u_h$ is $x_3$-independent, $\tl u_h$ has zero vertical mean, and $\tl u_3$ and $\tl \theta$ satisfy Dirichlet boundary conditions.  Also note that the incompressibility condition implies that  $ \mathrm{div} \tl u = 0$ and $ \mathrm{div}_h \bar u_h = 0.$  Furthermore, the corresponding vorticity $\omega = \mathrm{curl}\, u$, has the decomposition $\omega = (0,0,\bar \omega_3)^T + (\tl \omega_1,\tl \omega_2,\tl \omega_3)^T$ with $\tl\omega_3$ mean-zero in $z$, $\bar \omega_3 = \mathrm{curl} \, \bar u_h$, and $\tl\omega_i$ satisfying Dirichlet boundary conditions in $x_3$ for $i = 1,2$.

 With this decomposition, the $x_3$-dependent terms have the following Fourier decomposition,
\begin{align}
\tl u_i(x) &= \int_{\R^2} \sum_{n = 1}^\infty \widehat{\tilde u}_{i}(k) \re^{\ri k_h\cdot x_h} \cos(n\pi x_3) dk_h,\,\, i = 1,2\\
\tl u_3(x) & = \int_{\R^2} \sum_{n = 1}^\infty \widehat{\tilde u}_{3}(k) \re^{\ri k_h\cdot x_h}\sin(n\pi x_3) dk_h,\\
\tl \theta(x) & = \int_{\R^2} \sum_{n = 1}^\infty  \widehat{\tilde\theta}(k) \re^{\ri k_h\cdot x_h}\sin(n\pi x_3) dk_h,\\
\tl\omega_i(x) &= \int_{\R^2} \sum_{n = 1}^\infty \widehat{\tilde\omega}_{i}(k) \re^{\ri k_h\cdot x_h} \sin(n\pi x_3) dk_h,\quad i = 1,2\\
\tl\omega_3(x) &= \int_{\R^2} \sum_{n = 1}^\infty \widehat{\tilde\omega}_{3}(k) \re^{\ri k_h\cdot x_h} \cos(n\pi x_3)dk_h,
\end{align}
where the $\widehat{\tilde u}_{i}$ and $\widehat{\tilde\omega}_{i}$ terms can be obtained from the inverse Fourier transform.  The Biot-Savart law for these boundary conditions can be found to be 
\beq\label{e:bs-sf}
\widehat{\tilde u}(k) = \fr{1}{|k_h|^2 + 4\pi^2 n^2} \left(\begin{array}{ccc}0 & n\pi & \ri k_2 \\-n\pi & 0 & -\ri k_1 \\-\ri k_2 & \ri k_1 & 0\end{array}\right) \widehat{\tilde\omega}(k),
\eeq 
while the skew-Hermitian term $\mb{P}J_{\eta}\mb{P}$ takes the form
\beq\label{e:pjp-sf}
\widehat{\mb{P}J_{\eta}\mb{P}} = \fr{1}{|k|^2}\left(\begin{array}{cccc}  0		& (n\pi)^2 \eta  & \ri k_2 n\pi\eta  & \ri k_1 n \pi  \\ 
-(n\pi)^2\eta  & 0   & -\ri k_1n\pi \eta  &  \ri k_2 n \pi \\
 \ri k_2 n\pi \eta & -\ri k_1 n \pi \eta  & 0   & |k_h|^2  \\
 \ri k_1 n\pi   & \ri k_2 n \pi   & -|k_h|^2   & 0  \end{array}\right),
\eeq
where $ |k|^2:= |k_h|^2 + (\pi n)^2$ and $|k_\eta|^2:= |k_h|^2 + (\eta\pi n)^2$.    The matrix defined above once again has a double eigenvalue at $0$ and a pair of eigenvalues $\pm \ri p_\eta(k) = \pm \ri |k_\eta|/|k|$ with corresponding eigenvectors
 \begin{align}
a_g(k)&=\fr{1}{|k_\eta|}\left(\begin{array}{c}  \ri k_2 \\ -\ri k_1 \\ 0   \\ n\pi\eta\end{array}\right),\,\,
a_0(k)=\fr{1}{|k|}\left(\begin{array}{c}  \ri k_1 \\  \ri k_2 \\ -n\pi \\ 0 \end{array}\right),\,\,\\
a_+(k)&=\fr{1}{\sqrt{2}|k_h||k||k_\eta|}\left(\begin{array}{c}-n\pi(\ri k_2  \eta |k| + k_1 |k_\eta| )\\ n\pi(\ri k_1  \eta |k| - k_2 |k_\eta|) \\ \ri |k_h|^2|k_\eta| \\ |k_h|^2|k| \end{array}\right),\,\,
a_-(k) = \overline{a_+(k)}\label{e:evecs-sf}.
\end{align}
We can then decompose \eqref{e:1} into the following system
\begin{align}\label{e:0e}
\p_t\bar u_h &= \Delta \bar u_h - \mb{P}\lp[   \bar u_h\cdot \nabla \bar u_h +  Q\tl u\cdot \nabla \tl u_h \rp]\notag\\
\p_t \tl v &= \Delta \tl v - \Gamma\mb{P}J_\eta\mb{P} \tl v - \mb{P}(1-Q)\lp[ (\tl v\cdot\tl\nabla)\tl v\rp].
\end{align}
 Taking the two-dimensional curl of the equation for $\bar u_h$ we obtain
\begin{align}\label{e:2a}
\p_t \bar \omega_3 &= \Delta \bar \omega_3 - \bar u_h \cdot\nabla \bar \omega_3 - \bar N(v),\\
\p_t \tl v &= \Delta \tl v - \Gamma\mb{P}J_{\eta}\mb{P} \tl v - \tl N(v),\label{e:2b}\\
&\mathrm{div}\, \tl u = 0,\notag
\end{align}
where 
$$
\bar N(v) = Q\lp[ (\tl u\cdot \nabla) \tl\omega_3 - (\tl\omega \cdot \nabla )\tl u_3 \rp],\quad
\tl N(v) = \mb{P}(1 - Q)\lp[ (\tl u \cdot \tl\nabla) \tl v   + (\bar u_h\cdot \tl\nabla )\tl v +  (\tl u\cdot\tl\nabla) \bar v   \rp]
$$
and $\tl u$ and $\tl \omega$ can be related via the Biot-Savart law \eqref{e:bs-sf}. We shall use this formulation to prove Theorem \ref{t:0} in Sections \ref{s:disp} -\ref{s:asym} below.

\section{Existence and asymptotics in algebraically weighted spaces}\label{s:alg}
Our main result in this section concerns the existence and asymptotics of solutions of the Boussinesq equations \eqref{e:0d} with periodic boundaries conditions \eqref{e:pbc} for initial data that lie in algebraically weighted spaces.  The existence result follows
in a fairly standard fashion by rewriting the equation as an integral equation, coupled with estimates on the linear evolution.  We note that the coupling of the vorticity formulation with temperature fluctuations in our system requires a more subtle analysis of the linear evolution; see Appendix \ref{a:sg}.
The analysis of the asymptotics makes use of scaling variables which have been very useful in previous studies of the asymptotics of the Navier-Stokes equations.
We work with the baroclinic-barotropic decomposition \eqref{e:pbdec} described in Section \ref{s:prim}. 

\subsection{Dynamics on barotropic invariant subspace}\label{ss:btdyn}
To motivate our results, we first consider the system on the subspace $\tl v \equiv 0$, which is invariant under the evolution of the system of equations \eqref{e:pbdec}. In this subspace, the equations reduce to:
\begin{align}
(\bar u_h)_t &= \Delta \bar u_h  - \bar u_h \cdot\nabla \bar u_h   -Q\nabla p_h ,\notag\\
(\bar u_3)_t &= \Delta \bar u_3 + \Gamma \bar \theta  - \bar u \cdot \nabla \bar u_3 , \notag \\
\bar\theta_t &= \Delta \bar \theta - \Gamma \bar u_3 -  \bar u \cdot \nabla \bar \theta  , \label{e:pbdec_red}
\end{align}
where $p_h(x_1,x_2) = (p_1,p_2)(x_1,x_2).$  Note further that since $\partial_{x_3} \bar u_3 = 0$, we have 
$\nabla_h \cdot \bar u_h\equiv \partial_{x_1} \bar u_1 + \partial_{x_2} \bar u_2  = 0$.  Finally note that the equation for $\bar u_h$ is independent of the evolution of $(\bar u_3, \bar \theta)$ (and also of $\Omega$ - the
horizontal components of the barotropic velocity are unaffected by the rotation of the system).  Furthermore, this is just the 
two-dimensional Navier-Stokes equation for which the long-time asymptotics are well understood \cite{gallay2002invariant},
\cite{gallay2005global}.   In particular, all solutions with integrable initial vorticity converge to an Oseen vortex, an explicit solution of the equation with Gaussian vorticity profile.

We will study dynamics using a vorticity formulation.  Since velocity and vorticity have different spatial decay and regularity properties, we consider the vorticity $\bar \omega_3 = \p_2 \bar u_1 - \p_1 \bar u_2$ of the horizontal velocity components as well as the skew-gradients $\bar \omega_h = \nabla_h^\perp \bar u_3$ and $\bar \Theta = \nabla_h^\perp \bar\theta$, with $\nabla_h^\perp = (\p_2,-\p_1)^T$, of the vertical velocity and temperature respectively.  With these quantities, one readily obtains from \eqref{e:pbdec_red}, the following system
\begin{align}
(\bar \omega_h)_t&= \Delta \bar \omega_h + \Gamma\bar\Theta- \bar u_h \cdot\nabla \bar \omega_h + \bar\omega_h\cdot\nabla \bar u_h\notag\\
(\bar \omega_3)_t&= \Delta \bar \omega_3 - \bar u_h\cdot\nabla\bar\omega_3\notag\\
\bar\Theta_t&= \Delta \bar\Theta -\Gamma\bar\omega_h- \bar u_h \cdot\nabla \bar\Theta + \bar\Theta\cdot\nabla \bar u_h\notag\\
&\nabla_h\cdot\bar\Theta = \nabla_h\cdot\bar w_h = 0. \label{eq:e:pbdec_vort}
\end{align}
Here, $\bar u_h, \bar u_3,$ and $\bar \theta$ are obtained from $\bar \omega_3$, $\bar\omega_h$, and $\bar \Theta$ respectively via the Biot Savart laws
 \begin{align}
 \overline{u}_h(x_h) &= \fr{1}{2\pi}\int_{\R^2}\fr{(x_h - y_h)^\perp}{|x_h-y_h|^2 }\overline{\omega}_3(y_h) dy_h\label{e:BSh}\\
 \overline{u}_3(x_h) &= -\fr{1}{2\pi}\int_{\R^2}\fr{(x_h - y_h)}{|x_h - y_h|^2}\wedge \bar \omega_h(y_h) dy_h,\label{e:BS3}\\
 \overline{\theta}(x_h)&= -\fr{1}{2\pi}\int_{\R^2}\fr{(x_h - y_h)}{|x_h - y_h|^2}\wedge \overline{\Theta}(y_h) dy_h,\label{e:BS4}
 \end{align}
with $x\wedge y = x_1 y_2 - x_2 y_1,$ and $x^\perp = (x_2,-x_1)^T.$ Note also that the coupling terms in the equations for $\bar w_h$ and $\bar \Theta$ create an oscillation between these two components of the solution.  The next natural step would be to remove this oscillation by going to a rotating coordinate system, but while this simplifies the linear dynamics, it complicates the baroclinic nonlinearity in the resulting system.  Hence, we will use a rotating frame to study the linear system, and then use the stationary frame in the nonlinear system.   With this in mind we introduce:

\paragraph{Rotating coordinates and vorticity formulation}
Note that if we ignore all terms with derivatives on the right hand side of the equations for $\bar w_h$ and $\bar \Theta$ in \eqref{eq:e:pbdec_vort}, we obtain a linear oscillator with frequency $\Gamma$.  With this in mind, we introduce a rotating system of coordinates:
\begin{equation}\label{eq:rotation}
\left(\begin{array}{c}\underline{\omega}_h \\ \underline{\Theta}\end{array}\right) := \re^{\Gamma tJ_2} \left(\begin{array}{c}\overline{\omega}_h \\ \bar\Theta\end{array}\right) = \left(\begin{array}{cc}\cos(\Gamma t) I_2 & \sin(\Gamma t)I_2\\-\sin(\Gamma t)I_2 & \cos(\Gamma t) I_2\end{array}\right)\left(\begin{array}{c}\overline{\omega}_h \\ \bar\Theta\end{array}\right)
,\qquad J_2 = \left(\begin{array}{cc}0 & I _2\\-I_2 & 0\end{array}\right),
\end{equation}
with $I_2$ the two-dimensional identity matrix, so that \eqref{eq:e:pbdec_vort} now takes the form
\begin{align}\label{eq:subsystem}
(\ul{\omega}_h)_t&= \Delta \ul{\omega}_h - \ul{u}_h \cdot\nabla \ul{\omega}_h + \ul{\omega}_h\cdot\nabla \ul{u}_h\notag\\
(\bar{ \omega}_3)_t&= \Delta \bar{ \omega}_3 - \bar{u}_h\cdot\nabla\bar{\omega}_3\notag\\
\ul{\Theta}_t&= \Delta \ul{\Theta} - \bar{u}_h\cdot\nabla\ul{\Theta} + \ul{\Theta}\cdot\nabla \bar{u}_h\notag\\
&\nabla_h\cdot\bar\Theta = \nabla_h\cdot\bar w_h = 0.
\end{align}

To analyze the asymptotic behavior of the solutions of these equations it is convenient to introduce ``scaling variables,'' i.e. to rescale both the dependent and independent variables in the equations as:
\begin{align}
& \bar\omega (x_h,t) = \fr{1}{1+t} \bar w (\fr{x_h}{\sqrt{1+t}},\log(1+t)),\notag\\
& \bar\Theta (x_h,t) = \fr{1}{1+t} \bar \Theta (\fr{x_h}{\sqrt{1+t}},\log(1+t)),\notag\\
&\bar{\theta}(x_h,x_3,t)  =\fr{1}{\sqrt{1+t}}\bar{\phi}(\fr{x_h}{\sqrt{1+t}},\log(1+t)),\notag\\
& \bar u_h(x_h,t)= \fr{1}{\sqrt{1+t}} \overline{\mfu}_h(\fr{x_h}{\sqrt{1+t}},\log(1+t)) ,\notag\\
& \bar u_3(x_h,t)= \fr{1}{\sqrt{1+t}} \overline{\mfu}_3(\fr{x_h}{\sqrt{1+t}},\log(1+t))\notag\\
&  \xi = \fr{x_h}{\sqrt{1+t}},\quad \tau = \log(1+\tau).
\label{eq:2dscaling}
\end{align}

In terms of these variables \eqref{eq:e:pbdec_vort} takes the form 
\begin{align}
(\ul{w}_h)_\tau&= \mc{L}\ul{w}_h - \overline{\mfu}_h\cdot\nabla \ul{w}_h + \ul{w}_h\cdot\nabla \overline{\mfu}_h\notag\\
(\bar w_3)_\tau& = \mc{L}\bar{w}_3 - \overline{\mfu}_h\cdot\nabla\ul{w}_h\notag\\
\ul{\Phi}_\tau& = \mc{L}\ul{\Phi}  - \overline{\mfu}_h\cdot\nabla \ul{\Phi} + \ul{\Phi}\cdot\nabla \overline{\mfu}_h,\notag\\
&  \nabla_h\cdot \underline{\Phi} = \nabla_h\cdot \underline{w}_h = 0,
\label{eq:subsystem_scaled}
\end{align}
where 
$$
\mc{L}:=\Delta_\xi + \fr{1}{2}\xi\cdot\nabla_h +1,$$
and 
${\overline{\mathfrak{u}}_h}$ is the velocity field associated to the vorticity $\bar w_3$ via the two-dimensional Biot-Savart law \eqref{e:BSh} (which somewhat remarkably is unchanged by the introduction of scaling variables).  As noted in Section \ref{s:results}, we study these equations in the weighted Hilbert spaces $L^2(m)$ and $L^2_{2D}(m)$.  We also need the associated weighted Sobolev spaces
$$
H^1(m) = \{ v\in L^2(m)^4\,:\, \nabla v \in L^2(m)^4\},\quad H^1_{2D}(m) = \{ v\in L^2_{2D}(m)\,:\, \nabla v \in L^2_{2D}(m)\}
$$
with norm 
$$
\| v \|_{H^1(m)}^2:= \int_{\mb{D}}(1+|x_h|^2)^{m}\lp( |f(x)|^2 + |\nabla f(x)|^2    \rp) dx.
$$
We also define the weighted $L^q$ spaces in a similar way
$$
L^q(m):=\{v\in L^q(\mb{D})^4 : \|v\|_{L^q(m)}:= \|b^m v\|_{L^q(\mb{D})} <\infty\}, \quad
W^{1,q}(m):= \{v\in L^q(m)\,:\,\nabla v\in L^q(m)^4\}.
$$
The operator $\mc{L}$ then has the following spectral properties when posed on $L^2_{2D}(m)$. For a more detailed account of these facts see \cite[\S A]{gallay2002invariant} or \cite[\S 4.2]{roussier2003stabilite}.  
\paragraph{Spectral information}
\begin{enumerate}
\item For any $m\geq0$, $\mc{L}$ has eigenvalues $\lambda_k = -k/2$ for all $k\in \mb{N}$ with eigenfunctions
$$
\varphi_\alpha(\xi) = (\p_\xi^\alpha\varphi_0)(\xi),\quad \varphi_0(\xi) = \fr{1}{4\pi}\re^{-|\xi|^2/4},\quad \alpha\in \mb{N}^2, \quad |\alpha| = k.
$$
\item Each eigenvalue has a spectral projection $P_n:L^2_{2D}(m)\rightarrow L^2_{2D}(m)$
$$
(P_n f)(\xi) = \sum_{|\alpha|\leq n} \lp(\int_{\R^2} H_\alpha(\xi')f(\xi')d\xi'\rp)^{1/2}\varphi_\alpha(\xi),
$$ 
where $H_\alpha(\xi) = \fr{2^{|\alpha|}}{\alpha!} \re^{|\xi|^2/4}\p_\xi^\alpha(\re^{-|\xi|^2/4})$ are the Hermite polynomials.
\item $\sigma(\mc{L}) =  \{ \lambda\in \mb{C} \,:\, \mathrm{Re}\,(\lambda)\leq \fr{1-m}{2}\}\cup \{ -n/2\,:\, n=\mb{N} \}$.
\end{enumerate}
Hence $P_0$ is the projection onto the 0-eigenspace and takes the form $P_0 f = (\int_{\R^2} f(\xi)d\xi )\,\varphi_0$.  We also denote the complementary projections 
\begin{align}
Q_n:&L^2(m)\rightarrow L^2_{2D}(m),\quad n\geq0\notag\\
Q_n :&f\mapsto (I - P_n)Qf,\notag
\end{align}
as well as for completeness $Q_{-1} = Q$. We also denote $$L^2_{2D,n}(m):= \mathrm{Rg}_{L^2_{2D}(m)}\{ Q_n\}.$$

Next, take the weight $m>3$ in the Hilbert spaces $L^2(m)$ (and $L^2_{2D}(m)$) for which the spectrum of $\mc L$ consists of a simple eigenvalue $0$, with eigenfunction the Gaussian $\varphi_0$ and all of the rest of the spectrum in the complex half-plane with real part less than or equal to $-1/2$.  Thus, at least for small initial data, solution components, $\bar \omega_3(\tau)$, of \eqref{eq:subsystem_scaled} should converge toward some multiple of the Gaussian $\varphi_0$ with a rate $\sim e^{-\tau/2}$.  We next turn to the equations of $\ul{w}_h $ and $\ul{\Phi}$ to see 
how $\overline{\mathfrak{u}}_h$ affects their evolution.  Looking at the first equation in \eqref{eq:subsystem_scaled},
and replacing $\overline{\mathfrak{u}}_h$
with its asymptotic limit $A \overline{\mathfrak{u}}_h^0$, where $\overline{\mathfrak{u}}_h^0= \fr{1-\re^{-|\xi|^2/4}}{2\pi|\xi|^2}\left(-\xi_2 ,\xi_1 \right)^T $ is the velocity profile obtained from the gaussian vorticity $\varphi_0$, we obtain the equation
\begin{equation}
(\ul{w}_h)_{\tau} = \mc{L} \ul{w}_h - A \overline{\mathfrak{u}}_h^0 \cdot \nabla \ul{w}_h + A\ul{w}_h\cdot \nabla\overline{\mfu}_h^0.
\end{equation}

While we are not able to explicitly compute the entire spectrum of the operator on the right hand side of this equation, a very similar operator arises in the study of the linearization of the two-dimensional Navier-Stokes equation about an Oseen vortex (see [\cite{gallay2005global}, Section 4]) and using the insights gained there, one can 
analyze the leading eigenvalue of this operator. One has
\begin{Proposition} For any $A \in \R$,  the spectrum of $ \ul{\Phi}\mapsto\mc{L} \ul{\Phi} - A \overline{\mathfrak{u}}_h^0 \cdot \nabla \ul{\Phi} +  A \ul{\Phi}\cdot\nabla_h \overline{\mfu}_h^0\ $ on the weighted Hilbert space $L^2_{2D}(m)\cap\{\nabla_h\cdot\ul{\Phi} = 0\}$ consists
of a simple eigenvalue $-1/2$, with eigenfunction $\nabla_h^\perp\varphi_0$, and the half-plane $\{ \lambda \in \C ~|~
~ \Re{\lambda} \le -1 \}$.
\end{Proposition}
\begin{proof} 
 The fact that $\nabla_h^\perp\varphi_0$ is an eigenfunction with eigenvalue $-1/2$ follows by direct
computation, while the fact that the remainder of the spectrum lies in the half plane with real part less than or equal to $-1$ is proven in Lemma \ref{l:aLh} 
\end{proof}
Thus, the leading order asymptotics of the $\ul{w}_h$ and $\ul{\Phi}$ components in  \eqref{eq:subsystem_scaled} are also given by derivatives of a Gaussian 
$$
\ul{\Phi}(\xi,\tau) \approx \re^{-\tau/2}B_1\nabla_h^\perp\varphi_0(\xi ) + {\mc{O}}(e^{-\tau}),\quad
\ul{w}_h(\xi,\tau) \approx \re^{-\tau/2}B_2\nabla_h^\perp\varphi_0(\xi) + {\mc{O}}(e^{-\tau}),
$$ 
where $B_1 = \int_{\R^2} \bar u_3(x_h)dx_h$ and $B_2 = \int_{\R^2} \bar \theta(x_h)dx_h.$

If we now revert to our original, unscaled variables, and reexpress the barotropic motion in terms of the velocity, rather than
the vorticity, we find that solutions of \eqref{e:pbdec_red}, for small initial data, should behave as
\begin{align}
\bar{u}_h(x_h,t) & = \frac{A}{\sqrt{1+t}} {\bar{\mathfrak{u}}^0_h} (\frac{x_h}{\sqrt{1+t}}) + \mc{O}(\frac{1}{t}) , \notag \\
\bar{u}_3 (x_h) &= \frac{1}{1+t} \left( B_1 \cos(\Gamma t) + B_2 \sin(\Gamma t) \right) \varphi_0(\frac{x_h}{\sqrt{1+t}})
+  \mc{O}(\frac{1}{t^{3/2}}) , \notag \\
\bar{\theta}(x_h) &= \frac{1}{1+t} \left( - B_1 \sin(\Gamma t) + B_2 \cos(\Gamma t) \right) \varphi_0(\frac{x_h}{\sqrt{1+t}})
+  \mc{O}(\frac{1}{t^{3/2}}) . \label{eq:red_asymptotics}
\end{align}
We note a few somewhat surprising facts about these asymptotics:
\begin{Remark}The vertical component of the barotropic velocity, $\bar{u}_3$, and the barotropic termperature variations $\bar{\theta}$ 
are much more strongly localized than the horizontal components of the barotropic velocity ${\bar{u}}_h$.  From the Biot-Savart
law, one knows that $| \bar{\mathfrak{u}}^0_h(\xi)|  \sim |\xi|^{-1}$, as $|\xi| \to \infty$ whereas $\bar{u}_3$ and $\bar{\theta}$ decay as Gaussians.
\end{Remark}
\begin{Remark} Note that from a dynamical systems point of view, we will show that, at least locally, any of the vortex solutions defined in the barotropic subspace
$\{\tl v = 0\}$ are at least locally attractive for the full equations,  much as if we had a  center-manifold in a finite dimensional dynamical systems.
\end{Remark}

\subsection{Full dynamics}\label{ss:fulldyn}
We now show that these same asymptotics hold for small solutions of the full equations \eqref{e:pbdec}. We begin by taking the curl of $\bar u$ and $\tl u$ to eliminate the pressure terms.  As mentioned previously, because $\bar \theta$ has different regularity and spatial decay properties than $\bar \omega$, we consider the evolution of its horizontal skew-gradient, $\bar \Theta = \nabla^\perp_h \bar \theta = (\p_2 \bar\theta,-\p_1\bar\theta)^T$. Thus we study a system in terms of the variables $\bar \omega,\bar\Theta, \tl \omega,$ and $\tl \theta$, obtaining
\begin{align}
(\bar \omega)_t  &= \Delta \bar \omega + \left(\begin{array}{c} \Gamma \bar \Theta \\0\end{array}\right)-\bar N_1(v),\label{e:vdec1}\\
\bar \Theta_t &= \Delta \bar \Theta - \Gamma \bar w_h - \bar N_2(v)\label{e:vdec2}\\
\tl\omega_t &= \Delta\tl\omega - \Omega \p_3 \tl u + \Gamma \nabla_h^\perp \tl \theta  - \tl N_1(v)\label{e:vdec3}\\
\tl\theta_t &= \Delta \tl\theta - \Gamma \tl u_3 -\tl N_2(v)\label{e:vdec4},\\
&\mathrm{div}\, \tl \omega = 0,\quad \mathrm{div}_h\, \bar u_h = 0 \label{e:vdec5a}
\end{align}
where $\tl \omega = \nabla \times \tl u$, $\bar \omega_3 = \nabla \times \bar u_h$,$ \nabla_h^\perp = (\p_2,-\p_1)^T$, $\mathrm{div}_h u_h = \p_1 u_1 + \p_2 u_2$, and the nonlinearities are defined as  
\begin{align}
\bar N_1(v) &= (\bar u_h\cdot \nabla )\bar \omega - (\bar \omega\cdot \nabla) \bar u +Q\lp[( \tl u \cdot \nabla)\tl\omega - (\tl\omega \cdot \nabla) \tl u \rp],\notag\\
\bar N_2(v) &= (\bar u_h\cdot \nabla_h )\bar\Theta  - (\bar \Theta\cdot\nabla_h )\bar u_h + Q\lp[(\tl u\cdot \nabla)\nabla_h^\perp\tilde\theta -(\nabla_h^\perp\tilde\theta\cdot\nabla_h)\tl u_h + \p_3\tilde\theta \cdot (\nabla_h^\perp\tl u_3) - \p_3 \tl u_3\cdot (\nabla_h^\perp \tl\theta) \rp]\notag\\
\tl N_1(v)& = (1 - Q) \lp[ (\bar u \cdot \nabla)\tl \omega -( \tl \omega \cdot \nabla )\bar u + (\tl u \cdot \nabla) \bar \omega - (\bar \omega \cdot\nabla )\tl u + (\tl u \cdot \nabla) \tl \omega - (\tl\omega \cdot \nabla) \tl u\rp],\notag\\
\tl N_2(v)&=  (1 - Q) \lp[  (\tl u \cdot \nabla) \tl \theta + (\tl u \cdot \nabla )\bar\theta + (\bar u \cdot \nabla) \tl \theta  \rp].\notag
\end{align}
Note here that $\bar \omega = (\p_2 \bar u_3, -\p_1 \bar u_3, \p_2\bar u_1 - \p_1 \bar u_2)^T$ and the third component of $(\bar \omega\cdot \nabla) \bar u$ is zero due to the imcompressibility condition \eqref{e:vdec5a}. 

Our approach for this nonlinear system is the following.  We first introduce self-similar variables, and then use the properties of the approximate solution to derive a solution for the residual.  We then derive an equivalent mild/integral formulation for this residual and use estimates on the linear evolution to show nonlinear existence and asymptotics via a fixed-point argument.

\paragraph{Scaling variables}
We now convert the system \eqref{e:vdec1}-\eqref{e:vdec4} into scaling variables. In addition to
the definition of the variables $\bar{w},\overline{\mfu},\bar\Phi$ and $\bar{\phi}$ introduced in \eqref{eq:2dscaling},
we also make the change of coordinates 
\begin{align}
&\tl\omega(x_h,x_3,t) = \fr{1}{1+t}\tl w(\fr{x_h}{\sqrt{1+t}},x_3,\log(1+t)),\notag\\
&\tl u(x_h,x_3,t) = \fr{1}{\sqrt{1+t}}\tilde{\mathfrak{u}}(\fr{x_h}{\sqrt{1+t}},x_3,\log(1+t)),\notag\\
&\tl \bth(x_h,x_3,t) = \fr{1}{\sqrt{1+t}}\tl\phi(\fr{x_h}{\sqrt{1+t}},x_3,\log(1+t)),\notag
\end{align}
where as before $\xi = \fr{x_h}{\sqrt{1+t}}, \tau = \log(1+\tau)$.  Note also that $\nabla_h^\perp \bar\phi  = \overline{\Phi}.$ We then obtain from \eqref{e:vdec1}-\eqref{e:vdec4} the following system of equations
\begin{align}
\bar w_{\tau} &= \mc{L}\bar w + \re^{\tau} \left(\begin{array}{c}  \Gamma\overline{\Phi} \\0\end{array}\right) - \bar N_1(\bar w,\overline{\Phi}, \tl W,\tau),\label{e:vsc1}\\
\overline{\Phi}_\tau &=\mc{L}\overline{\Phi}- \re^{\tau}\Gamma \bar{w_h}-  \bar N_2(\bar w,\overline{\Phi}, \tl W,\tau),\label{e:vsc2}\\
\tl w_\tau &= (\mc{L} + \re^\tau \p_3^2)\tl w - \Omega \re^{3\tau/2}\p_3\tl u + \Gamma \re^{\tau}\left(\begin{array}{c} \nabla_h^\perp\tl \phi\\0\end{array}\right) -\tl N_1(\bar w,\overline{\phi}, \tl W,\tau),\label{e:vsc3}\\
\tl \phi_\tau &= (\mc{L} + \re^{\tau}\p_3^2 - 1/2)\tl \phi - \Gamma\re^{\tau} \tilde{\mathfrak{u}}_3 - \tl N_2(\bar w,\overline{\phi}, \tl W,\tau)\label{e:vsc4}
\end{align}
with the nonlinear terms
\begin{align}
\bar N_1(\bar w,\overline{\phi}, \tl W,\tau)&=(\overline{\mathfrak{u}}_h\cdot \nabla) \, \bar w -  \bar w\cdot \nabla \overline{\mfu} +Q\lp[ (\tilde{\mathfrak{u}} \cdot \nabla_\tau)\tl w - (\tl w \cdot \nabla_\tau )\tilde{\mathfrak{u}} \rp],\notag\\
\bar N_2(\bar w,\overline{\phi}, \tl W,\tau)&=(\overline{\mathfrak{u}}_h \cdot \nabla_h) \overline{\Phi} - (\overline{\Phi} \cdot \nabla_h)\overline{\mathfrak{u}}_h +Q\lp[ (\tilde{\mathfrak{u}}\cdot \nabla_\tau)\nabla_h^\perp\tilde\phi -(\nabla_h^\perp\tilde\phi\cdot\nabla_h)\tilde{\mathfrak{u}}_h + \re^{\tau/2}\p_3\tilde\phi\cdot (\nabla_h^\perp\tilde{ \mathfrak{u}}_3) - \re^{\tau/2}\p_3 \tl\phi\cdot (\nabla_h^\perp \tl\phi)  \rp],\notag\\
\tl N_1(\bar w,\overline{\phi}, \tl W,\tau)&=  (1-Q) \lp[  (\overline{\mfu} \cdot \nabla_\tau )\tl w - (\tl w\cdot \nabla_h )\overline{ \mfu} + (\tilde{\mathfrak{u}} \cdot \nabla_h )\bar w - (\bar w \cdot\nabla_\tau )\tilde{\mathfrak{u}} + (\tilde{\mathfrak{u}}\cdot \nabla_\tau )\tl w  - (\tl w \cdot \nabla_\tau )\tilde{\mathfrak{u}}     \rp]\notag\\
\tl N_2(\bar w,\overline{\phi}, \tl W,\tau)&=(1-Q)
\lp[ (\tilde{\mathfrak{u}} \cdot \nabla_\tau )\tilde{\phi} + (\tilde{\mathfrak{u}} \cdot \nabla_h) \overline{\phi} + (\overline{\mfu} \cdot \nabla_\tau )\tilde{\phi}     \rp]\notag
\end{align}
with $\tl W = (\tl w, \tl \phi)^T,$ $\nabla_h^\perp = (\p_2,-\p_1)^T$, $\nabla_\tau = (\p_1,\p_2,\re^{\tau/2}\p_3)^T.$
Note also that the incompressiblity conditions now take the form
$$
\nabla_\tau\cdot \tl w =\nabla_h\cdot \bar{\mathfrak{u}}_h= \nabla_h\cdot \bar w_h = \nabla_h\cdot\overline{\Phi}=0,
$$
with $\bar w$ and $\overline{\mfu}$, and $\overline{\Phi}$ and $\overline{\phi}$ are still related via a Biot-Savart laws, \eqref{e:BSh} - \eqref{e:BS4}.  Also, $\tl w$ and $\tilde{\mfu}$ are related by a scaled version of the Biot Savart law \eqref{e:pbbs} taking into account the $\re^{\tau/2}$ factor accompanying $x_3$-derivatives.
Furthermore, we shall need the following estimates which are derived in \cite[\S 4.2]{roussier2003stabilite} (see also \cite{gallay2002invariant}):
\begin{Proposition}(Biot-Savart)\label{p:bsest}
If $p\in (1,2)$ and $\bar w\in L^p(\R^2)$ then $\overline{\mfu}\in L^{2p/(2-p)}(\R^2)$ and 
\beq
||\overline{\mfu}||_{L^{2p/(2-p)}(\R^2)}\leq C||\bar w||_{L^p(\R^2)}.\label{e:bsbt}
\eeq
If $\tl w\in L^2(\mb{D})$ then $\tilde{\mathfrak{u}}\in L^q(\mb{D})$ for all $q\in [2,6]$ and there exists a $C>0$ such that
\beq
||\tilde{\mathfrak{u}}||_{L^q(\mb{D})}\leq C  \re^{-\tau/2}||\tl w||_{L^2(\mb{D})}.\label{e:bsbc}
\eeq
Furthermore, if $\tl w\in L^2(m)^3$ then $\tilde{\mfu}\in H^1(m)$ and there exists a $C>0$ such that 
\beq
||\tilde{\mathfrak{u}}||_{H^1(m)}\leq C ||\tl w||_{m}.\label{e:bsbc2}
\eeq
\end{Proposition}
\begin{proof}
The first estimate follows from \cite[Eqn. (16)]{gallay2005global} and the proof of the second
and third can be found in \cite[\S4.1.1, 4.1.4]{roussier2003stabilite}.
\end{proof}
Note, that if one desires estimates on $\overline{\mfu}_h$ or $\overline{\mfu}_3$ then one simply sets $\overline{\mfu} = (\overline{\mfu}_h,0)^T$ or $\overline{\mfu} = (0,0,\overline{\mfu}_3)^T$ respectively. We consider the existence, and asymptotic behavior of solutions of this equation in the space
\begin{equation}
Z = \{ (\bar{w},\bar \Phi ,\tl w, \tl \phi) \in L^2_{2D}(m)^3\times L^2_{2D}(m)^2\times L^2(m)^3\times H^1(m)\}\cap \{ \nabla_h\cdot\overline{w}_h = 0,\, \nabla_h\cdot \overline{\Phi} = 0,\,  \nabla_\tau\cdot\tilde w = 0\}\ , 
\end{equation}
with the natural norm on this space. We expect, from the form of the system \eqref{e:pbdec}, that the temperature should have the same regularity as the velocity.  If the vorticity is in $L^2$, then we expect (as in \eqref{e:bsbc2}) that the velocity would be in $H^1$.  For this reason we require increased regularity in the temperature component $\tl\phi$.

\paragraph{Approximate solutions and residual equation}
We take initial data, $( \bar{w}_0,\overline \Phi_0, \tilde{w}_0 , \tilde{\theta}_0)\in Z$.
(Note that the due to the way in which the scaling variables are defined, the initial values
for $\bar{w}$ are the same as that for $\bar{\omega}$ and similarly for the remaining variables.)
Motivated by our heuristic discussion in Section \ref{ss:btdyn} above, we then define
\begin{align}
&(\overline{\mfu}_h)_{app}(\xi) =A \overline{\mfu}_h^0(\xi),\notag\\ 
&(\overline{\mathfrak{u}}_3)_{app}(\xi,\tau)  = e^{-\tau/2}\left( B_1 c(\tau)+ B_2 s(\tau)\right)  \varphi_0(\xi)\ , \notag \\
& \bar{\phi}_{app}(\xi,\tau) =e^{-\tau/2} \left( -B_1 s(\tau) + B_2 c(\tau) \right)  \varphi_0(\xi)\ ,\notag\\
& \bar{w}_{app}(\xi,\tau) = \lp(0,0, A \varphi_0(\xi)\rp)^T+ \re^{-\tau/2}\lp(B_1c(\tau) + B_2s(\tau)\rp) \, \lp(\p_2\varphi_0(\xi),-\p_1\varphi_0(\xi),0   \rp)^T\notag\\
& \overline{\Phi}_{app}(\xi,\tau) =e^{-\tau/2} \left( -B_1 s(\tau) + B_2 c(\tau) \right)  \nabla_h^\perp\varphi_0(\xi) \label{e:approxsol}
\end{align}
where $c(\tau) := \cos(\Gamma(\re^{\tau}-1)), \, s(\tau) := \sin(\Gamma(\re^{\tau}-1))$, $(\overline{\mfu}_{h})_{app}$ is the velocity profile associated with $\varphi_0$ from the Biot-Savart law \eqref{e:BSh}, and the constants are determined by the initial data
$$
A = \int_{\R^2} (\bar{\omega}_3)_0 (\xi) d\xi,\quad B_1 = \int_{\R^2} (\bar{u}_3)_0 (\xi) d\xi,\quad B_2 = \int_{\R^2} \bar{\theta}_0 (\xi) d\xi.
$$
Direct computation then readily gives that the solution $\bar W_{app}  = (\bar w_{app},\overline{\Phi}_{app})^T$ solves the barotropic subsystem \eqref{e:vsc1} - \eqref{e:vsc2} with $(\tl w,\tl \phi) \equiv 0$. One also readily finds that the subspace $L^2_{2D,1}(m)^2\times L^2_{2D,0}(m)^0\times L^2_{2D,1}(m)^2$ is invariant under the evolution of this subsystem.
We can thus make the following decomposition, which splits off the leading order approximate dynamics
\begin{align}
\bar w_h &= {\bar w}_{app,h}  + \bar w_{R,h},\quad \bar w_3  =  {\bar w}_{app,3}  + \bar w_{R,3},\quad \overline{\Phi}= \overline{\Phi}_{app} + \overline{\Phi}_R\notag\\
 \overline{\phi} &= \overline{\phi}_{app} + \overline{\phi}_R,\quad \overline{\mfu}_h = (\overline{\mfu}_h)_{app} +\overline{\mfu}_{h}^R,
\end{align}
where the residuals satisfy
  $$
 \bar w_{R,h},\overline{\Phi}_R\in \mathrm{Rg}\{ Q_1\} = L^2_{2D,1}(m)^2,\quad \bar w_{R,3}\in \mathrm{Rg}\{Q_0\} =  L^2_{2D,0}(m).
 $$  
Note $\overline{\mfu}_R$ and $\overline{\phi}_R$ can also be obtained from $\bar w_{R}$ and $\bar \Phi_R$ respectively using the two-dimensional Biot-Savart laws \eqref{e:BSh} and \eqref{e:BS3} above. We thus consider the residual on the space $Y \times L^2(m)^3\times H^1(m)$ with
 \begin{align}
 Y &= Y_{h,1}\times Y_{3,0}\times Y_{p,1},\notag\\
 Y_{h,n} &= \{\bar w_h\in L^2_{2D,n}(m)^2\, |\, \mathrm{div}_h \bar w_h = 0\},\,\, Y_{3,n} = L^2_{2D,n}(m),\, Y_{p,n} = \{\bar \Phi\in L^2_{2D,n}(m)^2\, |\, \mathrm{div}_h \bar\Phi = 0\},\notag
 \end{align}
 and the norms
 $$
 \| (\bar w,\bar \Phi)\|_{m} := \| \bar w \|_m + \|\overline\Phi\|_{m},\quad  \| (\tl w,\tl\phi)\|_{*,m} := \|\tl w\|_m + \|\tl \phi\|_{H^1(m)}.
 $$
 In order to simplify the notation, we also define $\bar W_R = (\bar w_R, \overline{\Phi}_R)^T$ and $\tilde{W} = (\tilde{w},\tilde{\theta})^T$.  Inserting this decomposition into the full system \eqref{e:vsc1} - \eqref{e:vsc4}, and using the aformentioned facts of the approximate solution, we find that the residuals $\bar W_R$ and $\bar W$ satisfy the following system
\begin{align}
(\bar W_R)_\tau& = (\bar L(\tau)+  \bar \Lambda) \bar W_R + \overline{\mathbf{N}}(\bar W_R,\tl W),\label{e:vsc5}\\
\tilde{W}_\tau& = (\tilde L(\tau) + \tilde \Lambda ) \tl W + \tilde{\mathbf{ N}}(\bar W_R,\tl W),\label{e:vsc6}
\end{align}
with the linear operators
\begin{align}
\bar L(\tau)\bar W_R&:=\mc{L}\left(\begin{array}{c} \bar w_R \\  \overline{\Phi}_R\end{array}\right) + \re^{\tau}\Gamma\left(\begin{array}{c}\overline{\Phi}_R \\0\\- \overline{w}_{R,h}\end{array}\right),\label{e:linv1}\\
\tl L(\tau) \tl W&:= (\mc{L} + \re^\tau \p_{3}^2) \left(\begin{array}{c}\tl w \\\tl\phi\end{array}\right)  - \left(\begin{array}{c}0 \\\tl\phi/2\end{array}\right)
+ \Gamma \re^{\tau}\left(\begin{array}{c}\nabla^\perp_h \tl\phi\\0 \\ -\tilde{\mathfrak{u}}_3  \end{array}\right) - \left(\begin{array}{c}\Omega \re^{3\tau/2}\p_3\tilde{\mathfrak{u}}  \\ 0\end{array}\right),\label{e:linv2}\\
\bar\Lambda \bar W_R&:= \bar N(\bar W_{app} + \bar W_R,\tl W) - \lp(\bar N(0,\tl W) + \bar N(\bar W_{app},0) + \bar N(\bar W_R,0)\rp),\notag\\
&= \overline{\mfu}_{app}\cdot \nabla \bar W_R + \overline{\mfu}_R\cdot \nabla \bar W_{app}
 - \left(\begin{array}{c} (\bar w_{app}\cdot\nabla)\overline{\mfu}_{R,h} \\ 0\\ (\overline{\Phi}_{app}\cdot\nabla) \overline{\mfu}_{R,h} \end{array}\right)  
 - \left(\begin{array}{c} (\bar w_{R}\cdot\nabla)\overline{\mfu}_{app,h} \\0\\ (\overline{\Phi}_{R}\cdot\nabla) \overline{\mfu}_{R,h} \end{array}\right)\label{e:pertlin1}\\
\tl \Lambda \tl W&:=  \tl N(\bar W_{app},\tl W) - \tl N(0,\tl W),\notag\\
&= \overline{\mfu}_{app}\cdot \nabla_\tau \tl W - \tl w \cdot\nabla \left(\begin{array}{c} \overline{\mfu}_{app}  \\0 \end{array}\right) + \tilde{\mfu}\cdot \nabla  \left(\begin{array}{c} \overline{w}_{app}  \\ \bar\phi_{app} \end{array}\right) - \bar w_{app}\cdot \nabla_\tau \left(\begin{array}{c} \tilde{\mfu}  \\0 \end{array}\right),\label{e:pertlin2}
\end{align}
and with the nonlinear components $\overline{\mathbf{N}} = (\overline{\mathbf{N}}_h,\overline{\mathbf{N}}_3,\overline{\mathbf{N}}_p)^T, \,\tilde{\mathbf{N}}  = (\tl{\mathbf{N}}_1,\tl{\mathbf{N}}_2)^T$ defined as 
\begin{align}
\overline{\mathbf{N}}(\bar W_R,\tl W)&:= \bar Q\lp[ \bar N(\bar W_{app} + \bar W_R,\tl W) - \bar N(\bar W_{app},0) -  \bar\Lambda \bar W_R  \rp]\notag\\
\overline{\mathbf{N}}_h(\bar W_R, \tl W)&:= Q_1\lp[ \bar N_h(\bar W_{app} + \bar W_R,\tl W) - \bar N_h(\bar W_{app},0) -  (\bar\Lambda \bar w_R)_{h} \rp]\notag\\
&= (\overline{\mathfrak{u}}_{R,h}\cdot \nabla) \, \bar w_{h,R} -  (\bar w_R\cdot \nabla) \overline{\mfu}_{R,h} 
 +Q\lp[ (\tilde{\mathfrak{u}} \cdot \nabla_\tau)\tl w - (\tl w \cdot \nabla_\tau )\tilde{\mathfrak{u}} \rp],\notag\\
\overline{\mathbf{N}}_3(\bar W_R, \tl W)&:= Q_0\lp[ \bar N_3(\bar W_{app} + \bar W_R,\tl W) - \bar N_3(\bar W_{app},0) -  (\bar\Lambda \bar w_R)_{3} \rp]\notag\\
&= (\overline{\mathfrak{u}}_{R,h}\cdot \nabla) \, \bar w_{3,R}
 +Q\lp[ (\tilde{\mathfrak{u}} \cdot \nabla_\tau)\tl w_3 - (\tl w \cdot \nabla_\tau )\tilde{\mathfrak{u}}_3 \rp],\notag\\
 \overline{\mathbf{N}}_p(\bar W_R, \tl W)&:=  Q_1\lp[ \bar N_p(\bar W_{app} + \bar W_R,\tl W) - \bar N_p(\bar W_{app},0) -  (\bar\Lambda \bar \Phi_R)_{p}  \rp]\notag\\
 &=(\overline{\mathfrak{u}}_{R,h} \cdot \nabla_h) \overline{\Phi}_R - (\overline{\Phi}_R \cdot \nabla_h)\overline{\mathfrak{u}}_{h,R} \notag\\
 &\qquad\qquad+Q\lp[ (\tilde{\mathfrak{u}}\cdot \nabla_\tau)\nabla_h^\perp\tilde\phi -(\nabla_h^\perp\tilde\phi\cdot\nabla_h)\tilde{\mathfrak{u}}_h + \re^{\tau/2}\p_3\tilde\phi\cdot (\nabla_h^\perp\tilde{ \mathfrak{u}}_3) - \re^{\tau/2}\p_3 \tilde{\mfu}_3\cdot (\nabla_h^\perp \tl\phi)  \rp],\notag\\
\tilde{\mathbf{N}}(\bar W_R,\tl W)&:= \tl N(\bar W_{app} + \bar W_R,\tl W)  - \tl\Lambda \tl W,\notag\\
&= (1-\bar Q)\Big[ \overline{\mfu}_R\cdot \nabla_\tau \tl W  - \tl w\cdot \nabla \left(\begin{array}{c}\overline{\mfu}_R \\0\end{array}\right) + \tilde{\mfu}\cdot \nabla \left(\begin{array}{c}\overline{w}_R \\ \bar \phi_R\end{array}\right) \notag\\
&\qquad\qquad\quad- \bar w_R\cdot \nabla_\tau\left(\begin{array}{c}\tilde{\mfu} \\0\end{array}\right) + \tilde{\mfu}\cdot \nabla_\tau \tl W - \tl w\cdot\nabla_\tau \left(\begin{array}{c}\tilde{\mfu} \\0\end{array}\right)\Big].
\end{align}
Here subscripts in the barotropic nonlinearities denote the first two, third, and last two components of vectors in $Y$, and the projection $\bar Q:= Q_1\oplus Q_1\oplus Q_0\oplus Q_1\oplus Q_1$ projects $L^2_{2D}(m)^5$ onto $Y$.

Before considering the full nonlinear equation \eqref{e:vsc5}-\eqref{e:vsc6}, we list some results on the barotropic and baroclinic components of the linear evolution. Since the linear operators are time-dependent, we must characterize the linear evolution in terms of evolutionary families of operators in the sense of Pazy \cite[Ch. 5]{pazy2012semigroups}.   We leave the details of their proof to the Appendix.  We define $a(\tau) = 1 - \re^{-\tau}$, the standard basis vectors $e_i\in \R^5$, and the projections $e_i^\perp = 1 - e_i e_i^T$

\begin{Proposition}(Barotropic evolution)\label{p:SS}
Let $ m>3$, $0<\mu<1/2$, $\alpha\in \N^2$ with $|\alpha|\leq1$, and $\bar W_0\in Y.$ Then the operator $\bar L(\tau) + \bar \Lambda$ generates an evolutionary family of operators, $\overline{\mc{S}}(\tau,\sigma),$ with $\tau\geq\sigma\geq0$, on $Y$. If also $b^m \bar W_0 \in L^q(\R^2)^5$, then it satisfies the following decay estimates for any $q\in [1,2],$
\begin{align}
\|e_3^\perp\lp(\p^\alpha\overline{\mc{S}}(\tau,\sigma)\bar W_0\rp)\|_{m} &\leq \fr{C\re^{-(1/2+\mu)(\tau - \sigma)}}{a(\tau - \sigma)^{1/q - 1/2+|\alpha|/2}}\|b^m  \bar W_0\|_{L^q} ,\notag\\
 \|e_3^T\lp(\p^\alpha\overline{\mc{S}}(\tau,\sigma)\bar W_0\rp)\|_m 
   &\leq \fr{C\re^{-(\tau - \sigma)/2}}{a(\tau - \sigma)^{1/q - 1/2+|\alpha|/2}} 
  \| b^m e_3^T  \bar W_0 \|_{L^q}   \notag
\end{align}
for some constant $C>0$.
\end{Proposition}
\begin{proof}
See Appendix \ref{ss:Abt}.
\end{proof}

\begin{Proposition}(Baroclinic Evolution)\label{p:btev}
The linear operator $\tl L(\tau) + \tl \Lambda$ generates an evolutionary family of operators $\tilde{\mc{S}}(\tau,\sigma)$, for $\tau\geq \sigma$, on $L^2(m)^3\times H^1(m)$ which, for $m>1$,$\alpha\in \mb{N}^3$ with $|\alpha|\leq 1, q\in [1,2], 0<\delta<4\pi$, $f\in L^2(m)^3\times H^1(m)$ and with $b^m f\in L^q(\mb{D}),$ satisfies the following estimate
\begin{align}
\| \p^\alpha\tilde{\mc{S}}(\tau,\sigma)f\|_{*,m} &\leq \fr{\re^{-(4\pi^2-\delta)(\re^{\tau}-\re^\sigma)}}{a(\tau-\sigma)^{1/q - 1/2 + \fr{\alpha_1+\alpha_2}{2}}a(\re^{\tau}-\re^\sigma)^{(1/q-1/2)/2 + \alpha_3/2}} (||b^m e_4^\perp f||_{L^q(\mb{D})} + \|b^m e_4^T f\|_{W^{1,q}(\mb{D})}).\label{e:fullest2}
\end{align}
for $\tau>\sigma\geq0,$ and some constant $C>0.$
\end{Proposition}
\begin{proof}
See Appendix \ref{ss:Abc}.
\end{proof}

 \paragraph{Fixed point operator}

Having collected information about the linear system in $W_R :=(\bar W_R,\tl W)^T,$ we now study the nonlinear system \eqref{e:vsc5}-\eqref{e:vsc6}.  In particular we study mild solutions of this system via the equivalent integral formulation
\begin{align}
\bar W_R(\tau)&= \overline{\mathbf{F}}(W_R):=\overline{\mc{S}}(\tau,\sigma) \bar W_{R,0} + \int_0^\tau \overline{\mc{S}}(\tau,\sigma)\overline{\mathbf{N}}(W_R(\sigma))d\sigma,\label{e:bmd1}\\
\tl W(\tau)&= \tilde{\mathbf{F}}(W_R):= \tilde{\mc{S}}(\tau,\sigma) \tl W_{0} + \int_0^\tau \tilde{\mc{S}}(\tau,\sigma) \tilde{\mathbf{N}}(W_R(\sigma)) d\sigma.\label{e:bmd2}
\end{align}
We note here that neither $\overline{\mc{S}}$ nor $\tilde{ \mc{S}}$ act diagonally on vectors in $\R^5$ and $\R^4$ respectively, so the mild formulation cannot be broken down into components as in \cite{roussier2003stabilite}. Using this formulation we can then prove the following result about the asymptotic behavior of small solutions of \eqref{e:vsc4}.

\begin{Theorem} \label{t:alg-as}
 There exists $K_0 > 0$ such that for $0<\mu<1/2$, $m>3$, initial data $W_{R,0} = (\bar W_{R,0},\tl W_0)\in Y\times L^2(m)^3\times H^1(m)$, with $\nabla \cdot \tilde{w}_0 = 0$, and $\| \bar W_{R,0} \|_{m} + \|\tl W_0\|_{*,m} < K_0$, there exists a unique solution of \eqref{e:bmd1}-\eqref{e:bmd2} in $C^0([0,\infty), Y\times L^2(m)^3\times H^1(m)) $ which satisfies the asymptotic estimates:
\begin{align}
&\lim_{\tau \to \infty}  \Bigg[ \re^{\tau/2}\|\bar w_{R.3}(\cdot,\tau) \|_{L^2_{2D}(m)}\,+  \re^{(1/2+\mu)\tau}\lp(\| \overline{\Phi}_{R}(\cdot, \tau)  \|_{L^2_{2D}(m)^2} + \|\bar w_{R,h}(\cdot, \tau) \|_{L^2_{2D}(m)^2} \rp)\Bigg] = 0 , \notag \\
& \lim_{\tau \to \infty} \re^{\gamma \tau} 
\left[ \| \tilde{w}(\cdot,\tau) \|_{L^2(m)^3} + \| \tilde{\phi}(\cdot,\tau) \|_{L^2(m)} \right] = 0 .\label{e:asymp1}
\end{align}
for any $\gamma>3/4$.
\end{Theorem}
The solution of the residual equation for sufficiently small initial residual data then readily implies existence and asymptotics in the full solution $W = W_{app} + W_R$.
\begin{Corollary}\label{c:alg-as}
There exists $K_0>0$ such that for $0< \mu < 1/2$, initial data $W_0\in Z$ with $\| \bar W_0(\cdot) - \bar W_{app}(\cdot, 0) \|_{m} + \|\tl W_0\|_{*,m} < K_0$, there exists a unique solution of \eqref{e:vsc1} - \eqref{e:vsc4} in $C^0([0,\infty], Z)$ which satisfies the following asymptotics
\begin{align}
&\lim_{\tau \to \infty}  \Bigg[ \re^{\tau/2}\|\bar w_3(\cdot,\tau) - \bar w_{app,3}(\cdot,\tau)\|_{L^2_{2D}(m)}\,+  \notag\\
&\,\qquad\re^{(1/2+\mu)\tau}\lp(\| \overline{\Phi}(\cdot, \tau) - \bar{\Phi}_{app} (\cdot, \tau) \|_{L^2_{2D}(m)^2} + \|\bar w_h(\cdot, \tau) - \bar w_{app,h}(\cdot, \tau))\|_{L^2_{2D}(m)^2} \rp)\Bigg] = 0 , \notag \\
& \lim_{\tau \to \infty} \re^{\gamma \tau} 
\left[ \| \tilde{w}(\cdot,\tau) \|_{L^2(m)^3} + \| \tilde{\phi}(\cdot,\tau) \|_{H^1(m)} \right] = 0 .\label{e:asymp1}
\end{align}
for any $\gamma>3/4$.
\end{Corollary}

\begin{Remark} Note that if we rewrite these estimates in terms of our original variables, it
says that the asymptotics of \eqref{eq:red_asymptotics} are correct, (at least if we replace
the ${\mc{O}}(\frac{1}{t})$ error terms with ${\mc{O}}(\frac{1}{t^{(1-\epsilon)}})$ and similarly
for the ${\mc{O}}(\frac{1}{t^{3/2}})$ terms) and that the baroclinic components of the velocity
and temperature decay at least like ${\mc{O}}(\frac{1}{t^{(\gamma + 1/2 -\epsilon)}})$.  In fact we could probably use the methods below to establish that the baroclinic components decay more rapidly than any inverse power of $t$.  For the rotating fluid layer without thermal effects, Roussier-Michon shows (\cite[Thm. 3.3.1]{roussier2003stabilite}) that for the linearized evolution the baroclinic velocity actually decays exponentially fast.
\end{Remark}

 \begin{proof}
 To begin, we define a Banach space
\begin{equation}
X_{\mu,\gamma} = \{ W_R, \nabla W_R \in C^0([0,\infty);Y\times L^2(m)^3\times H^1(m)) ~|~  \nabla_\tau \cdot\tl w(\tau) = \nabla_h\cdot\overline{\mfu}_h(\tau)= \nabla_h\cdot \overline{w}_h= 0\} ,
\end{equation}
with norm 
\begin{align}
\| W \|_{X_{\mu,\gamma}} &= \sup_{\tau \ge 0} e^{\tau/2} \Bigg( 
\| \bar w_{R,3}(\tau)\|_{m} + \re^{\mu\tau}\|e_3^\perp \bar W_R(\tau)\|_{m} + a(\tau)^{1/2}(\|\nabla \bar w_{R,3}(\tau)\|_{m} + \re^{\mu \tau}\|e_3^\perp\nabla \bar W_R(\tau)\|_{m})\Bigg) \notag\\
&+e^{ \gamma\tau}\left(  \|\tilde{W}(\tau) \|_{*,m} +  a(\tau)^{1/2}\| \nabla_h \tilde{W}(\tau) \|_{*,m}+ a(\re^\tau)^{1/2}\| \p_3\tilde{W}(\tau) \|_{*,m}\right) ,
\end{align}
where we recall that $e_3^\perp = 1 - e_3e_3^T$ with $e_3 = (0,0,1,0,0)^T.$ We let $\mc{S}(\tau,\sigma)$ be the direct sum of the evolution operators $\overline{\mc{S}}(\tau,\sigma)$ and $\tilde{\mc{S}}(\tau,\sigma)$, and $F = (\bar F, \tl F)^T$ where $\bar F$ and $\tl F$ are the integral terms in \eqref{e:bmd1} and \eqref{e:bmd2} respectively.  The existence of a global solution then follows by finding a fixed point of the mapping, 
$$
\mathbf{F}(W_R):= (\overline{\mathbf{F}}(W_R),\tilde{\mathbf{F}}(W_R))^T= S(\tau,\sigma) W_{R,0} + F(W_R)
$$
defined on $X_{\mu,\gamma}$. Paired with the linear estimates on $\overline{\mc{S}}$ and $\tilde{\mc{S}}$, this follows in a standard fashion from the estimates
\begin{align}
\| F(W_R) \|_{X_{\mu,\gamma}}  & \le
C_1\| W_R \|_{X_{\mu,\gamma}}^2  , \label{e:inv} \\
\| F(W_R) - F(W'_R) \|_{X_{\mu,\gamma}} &  \le C_2 \left( \sup_{\tau \ge 0} \| W_R( \tau) \|_{Z} +
\sup_{\tau \ge 0} \| W' _R( \tau) \|_{Z} \right)  \| W_R - W'_R \|_{X_{\mu,\gamma}} \ . \label{eq:lipshitz}
\end{align}
The estimate for \eqref{eq:lipshitz} follows in a similar way to \eqref{e:inv}.  For \eqref{e:inv}, since the evolutionary operators $\bar S$ and $\tl S$ do not act diagonally we cannot consider the nonlinearities in $\mathbf{F}$ component-wise as in \cite{roussier2003stabilite}. We thus use Proposition \ref{p:SS} to estimate $\overline{\mathbf{F}}$ and Proposition \ref{p:btev} to estimate $\tilde{\mathbf{F}}$, both with $q = 3/2.$  

\paragraph{Estimates on $\overline{\mathbf{F}}$:}

We first find
\begin{align}
\| e_3^\perp &\bar F(W_R)(\tau)\|_{m} \leq C \int_0^\tau \|e_3^\perp \overline{\mc{S}}(\tau,\sigma) \overline{\mathbf{N}}(\bar W_R,\tl W) \|_m d\sigma\notag\\
&\leq C \int_0^\tau \fr{\re^{-(1/2+\mu)(\tau-\sigma)}}{a(\tau - \sigma)^{1/6}} \Big[   
\| b^m \overline{\mfu}_{R,h} \cdot \nabla \bar W_R\|_{L^{3/2}} +  \| b^m \bar w_{R,h}\cdot\nabla_h \overline{\mfu}_{R,h}\|_{L^{{3/2}}} + \| b^m \overline{\Phi}_{R}\cdot\nabla_h \overline{\mfu}_{R,h}\|_{L^{{3/2}}} \notag\\
&\quad+ \|b^m \tilde{\mfu}\cdot\nabla_\sigma \tl w\|_{L^{3/2}} + \|b^m \tilde{\mfu}\cdot\nabla_\sigma (\nabla_h^\perp\tl \phi)\|_{L^{3/2}}  +\|b^m\tl w\cdot \nabla_\sigma \tilde{\mfu}\|_{L^{{3/2}}} + \|b^m\nabla_h^\perp\tl\phi\cdot \nabla_h\tilde{\mfu}\|_{L^{{3/2}}} \notag\\
&\quad+\re^{\sigma/2} \|b^m \p_3\tilde\phi\cdot (\nabla_h^\perp\tilde{ \mathfrak{u}}_3)\|_{L^{{3/2}}} + \re^{\sigma/2}\|b^m \p_3 \tilde{\mfu}_3\cdot (\nabla_h^\perp \tl\phi) \|_{L^{{3/2}}}
\Big]d\sigma\notag\\
&\leq  C \int_0^\tau \fr{\re^{-(1/2+\mu)(\tau-\sigma)}}{a(\tau - \sigma)^{1/6}} \Big[
\| \bar w_{R,3}\|_m \|\nabla \bar W_R\|_m + ( \|\bar w_{R,h}\|_m + \|\bar \Phi_R\|_m)\|\bar w_{R,3}\|_{H^1(m)} \notag\\
&\quad + \re^{\sigma/2}\| \tl w\|_m \|\nabla \tl W\|_{*,m} + \re^{\sigma/2} \|\tl W\|_{*,m} \|\tl w\|_{H^1(m)} \Big]d\sigma \notag
\end{align}
where $C>0$ is a constant which may change from line to line. To obtain the last inequality, we go term-by-term. The two-dimensional Biot-Savart law in Proposition \ref{p:bsest}, the embedding $L^2_{2D}(m)\hookrightarrow L^q(\R^2)$ for all $q\in [2/(m+1),2],$ and H\"{o}lder's inequality give
\begin{align}
\|b^m \overline{\mfu}_{R,h}\cdot \nabla \bar W_R\|_{L^{3/2}} &\leq\| \overline{\mfu}_{R,h}\|_{L^6(\R^2)}\|\nabla \bar W_R\|_m\notag\\
&\leq  C\| \overline{w}_{R,3}\|_{L^{3/2}(\R^2)}\| \nabla \bar W_R\|_m\notag\\
&\leq C\| \overline{w}_{R,3}\|_m\| \nabla \bar W_R\|_m.\notag
\end{align}
The estimate $\|\nabla \overline{\mfu}_h\|_{L^p} \leq  \|\overline{w}_3\|_{L^p}$, which can be derived from \cite[Lem. 2.1]{gallay2002invariant}, along with the Gagliardo-Nirenberg inequality gives
\begin{align}
\| b^m\bar w_{R,h}\cdot \nabla \overline{\mfu}_{R,h}\|_{L^{3/2}} &\leq \|\bar w_{R,h}\|_m \|\nabla \overline{\mfu}_{R,h}\|_{L^6}\notag\\
&\leq C\|\bar w_{R,h}\|_m\|\bar w_{R,3}\|_{L^{6}}\notag\\
&\leq C \|\bar w_{R,h} \|_m \|\nabla \bar w_{R,3}\|_{L^2}^{2/3} \| \bar w_{R,3}\|_{L^2}^{1/3}\notag\\ 
&\leq  C\|\bar w_{R,h}\|_m(\|\bar w_{R,3}\|_{m} + \|\nabla \bar w_{R,3}\|_{m}).\notag
\end{align}
A similar approach gives 
\begin{align}
\|b^m \overline{\Phi}_R\cdot \nabla_h \overline{\mfu}_{R,h}\|_{L^{3/2}}\leq C \|\overline{\Phi}_R\|_m (\|\bar w_{R,3}\|_{m} + \|\nabla \bar w_{R,3}\|_{m}).
\end{align}
For the baroclinic contributions, the three-dimensional Biot-Savart law in Proposition \ref{p:bsest}, the embedding $L^2(m)\hookrightarrow L^q(\mD)$ for all $q\in [1,2]$, and finally the estimate $\|\nabla \tilde{\mfu}\|_{L^6} \leq  \|\tilde{w}\|_{H^1(m)}$ from \cite[Lem. 2.4.2]{roussier2003stabilite} give
\begin{align}
\|b^m \tl w\cdot \nabla_\sigma \tilde{\mfu}\|_{L^{3/2}}&\leq C\re^{\sigma/2} \|\tl w\|_m \|\tl w\|_{H^1(m)},\notag\\
\|b^m\tilde{\mfu}\cdot \nabla_\sigma \tl w\|_{L^{3/2}} + \|b^m\tilde{\mfu}\cdot \nabla_\sigma \nabla^\perp\tl\phi\|_{L^{3/2}}&\leq C\re^{\sigma/2}\|\tilde{\mfu}\|_{L^6} \|\nabla \tl W\|_{*,m}\notag\\
&\leq C\re^{\sigma/2}\|\tilde{w}\|_{L^2} \|\nabla \tl W\|_{*,m}\notag\\
&\leq C\re^{\sigma/2} \|\tl w\|_m \|\nabla \tl W\|_{*,m}.
\end{align}
We next have
\begin{align}
\| b^m \nabla_h^\perp \tl \phi \cdot \nabla_h \tilde{\mfu}\|_{L^{3/2}} + \re^{\sigma/2}\| b^m \p_3\tl \phi \nabla_h^\perp \tilde{\mfu}_3\|_{L^{3/2}} + \re^{\sigma/2}\|b^m \p_3\tilde{\mfu}_3 \nabla_h^\perp \tilde\phi\|_{L^{3/2}}&\leq
C\re^{\sigma/2}\|\nabla \tl\phi\|_m \|\nabla \tilde{\mfu}\|_{L^6}\notag\\
&\leq C\re^{\sigma/2}\|\tl W\|_{*,m} \|\tl w\|_{H^1(m)}.
\end{align}
From this we can then conclude
\begin{align}
\re^{(1/2+\mu)\tau}\| e_3^\perp\bar F(W_R)(\tau)\|_{m} &\leq C\int_0^\tau \fr{\re^{(1/2+\mu)\sigma}}{a(\tau - \sigma)^{1/6}} \Big[ \fr{\re^{-\sigma}}{a(\sigma)^{1/2}}  \| \bar W_R\|_{X_{\mu,\gamma}}^2 + \fr{\re^{-(1+\mu)\sigma}}{a(\sigma)^{1/2}}\| \bar W_R\|_{X_{\mu,\gamma}}^2\notag\\
&\qquad\qquad\qquad\qquad\quad+ \re^{(1/2 - 2\gamma )\sigma}(C + a(\sigma)^{-1/2} + a(\re^\sigma)^{-1/2}) \|\tl W\|_{X_{\mu,\gamma}}^2 \Big] d\sigma\notag\\
&\leq C \| W_R\|_{X_{\mu,\gamma}}^2,\label{e:nlfp1}
\end{align}
and in a similar manner
\begin{align}
\re^{\tau/2}\|e_3^T \bar F(W_R)\|_{m}&\leq C \int_0^\tau \re^{\sigma/2} \Big[  \| \bar w_{R,3}\|_m \|\nabla \bar w_{R,3}\|_{m}+ \re^{\sigma/2}( \|\tl w\|_m \|\nabla \tl W\|_m + \|\tl w\|_m \|\tl W\|_{H^1(m)})\Big] d\sigma\notag\\
&\leq C\|W_R\|_{X_{\gamma,\mu}}^2.\label{e:nlfp2}
\end{align}
Here we make a slight abuse of notation, letting $\| \bar W_R\|_{X_{\mu,\gamma}}$ and $\|\tl W\|_{X_{\mu,\gamma}}$ denote the $X_{\mu,\gamma}$ norm of $(\bar W_R,0)$ and $(0,\tl W)$ respectively.
Continuing in this way one readily obtains estimates on the gradient terms
\begin{align}
a(\tau)^{1/2}\re^{\tau/2} \lp( \|e_3^T\nabla \bar F(W_R)\|_m + \re^{\mu\tau} \|e_3^\perp \nabla \bar F(W_R)\|_m\rp)\leq C\| W_R\|_{X_{\mu,\gamma}}^2.\label{e:nlfp3}
\end{align}
Combining \eqref{e:nlfp1}, \eqref{e:nlfp2} and \eqref{e:nlfp3} we then obtain 
$$
\| \bar F(W_R)\|_{X_{\mu,\gamma}} \leq C \|W_R\|_{X_{\mu,\gamma}}^2.
$$

\paragraph{Estimates on $\tilde{\mathbf{F}}$:}

For $\tl F$ we use similar estimates and \eqref{e:fullest2} of Proposition \ref{p:btev} with $q =3/2$ to find
\begin{align}
\re^{\gamma \tau}\|\tl F(W_R)\|_{*,m} &\leq 
\re^{\gamma\tau}\int_0^\tau\Big\|\tilde{\mc{S}}(\tau,\sigma)(1-Q)\Bigg[ 
\overline{\mfu}_R\cdot \nabla_\sigma \tl W  - \tl w\cdot \nabla \left(\begin{array}{c}\overline{\mfu}_R \\0\end{array}\right) + \tilde{\mfu}\cdot \nabla \left(\begin{array}{c}\overline{w}_R \\ \bar \phi_R\end{array}\right) \notag\\
&\qquad- \bar w_R\cdot \nabla_\sigma\left(\begin{array}{c}\tilde{\mfu} \\0\end{array}\right) + \tilde{\mfu}\cdot \nabla_\sigma \tl W - \tl w\cdot\nabla_\sigma \left(\begin{array}{c}\tilde{\mfu} \\0\end{array}\right)\Bigg]\Big\|_{*,m} d\sigma\notag\\
&\leq C\int_0^\tau \fr{\re^{-(4\pi^2-\delta)(\re^\tau - \re^\sigma)}\re^{\gamma\tau} }{a(\tau - \sigma)^{1/6}a(\re^{\tau} - \re^{\sigma})^{1/12}}\Bigg[
\|b^m\overline{\mfu}_R\cdot \nabla_\sigma \tl w\|_{L^{3/2}} + \|b^m\overline{\mfu}_R\cdot \nabla_\sigma \tl\phi\|_{W^{1,3/2}}
+ \|b^m\tl w\cdot \nabla \overline{\mfu}_R\|_{L^{3/2}}\notag\\
&\qquad\qquad+ \|b^m\tilde{\mfu}\cdot \nabla \overline{w}_R\|_{L^{3/2}} + \|b^m \tilde{\mfu}\cdot \nabla \bar\phi_R\|_{W^{1,{3/2}}} 
+ \|b^m\bar w_R\cdot \nabla_\sigma \tilde{\mfu}\|_{L^{3/2}}\notag\\
&\qquad\qquad+ \|b^m \tilde{\mfu}\cdot \nabla_\sigma \tl w\|_{L^{3/2}} + \| b^m \tilde{\mfu}\cdot \nabla_\sigma \tl \phi\|_{W^{1,{3/2}}}
+ \| b^m\tl w\cdot\nabla_\sigma\tilde{\mfu} \|_{L^{3/2}}
\Bigg]d\sigma\notag\\
&\leq C\int_0^\tau \fr{\re^{-(4\pi^2-\delta)(\re^\tau - \re^\sigma)}\re^{\gamma\tau} }{a(\tau - \sigma)^{1/6}a(\re^{\tau} - \re^{\sigma})^{1/12}}\Bigg[
\|\bar w_{R,3}\|_m\lp( \|\tl W\|_{*,m} + \re^{\sigma/2} \|\nabla \tl W\|_{*,m}\rp) \notag\\
&\qquad+ \|e_3^\perp \bar W_R\|_{m}\lp( \re^{\sigma/2}\|\tl W\|_{*,m} + \|\nabla \tl W\|_{*,m}\rp) 
+ \|\tl W\|_{*,m} \|\nabla \bar W\|_{m} + \re^{\sigma/2} \|\tl W\|_{*,m} \|\nabla \tl W\|_{*,m}
\Bigg] d\sigma\notag\\
&\leq C\int_0^\tau \fr{\re^{-(4\pi^2-\delta)(\re^\tau - \re^\sigma)}\re^{\gamma(\tau - \sigma)} }{a(\tau - \sigma)^{1/6}a(\re^{\tau} - \re^{\sigma})^{1/12}}\Bigg[ 1 + a(\sigma)^{-1/2} a(\re^\sigma)^{-1/2}
\Bigg]d\sigma\,\, \|W_R\|_{X_{\mu,\gamma}}^2.
\end{align}
Here we have used similar estimates as for the bound on $\bar F$ described above as well as the following
\begin{align}
\|b^m \tilde{\mfu}\cdot\nabla_\sigma \tl \phi\|_{W^{1,{3/2}}}&\leq C\lp(\|b^m \tilde{\mfu}\cdot\nabla_\sigma \tl \phi\|_{L^{3/2}} + \|b^m \nabla(\tilde{\mfu}\cdot\nabla_\sigma \tl \phi)\|_{L^{3/2}}\rp)\notag\\
&\leq C(\|\tilde{\mfu}\|_{L^6}\|b^m \nabla_\sigma \tilde\phi\|_{L^2} + \|\nabla\tilde{\mfu}\|_{L^6}\|b^m \nabla_\sigma\tl\phi\|_{L^2} + \|\tilde{\mfu}\|_{L^6}\|\nabla_\sigma\tl\phi\|_{H^1(m)})\notag\\
&\leq C\re^{\sigma/2}(\|\tilde{w}\|_m\| \nabla \tilde\phi\|_{m} + \| \tilde{w}\|_{H^1(m)}\|\nabla\tl\phi\|_{m} + \|\tilde{w}\|_m\|\nabla\tl\phi\|_{H^1(m)}),\\
\|b^m \tilde{\mfu}\cdot\nabla \bar\phi_R\|_{W^{1,3/2}}\|&\leq C(\|b^m \tilde{\mfu}\cdot\nabla\bar \phi_R\|_{L^{3/2}} + \|b^m\nabla(\tilde{\mfu}\cdot\nabla\bar \phi_R)\|_{L^{3/2}})\notag\\
&\leq C(\|\tilde{\mfu}\|_{L^6} \|b^m \nabla\phi_R\|_{L^2} + \|\nabla \tilde{\mfu}\|_{L^6}\|b^m\nabla\bar\phi_R\|_{L^2} + \|\tilde{\mfu}\|_{L^6} \|\nabla\bar\phi_R\|_{H^1(m)})\notag\\
&\leq C( \|\tl w\|_m\|\bar\Phi_R\|_m + \|\tl w\|_{H^1(m)}\|\bar\Phi_R\|_m + \|\tl w\|_m \|\bar\Phi_R \|_{H^(m)}  )
\end{align}
and other estimates which follow in a similar way.  Then, along with similar estimates for $a(\tau)^{1/2}\|\nabla \tl F\|_m$, we obtain
$$
\|\tl F(W_R)\|_{X_{\mu,\gamma}} \leq C \|W_R\|_{X_{\mu,\gamma}}^2.
$$

We thus obtain the quadratic estimate on $F$ in \eqref{e:inv} and conclude the existence of a fixed point for sufficiently small $W_0\in Z$.  This implies global existence of solutions and the structure of the $X_{\mu,\gamma}$ norm implies the temporal decay prescribed in \eqref{e:asymp1}.

 \end{proof}

\section{Global existence and asymptotics}
\subsection{Dispersive Estimates}\label{s:disp}

The proofs of the different choices of boundary conditions in Theorems \ref{t:0} are very similar, so we only provide the details
of the proof for 
 the stress-free case.  Hence for the remaining sections we will consider the boundary conditions \eqref{e:sfbc}.  Indeed, the proofs of these results follow very closely the strategy used by Roussier-Michon and Gallay in \cite{gallay2009global} so we will mainly just highlight the differences in the proof necessitated
 by the additional temperature dependence, {\it vis-a-vis} the purely rotational problem considered in
 that reference.  In contrast to the preceding section, the dispersive nature of the linearized problem is paired with large-rotation rate, $|\Omega|\gg1$, to obtain global existence for initial data in which the 
 smallness assumptions are imposed only on the quasi-geostrophic part of the initial data.

To characterize such dispersive effects, it will be important  to consider the linear Rossby-type equation
\beq\label{e:lrs}
\p_t\tl v+ \Gamma\mb{P}J_{\eta} \mb{P} \tl v = \Delta \tl v,\quad \mathrm{div}\tl u = 0,
\eeq
where $\mb{P}J_\eta\mb{P}$ was defined in \eqref{e:pjp-sf} above.  We note that since the spatial domain $\mb{D}$ is bounded in the vertical direction, and since this equation only acts on the baroclinic part of the solution (i.e. the part with non-trivial $x_3$-dependence), all eigenvalues of the linear operator will have a negative real part $\sim -4 \pi n^2$, where $n$ is the Fourier index in the $x_3$ direction. This
immediately leads to the fact that for all $s\geq0$ and $\tl v_0\in (1 - Q)H^s(\mb{D})^4$ with $\mathrm{div} \tl v_0 = 0$, a solution $\tl v(t)$ to \eqref{e:lrs} satisfies
\beq\label{e:poin}
||\tl v(t)||_{H^s(\mb{D})} \leq ||\tl v_0||_{H^s(\mb{D})} \re^{-4\pi^2 t},\quad t\geq 0.
\eeq

In addition,  recalling that $\mb{P}J_\eta\mb{P}$  is anti-symmetric, we see that
\begin{equation}
\frac{1}{2} \partial_t \| \nabla \tilde{v} \|_{L^2(\mD)}^2 = - \| \Delta \tilde{v} \|_{L^2(\mD) }^2\ .
\end{equation}
From this, we immediately conclude
\begin{Lemma}  If $\tilde{v}_0 \in H^1(\mD)^4$, then for any $T >0$, solutions of \eqref{e:lrs} satisfy
\begin{equation}
\int_0^T \| \Delta \tilde{v}(\cdot, t) \|_{L^2}^2 dt \le \frac{1}{2} \| \nabla \tilde{v}_0 \|_{L^2}^2 \ .
\end{equation}
\end{Lemma}

\begin{Remark}\label{r:decomp}
Due to the form of the eigenvalues of the linearized equation, solutions  $
\hat{\tilde{v}}(k,t) = \hat{\tilde{v}}(k_h,n,t)$
with $|k|  = \sqrt{k_1^2+k_2^2+4\pi^2 n^2}\sim |k_1|+|k_2|+|n| \geq R$, decay like $\sim e^{-R^2 t}$, so 
in order to understand the dispersive properties of the solutions, it suffices to study the part of the solution
localized in a neighborhood of zero in Fourier space.
\end{Remark}

With this in mind, let $B_R = \{k=(k_h,n)\in \R^2\times \mb{Z} \,|\, | k |\leq R\}$ and $S $ be the projection onto the geostrophic eigenspace, defined as above
$$
\widehat{Sv}(k) = \la a_g(k),\hat v(k)\ra_{\C^4} a_g(k).
$$ 
Because the quasi-geostrophic eigenvalue $-|k|^2$ corresponding to $a_g$ is independent
of $\Omega$ and $\Gamma$, we cannot expect any dispersive smoothing in this mode.
However for the other modes, we can 
prove the following estimates on solutions of the linear Rossby equation \eqref{e:lrs} for initial data compactly supported in Fourier space perpendicular to the geostrophic mode.
\begin{Proposition}\label{p:dis}
For any $R>0$, there exists $C_R>0$ such that, for all $\tl v_0\in (I - Q) L^2(\mb{D})^4$ with $\mathrm{div} \,\tl u_0 = 0$ and $\mathrm{supp}\,\widehat{\tl v}_0 \subset B_R$, the solution $\tl v$ of \eqref{e:lrs} satisfies
\beq
|| \tl v||_{L^1(\R_+,L^\infty(\mb{D}))} \leq C_R\lp( |\eta|^{-1/4} ||(1-S)\tl v_0||_{L^2(\mb{D}) }+||S\tl v_0||_{L^2(\mb{D})}\rp) .
\eeq
\end{Proposition}
\begin{proof}

In Fourier space, equation \eqref{e:lrs} takes the form
\beq\label{e:lrsf}
\lp(\p_t  +|k|^2I_4 + \Gamma\widehat{\mb{P}J_{\eta}\mb{P}} \rp) \widehat{\tl v}(k_h,n,t) = 0,\quad n\in \mb{Z}\diagdown \{0\}, k_h\in \R^2,
\eeq

We decompose $\tilde{v} = \tilde{v}^{+} + \tilde{v}^{-} + \tilde{v}^g$, where in Fourier space
\begin{align}
\widehat{\tilde{v}}^{\pm}(k,t)  & = e^{-t |k|^2 \pm i t \Gamma p_{\eta}(k) } \langle (\widehat{\tilde{v}}_0)(k) , a_{\pm}(k) \rangle_{\mb{C}^4} a_{\pm}(k) \notag\\
\widehat{\tilde{v}}^g(k,t) & =  \re^{-|k|^2 t} \la (\widehat{\tl v}_0)(k),a_g(k)\ra_{\mb{C}^4} a_g(k)\ ,
\end{align}
where we recall that the explicit expressions for the eigenvectors $a_g$ and $a_{\pm}$ are 
given in \eqref{e:evecs-p}.
Note here that the incompressibility condition gives $\la (\widehat{\tl v}_0)(k),a_0(k)\ra_{\mb{C}^4} = 0.$
Also note that the forms of the modes $\widehat{\tilde{v}}^{\pm}$ are almost identical to those
of the dispersive modes in \cite[Eqn. (87)]{gallay2009global}. (see also \cite[\S 4.2]{charve2005convergence}  for a discussion of similar dispersive estimates, albeit in an unbounded domain.)  Following the arguments of (\cite{gallay2009global}; App.B) more-or-less line for line, one obtains
\begin{equation}
|| \tl v^{\pm} ||_{L^1(\R_+,L^\infty(\mb{D}))} \leq C_R\lp( |\eta|^{-1/4} ||(1-S)\tl v_0||_{L^2(\mb{D}) } \rp)\ .
\end{equation}

Thus, we need only estimate $\tilde{v}^g$.  However this term can be estimated directly without recourse
to the duality methods used in \cite{gallay2009global}.  Note that 
\begin{equation}
\tilde{v}^g(x,t)  = \sum_{n\neq0} \int_{\R^2}  e^{i k_h\cdot x_h} e^{2\pi i n x_3} e^{-| k_h |^2t} e^{-4 \pi^2 n^2 t} 
 \la (\widehat{\tl v}_0)(k),a_g(k)\ra_{\mb{C}^4} a_g(k) d k_h\ .
\end{equation}
Thus, using the fact that all terms have $n\neq 0$, and the support condition on $\tl v_0$ we find
\begin{align}
| \tilde{v}^g(x,t) | & \le  \sum_{n\neq0} \int_{\R^2}  e^{-| k_h |^2t} e^{-4 \pi^2 n t} 
| \la (\widehat{\tl v}_0)(k),a_g(k)\ra_{\mb{C}^4} a_g(k) |  d k_h  \\ \nonumber
& \le    \sum_{n \ne 0}  \int_{|k_h|\le R} e^{-| k_h |^2t} e^{-4 \pi^2 n t} 
| \la (\widehat{\tl v}_0)(k),a_g(k)\ra_{\mb{C}^4} a_g(k) |  d k_h \le C_R^g e^{-4 \pi^2 t}\ ,
\end{align}
Integrating this
estimate with respect to $t$ and combining it with the estimates on $\tilde{v}^{\pm}$ immediately
yields Proposition \ref{p:dis}.

\end{proof}

Note that the restriction on the support of $\hat{\tilde{v}}$ in Proposition \ref{p:dis} means that the
solution of \eqref{e:lrs} lies in any Sobolev space $H^s$ with $s \ge 0$.  From this, we immediately
obtain the following:
\begin{Corollary} With the assumptions of Proposition \ref{p:dis}, for any $1 \le p \le \infty$ and
$2 \le q \le \infty$, with $\frac{1}{p}+\frac{2}{q} \le 1$, the solution of \eqref{e:lrs} obeys:
$$
|| \tl v||_{L^p(\R_+,L^q(\mb{D}))} \leq C_R\lp( |\eta|^{-1/(4p)} ||(1-S)\tl v_0||_{L^2(\mb{D}) }+||S\tl v_0||_{L^2(\mb{D})}\rp) .
$$
\end{Corollary}

\subsection{Global Existence}\label{s:glob}

To make use of these dispersive properties in the nonlinear problem we make the following decomposition
$$
\tl v(x,t) = \lambda(x,t) + r(x,t),
$$
where $\lambda$ satisfies
\beq\label{e:lrl}
\p_t\lambda+ \Gamma \mb{P}J_{\eta}\mb{P}\lambda = \Delta \lambda, \quad \mathrm{div}\lambda = 0,
\eeq
with initial data $\lambda_0 = P_R \tl v_0$ where $P_R$ is the multiplier defined by
\begin{equation}\label{e:PRdef}
\widehat{(P_R f)}_n(k) = \chi\lp( \fr{\sqrt{|k|^2 + (2\pi n)^2}}{R} \rp)\hat f_n(k),
\end{equation}
and $\chi$ is a smooth function with $\chi(k) = 1$ for all $|k|<1$ and $\chi(k) = 0$ for $|k|>2.$ The remainder $r$ must then solve
\beq\label{e:rt}
\p_t r + \Gamma\mb{P}J_{\eta}\mb{P} r + N_3 = \Delta r, \quad \mathrm{div}r = 0,
\eeq
with $N_3 = \mb{P}\lp[  (\bar u\cdot \nabla)\tl v + (\tl u\cdot \nabla) \bar v + (1 - Q) (\tl u\cdot \nabla)\tl v   \rp]$ and initial condition $r_0 = (1 - P_R) \tl u_0$.

The estimates of the previous section control the evolution of $\lambda$, while we expect $r$ will decay exponentially fast at the linear level (see Remark \ref{r:decomp}).  Such linear estimates are then sufficient to show that solution $r(t)$ with $r_0$ small will remain so. Also note, we obtain the following local existence result using a standard fixed point argument.
\begin{Proposition}(Local Existence)\label{p:lex}
For any $R>0$, there exists $T_R>0$ such that for all $\Omega, \Gamma\in \R$ and initial data $v_0\in X_\mathrm{sf}$ with $||v_0||_{X_\mathrm{sf}}<R$ the equation \eqref{e:1} has a unique local solution $v\in C^0([0,T_R],X_\mathrm{sf})$ satisfying $v(0) = v_0$.
\end{Proposition}

Note that this local existence result imposes no restriction on the size of the initial data.  We now
show that if the initial size of the geostrophic component of the solution is sufficiently small, the
solution can be extended for all time, by deriving a bound on the solution in  the $X_\mathrm{sf}$-norm 
which is uniform in time.  
We will prove
\begin{Theorem}\label{t:gle}
For all $v_0\in X_\mathrm{sf}$ with $S(1-Q)v_0$ sufficiently small, that is the projection onto the baroclinic-geostrophic portion of $u_0$ is small, there exists $\Omega_0$ such that for all $\Omega\in \R$ with $|\Omega|\geq \Omega_0$ the stably-stratified system \eqref{e:1} with stress-free boundary conditions has a unique  global solution $v\in C^0([0,\infty),X_\mathrm{sf})$ with $v(0) = v_0$.  Furthermore, there exists a constant $C>0$ such that $||v(t)||_{X_\mathrm{sf}} \leq C$ for all $t>0$. 
\end{Theorem}

As noted above, the proof of this theorem follows by showing that there exists a uniform
bound of the functional
\begin{equation}
\tilde{\Psi}(t) = \| \overline{\omega}_3(t) \|_{L^1} +  \| \overline{\omega}_3(t) \|^2_{L^2}
+ \| \tilde{v}(t) \|^2_{H^1(\mD)}\ .
\end{equation}
Such a bound implies that the $X_\mathrm{sf}$-norm of the solution is uniformly bounded and hence 
we can extend the local existence theorem indefinitely.
Using the decomposition of $\tilde{v}$ defined above, a bound on $\tilde{\Psi}$ is equivalent
to a bound on 
\begin{equation}
\overline{\Psi}(t) = \| \overline{\omega}_3(t) \|_{L^1} +  \| \overline{\omega}_3(t) \|^2_{L^2}
+ \| r(t) \|^2_{H^1(\mD)}+ \| \lambda(t) \|^2_{H^1(\mD)} \ .
\end{equation}
From the estimates in Section \ref{s:disp}, we conclude that $\sup_{t \ge 0} \| \lambda(t) \|^2_{H^1(\mD)} 
\le  \| \lambda_0  \|^2_{H^1(\mD)} $.
In addition, Poincar\'e's inequality implies that $ \| r(t) \|^2_{H^1(\mD)} \le C \| \nabla r(t) \|_{L^2(\mD)}^2$.
Thus, our bound on the $X_\mathrm{sf}$-norm of the solution will follow from a bound on 
\begin{equation}\label{e:psidef}
\Psi(t) = \| \overline{\omega}_3(t) \|_{L^1} +  \| \overline{\omega}_3(t) \|^2_{L^2} + \| \nabla r(t) \|_{L^2}^2\ .
\end{equation}

We control the evolution of the various terms in $\Psi(t)$ using the following energy estimates
modeled on \cite{gallay2009global}:

\begin{Proposition}\label{p:en}
There exists a constant $C_1$ such that if $v\in C^0([0,T],X_\mathrm{sf})$ is a solution of \eqref{e:1} for some $\Gamma,\Omega\in \R$ and if $v_0$ as above for some $R>0$, then the solutions of \eqref{e:lrl}, \eqref{e:rt}, and  \eqref{e:2a} satisfy for any $t\in (0,T]:$ 
\begin{align}
\fr{d}{dt}||\bar\omega_3(t)||_{L^2(\R^2)}^2&\leq - ||\nabla \bar\omega_3(t)||^2_{L^2(\R^2)} + 8 || \, |\tl u(t)|\, |\nabla\tl u(t)| ||^2_{L^2(\mD)},\label{e:est3}\\
||\bar \omega_3(t)||_{L^1(\R^2)}&\leq ||\bar\omega_3(0)||_{L^1(\R^2)} + 2 \int_0^t ||\tl u(s)||_{L^2(\mD)}\, ||\Delta \tl u(s)||_{L^2(\mD)} ds,\label{e:est4} \\
\fr{d}{dt}|| \nabla r(t)||^2_{L^2(\mD)}&\leq - ||\Delta r(t)||^2_{L^2(\mD)} + C_1 ||\nabla r(t)||^2_{L^2(\mD)} \,||\nabla\bar u(t)||_{L^2(\R^2)}^2 \, ||\Delta \bar u(t)||^2_{L^2(\R^2)}\notag\\
&\quad + C_1\lp( ||\bar u(t)||^2_{L^4(\R^2)}||\nabla\lambda(t)||_{L^4(\mD)}^2 + ||\nabla \bar u(t)||_{L^2(\R^2)}^2||\lambda(t)||_{L^\infty(\mD)}^2 + ||\,|\tl u(t)|\,|\nabla \tl v(t)|\,||_{L^2(\mD)}^2 \rp).\label{e:est5}
\end{align}

\end{Proposition}

\begin{proof}
Note that for stress-free boundary conditions, the equations for $\bar{u}_h$ and $\bar{\omega}_3$ are
exactly the same as they are in the case of rotating fluids, so the proofs of \eqref{e:est3} and \eqref{e:est4}
follow exactly as in Proposition 2.5 of \cite{gallay2009global}. 
For \eqref{e:est5} we must make a few small changes to account for the presence of the temperature
term in our equation.    If we compute $\frac{1}{2}\fr{d}{dt} \| \nabla r \|_{L^2}^2$, the dissipative term in \eqref{e:rt} gives rise to $- \| \Delta r\|_{L^2}^2$, while the term $\Gamma\mb{P}J_{\eta}\mb{P} r$ makes no
contribution due to anti-symmetry.  Thus, we need only estimate the contributions of the nonlinear
term:
\begin{equation}
| \int_\mD (\Delta r)\cdot \tl N dx | \le | \int_\mD (\Delta r) \cdot\left[ (\bar{u} \cdot \nabla ) \tilde{v} + (\tilde{u} \cdot \nabla)
\bar{v} - (1-Q) (\tilde{u} \cdot \nabla) \tilde{v} \right] dx |\ ,
\end{equation}
and the Helmholtz projector has vanished due to the fact that $\nabla \cdot r = 0$. These terms are broken up and estimated in turn. For example:
\begin{equation}
| \int (\Delta r) \cdot (\bar{u} \cdot \nabla) \tilde{v} dx | = | \int (\Delta r) \cdot (\bar{u} \cdot \nabla) \lambda  dx + \int (\Delta r) \cdot (\bar{u} \cdot \nabla) r dx
|\ ,
\end{equation}
and each of these terms is estimated in a fashion analogous to that used in  \cite{gallay2009global},
leading to the bound
\begin{equation}
| \int (\Delta r) \cdot (\bar{u} \cdot \nabla) \tilde{v} dx | \leq \frac{1}{16} \| \Delta r \|_{L^2}^2 + C \| \bar{u} \|_{L^4}^2 \| \nabla \lambda\|_{L^4}^2 
+ \frac{1}{16} \| \Delta r \|_{L^2}^2 + C \| \nabla r \|_{L^2}^2 \| \nabla \bar{u} \|_{L^2}^2 \| \Delta \bar{u} \|_{L^2}^2 \ .
\end{equation}
The remaining terms are estimated in a similar way leading to \eqref{e:est5}.
\end{proof}

Note that \eqref{e:est4} gives us control of the $\| \bar{\omega}_3 \|_{L^1}$ term in $\Psi(t)$ 
provided we can control the evolution of $\tilde{u}$ and $\Delta \tilde{u}$ and these are in turn
controlled by the evolution of $r$ and $\lambda$.  The evolution of $\lambda$ is controlled by
the estimates of the previous section, and thus, we turn our attention to the ``reduced'' functional
\begin{equation}
\Phi(t) = \| \bar{\omega}_3 \|_{L^2}^2 + \| \nabla r \|_{L^2}^2 \ .
\end{equation}
\begin{Remark}\label{r:rvsw} Note that we expect the two terms in $\Phi$ to have different properties - 
$\bar{\omega}_3$ may be large, but is not expected to grow much, while we can make $\nabla r$
arbitrarily small (at least initially) by choosing $R$ (in \eqref{e:PRdef}) sufficiently large.
\end{Remark}
Differentiating with respect to $t$ and using the estimates of Proposition \ref{p:en}, we obtain
\begin{align}\label{e:p1}
\fr{d}{dt}\Phi(t) &\leq  - \lp(  ||\nabla\bar\omega_3(t)||^2_{L^2} + ||\Delta r(t)||^2_{L^2}  \rp)
+ C ||\,| \tl u(t)| \,|\nabla \tl v(t)|||_{L^2}^2\notag\\
&\quad\quad +C_1\lp( ||\nabla r(t)||^2_{L^2} \,||\nabla\bar u(t)||_{L^2}^2 \, ||\Delta \bar u(t)||^2_{L^2}+||\bar u(t)||^2_{L^4}||\nabla \lambda(t)||_{L^4}^2 + ||\nabla \bar u(t)||_{L^2}^2||\lambda(t)||_{L^\infty}^2 \rp)
\end{align}
We bound the term
\begin{eqnarray}
\| | \tilde{u} | |\nabla \tilde{v} | \|^2_{L^2} \le C \left(\| \nabla r \|_{L^2}^3 \| \Delta r\|_{L^2}
+ \| \nabla r \|_{L^2}^2 ( \| \nabla \lambda \|_{L^{\infty}}^2 + \| \lambda \|_{L^{\infty}}^2 ) \right)
+ C  \| \nabla \lambda \|_{L^{2}}^2 \| \lambda \|_{L^{\infty}}^2 \ ,\label{e:bdunv}
\end{eqnarray}
while we bound $\bar{u}$ with the aid of the Biot-Savart law (see \cite[App. B]{gallay2002invariant})
\begin{equation}
\| \bar{u} \|_{L^4}^2 \le C \| \bar{\omega}_3 \|_{L^{4/3}}^2 \le C \| \bar{\omega}_3 \|_{L^1} \| \bar{\omega}_3 \|_{L^2}
\end{equation}
Recalling Remark \ref{r:rvsw} about the expected relative sizes of $\bar{\omega}_3$ and $r$, we see that \eqref{e:bdunv} implies
\begin{Lemma}\label{l:diffeq}
There exist constants $C_2,C_3,C_4$ such that the following holds.  Let $v\in C^0([0,T],X_\mathrm{sf})$ be a solution of \eqref{e:1} which is decomposed $v = \bar v + \lambda + r$ as above for some $R>0$.  Assume as well that there exist $K\geq1$ and $\epsilon \in (0,1]$ such that the corresponding components satisfy
\beq\label{e:asbd}
||\bar \omega_3(t)||_{L^2(\R^2)} \leq K, \quad ||\nabla r(t)||_{L^2(\mD)}<\epsilon,
\eeq
for all $t\in [0,T]$.  Then
\begin{align}\label{e:diffeq1}
\fr{d}{dt} \Phi(t) \leq& -\lp(  ||\nabla \bar\omega_3(t)||^2_{L^2(\R^2)} + ||\Delta r(t)||^2_{L^2(\mD^2) } \rp)\notag\\
&\quad\quad + C_2\epsilon^2 K^2 ||\Delta \bar u(t)||^2_{L^2(\R^2)} + C_3 \epsilon^2||\Delta r(t)||^2_{L^2(\R^2)},\notag\\
&\quad\quad + \Phi(t) G(t) + F(t) \notag \\
& \le -  \frac{1}{2} \lp(  ||\nabla \bar\omega_3(t)||^2_{L^2(\R^2)} + ||\Delta r(t)||^2_{L^2(\mD^2) } \rp)
+  \Phi(t) G(t) + F(t)\ ,
\end{align}
for all $t\in(0,T]$.  Here
\begin{align}
F(t)&= C_4 ( ||\lambda(t)||_{L^\infty(\mD)}^2||\Delta\lambda(t)||^2_{L^2(\mD)}
+ \| \bar{\omega}_3 \|_{L^1}^2 \| \nabla \lambda \|_{L^4}^2),\notag\\
G(t)&= C_4(||\nabla \lambda(t)||_{L^\infty(\mD)}^2 + ||\lambda(t)||_{L^\infty(\mD)}^2 + ||\nabla \lambda(t)||^2_{L^4(\mD)}).
\end{align}
\end{Lemma}

We now prove that with the aid of Gronwall's inequality, the $X_\mathrm{sf}$-norm of the solution of \eqref{e:1} remains uniformly bounded for all time, and hence, the local existence theorem can be extended without limit.  More precisely, we show
\begin{Proposition} For any initial conditions, $v_0 \in X_\mathrm{sf}$ of \eqref{e:1}, there exists $K, \Omega_0 >0$ and 
$\epsilon \in (0,1)$, such that if $|\Omega| > \Omega_0$, and if the projection onto the geostrophic
mode, $S \tilde{v}_0$, is sufficiently small, then for any $T>0$, 
\begin{eqnarray} \label{e:omegaest}
& \sup_{0 \le t \le T}\{ \| \bar{\omega}_3(t) \|_{L^1} + \| \bar{\omega}_3(t) \|_{L^2}\} \le K \\ \label{e:rest}
& \sup_{0\le t \le T} \| \nabla r(t) \|_{L^2}  \le  \epsilon\ .
\end{eqnarray}
\end{Proposition}

\begin{Remark} Note that Theorem \ref{t:gle} follows immediately from this proposition and the estimates
on $\lambda$ in Section \ref{s:disp}.
\end{Remark}

\begin{proof} Choose $K$ and $R$ large enough that
\begin{equation}
\| \bar{\omega}_3 |_{t=0} \|_{L^1} + \| \bar{\omega}_3 |_{t=0} \|_{L^2} \le \frac{1}{16} K\ ,
\end{equation}
and
\begin{equation}
\| \nabla r|_{t=0} \|_{L^2} \le \frac{1}{16} \epsilon\ .
\end{equation}
By the local existence theorem, there exists $T^{**} > 0$ such that \eqref{e:omegaest} and \eqref{e:rest}
hold for $0 \le t \le T^{**}$.  Let $T^*$ be the supremum over the set of values of $T$ for which these
estimates hold.  We claim that $T^* = \infty$.

Suppose instead that $T^* < \infty$.  Note that from the estimates of Section \ref{s:disp}, if we choose
$\Omega_0 $ sufficiently large, and geostrophic projection sufficiently small, we can insure that
$$
\re^{\int_0^T G(s) ds} \le 2\ ,\qquad \int_0^T F(s) ds<\fr{1}{16}K,
$$
for any $T > 0$.  Furthermore, using the estimates on $\lambda$ from the previous section, plus \eqref{e:rest} and \eqref{e:est4}, we can insure that both 
$$
\| \bar{\omega}_3(t) \|_{L^1}  \le 2 \| \bar{\omega}_3 |_{t=0} \|_{L^1} \le \frac{1}{8} K\ .
$$
for all $0 \le t \le T^*$.  If one then applies Gronwall's inequality to \eqref{e:diffeq1}, one finds
that 
\begin{equation}
\Phi(t) \le \frac{1}{2} K\ ,
\end{equation}
for all $0 \le t \le T^*$.  If we then apply a similar argument with Gronwall's inequality to 
$\| \nabla r(t) \|_{L^2}$, we also find that for  $\Omega_0$ sufficiently large, and the geostrophic 
projection sufficiently small, we have
\begin{equation}
\| \nabla r(t) \|_{L^2} \le \frac{1}{2} \epsilon\ ,
\end{equation}
for all $0 \le t \le T^*$.  However, these two estimates imply that we could extend the time for which
\eqref{e:omegaest}  and \eqref{e:rest} hold beyond $T^*$, contradicting its definition.  Hence, $T^* = \infty$
as desired.

\end{proof}

\subsection{Asymptotics}\label{s:asym}
\subsubsection{Exponential decay of $\tl v$}
\begin{Proposition}
For all $0<\mu<2\pi^2$, we have
\begin{align}
\sup_{t\geq 0} \re^{\mu t}||\nabla \tl v(t)||_{L^2} &<\infty, \\
\sup_{t\geq 1}\re^{\mu t}||\Delta \tl u(t)||_{L^2} &< \infty.
\end{align}
\end{Proposition}
\begin{proof}
This once again follows in a similar way as in \cite{gallay2009global}.  We already have by Poincare's inequality that $$||\lambda(t) ||_{H^s}\leq C \re^{-4\pi^2 t}.$$  Then using the global bound derived in Theorem \ref{t:gle}, along with dispersive estimates for $\lambda$, we find that \eqref{e:poin} and \eqref{e:est5} implies
\beq
\fr{d}{dt} ||\nabla r(t)||^2_{L^2(\mD)} + \fr{1}{2} ||\Delta r(t)||^2_{L^2(\mD)} \leq C ||\nabla r(t)||_{L^2(\mD)}^2||\Delta \bar u(t)||^2_{L^2(\R^2)} + C_2 \re^{-8\pi^2 t}.
\eeq
Setting $f(t) = \re^{\mu t} ||\nabla r(t)||^2_{L^2(\mD)}$ one obtains 
\beq
f'(t) \leq C f(t) ||\Delta \bar u(t)||^2_{L^2(\R^2)} + C_2 \re^{-(8\pi^2 - \mu)t}.
\eeq
Then pairing this with the fact that $\int_0^\infty ||\Delta \bar u(t)||^2_{L^2}dt <\infty$ (from the Gronwall estimate on $\Psi(t)$), we obtain that $f(t) \leq C_3$ and thus that 
\beq
||\nabla r(t)||_{L^2(\mD)}\leq C_3\re^{-\mu t/2}.
\eeq
From which we conclude the exponential decay of $||\tl v(t)||_{H^1}$ as $t\rightarrow \infty$ with any rate $0<\mu<2\pi^2$ as 
$$
||\tl v||_{H^1(\mD)} \sim ||\nabla \tl v||_{L^2(\R^2)} \leq ||\nabla r||_{L^2(\mD)}+||\nabla \lambda||_{L^2(\mD)}.
$$
We also remark that optimal decay rates can be found in a similar way by multiplying \eqref{e:2a} with $\Delta \tl v$ to find
\beq
\sup_{t\geq 0} \re^{\mu t}||\nabla \tl v(t)||_{L^2} <\infty, \quad \text{for all } \mu < 4\pi^2,
\eeq
and furthermore, by differentiating in time, that 
$$
\sup_{t\geq 1}\re^{\mu t}||\Delta \tl u(t)||_{L^2} < \infty,\quad \text{for all } \mu < 4\pi^2.
$$
\end{proof}

\subsubsection{Diffusive decay of $\bar \omega_3$.}

In this section we show that the solution $\bar \omega_3(t)$ of \eqref{e:2a} converges to Oseen's vortex as $t\rightarrow+\infty.$  This can be obtained using the approach in \cite[\S 3.4]{gallay2009global} and thus we only outline the argument. One first introduces scaling variables by defining
\begin{align}
\bar\omega_3(x_h,t) &= \fr{1}{1+t}\bar w_3(\fr{x_h}{\sqrt{1+t}},\log(1+t)),\notag\\
\bar u_h(x_h,t) &= \fr{1}{\sqrt{1+t}} \overline{\mathfrak{u}}_h(\fr{x_h}{\sqrt{1+t}},\log(1+t)),\notag\\
\xi &= \fr{x_h}{\sqrt{1+t}},\quad \tau = \log(1+t)).\notag
\end{align}
It then follows that $\bar w_3$ satisfies the equation
\beq\label{e:scvar}
\bar w_{3,\tau} = \mc{L} \bar w_3 - (\overline{\mathfrak{u}}_h\cdot \nabla_\zeta) \bar w_3 - \bar{\mathbf{N}},
\eeq
where $\mc{L}:=\Delta_\zeta +\fr{1}{2}(\xi\cdot\nabla_\xi) + 1$, $\bar{\mathbf{N}}(\xi,\tau) = \re^{2\tau} \bar N(v(\xi\re^{\tau/2},x_3,\re^\tau-1)),$ and $\bar N$ is defined in \eqref{e:2a}.   The exponential decay of $\tl v$ gives that that 
$$
\int_0^\infty ||\bar{\mathbf{N}}(\cdot,\tau)||_{L^1(\R^2)}^2d\tau <\infty.
$$ 
Solutions $\bar w_3(\tau)$ of \eqref{e:scvar} with initial data $\bar w_{3,0}$ in $L^1(\R^2)$ are thus globally defined for $\tau\geq0$ satisfying
$$
\bar w_3\in C^0([0,\infty),L^1(\R^2)), \quad ||\bar w_3(\tau)||_{L^1(\R^2)}\leq C,
$$
for all $\tau\geq0.$ Asymptotics for the solution $\bar w_3(\tau)$ are then obtained by explicitly characterizing its omega-limit set 
$
\Omega_\infty:= \{ w_\infty\in L^1(\R^2)\,|\, \exists \tau_n\rightarrow\infty,\,\, \bar w_3(\tau_n)\rightarrow w_\infty\}.
$
 As in \cite[Lem. 3.1]{gallay2009global}, one uses Duhamel's formula to express the solution $\bar w_3(\tau)$ of \eqref{e:scvar} in terms of the explicit semi-flow, $\mathbf{\Phi}(\tau)$, associated with the purely barotropic equation
$$
\bar w_{3,\tau} = \mc{L} \bar w_3 - \overline{\mathfrak{u}}_h\cdot\nabla \bar w_3
$$
 which one can use to prove that the trajectory $\{\bar w_3(\tau)\}_{\tau\geq0}$ is relatively compact in $L^1(\R^2)$.  Then, as in \cite[Lem 3.2]{gallay2009global}, one compares the solution $\bar w_3(\tau)$ to that of the limiting equation, $\Phi(\tau) \bar w_{3,0}$, to obtain that $\Phi(\tau) \Omega_\infty = \Omega_\infty$ for all $\tau\geq0$.  Continuing as in the proof of \cite[Prop. 3.4]{gallay2005global} then implies the following result.
\begin{Proposition}
Let $A = \int_{\R^2} \bar w_{3,0}(\xi)d\xi = \int_{\R^2} \bar\omega(x_h,0)dx_h$, then the omega-limit set of a solution, $\bar w_3(\zeta,\tau)$, of \eqref{e:scvar} with initial condition $\bar w_3(\xi,0)  =\bar w_{3,0}(\xi)$ in $L^1(\R^2)$ satisfies 
$
\Omega_\infty = \{ A G \}.
$
\end{Proposition} 
It follows from this proposition that $\| \bar w_3 (\tau) -  A G \|_{L^1(\R^2)} \rightarrow 0$ as $\tau \rightarrow \infty$.  Translating this result into unscaled variables one can obtain leading order asymptotics for the solution $\bar\omega(x_h,t)$ (and thus $\bar u_h(x_h,t)$) to complete the proof of Theorem \ref{t:0}.

\appendix

\section{Estimates on the linear evolution}\label{a:sg}
\subsection{Barotropic evolution}\label{ss:Abt}
In this section we build, in a step-by-step fashion, the solution operators for the linear equations $\bar W_\tau = (\bar L(\tau) + \bar \Lambda) \bar W,\, \bar W(\sigma) = \bar W_0,\,$ defined for $\tau\geq \sigma$ in \eqref{e:vsc5}, \eqref{e:linv1}, and \eqref{e:pertlin1} above.   The barotropic linear evolution can be characterized and estimated more easily in the rotating frame as discussed in Section \ref{ss:btdyn}. That is, converting the linear part of \eqref{e:pbdec_red} into the rotating frame with coordinates
\begin{align}
\left(\begin{array}{c}\underline{\omega}_h \\ \underline{\Theta}\end{array}\right) := \re^{\Gamma tJ_2} \left(\begin{array}{c}\overline{\omega}_h \\ \overline{\Theta}\end{array}\right)
,\qquad J_2 = \left(\begin{array}{cc}0 & I _2\\-I_2 & 0\end{array}\right),\label{e:tJ}
\end{align}
one obtains
\begin{align}
(\underline{\omega}_h)_t = \Delta \underline{\omega}_h\notag\\
(\underline{\omega}_3)_t = \Delta \underline{\omega}_3\notag\\
\underline{\Theta}_t = \Delta\underline{\Theta}.
\end{align}
Then moving into scaling variables
$$
\underline{\omega}(x_h,t) = (1+t)^{-1} \underline{w}(x(1+t)^{-1/2},\log(1+t)),\quad
\ul{\Theta}(x_h,t) = (1+t)^{-1} \underline{\Phi}(x(1+t)^{-1/2},\log(1+t)),
$$
we obtain a linear system for $\ul{W}:= (\underline{w}, \underline{\Phi})^T,$
\begin{align}\label{e:rtsc}
\underline{W}_\tau = \mc{L} \, \underline{W}
\end{align}
which generates the strongly continuous semigroup, $\underline{S}_0(\tau):=\re^{\mc{L} t}$ on $L^2_{2D}(m)^5$.  Since the linear operator is diagonal, we can study spectral properties component-wise.  First we have that any $\lambda$ in the point spectrum satisfies $\rre\,\lambda\leq 0$.  Due to the incompressibility conditions, $\nabla_h \cdot \ul{\Phi} = \nabla_h\cdot\ul{w}_h$, we have that the $0$-eigenspace is spanned by the vector $(0,0,\varphi_0,0,0)^T$ and the $(-1/2)$-eigenspace is spanned by the vectors $(\p_2\varphi_0,-\p_1\varphi_0,0,0,0)^T$, $(0,0,0,\p_2\varphi_0,-\p_1\varphi_0)^T$, and $(0,0,\p_i\varphi_0,0,0)^T$ for $i = 1,2$.  

Since we must study different components of the evolution we define the following projection operators
$$
e_i^\perp = 1 - e_ie_i^T, i = 1,...,5,\qquad e_h^T = e_1e_1^T + e_2e_2^T,\quad e_h^\perp = 1 - e_h^T,\quad
e_p^T = e_4e_4^T + e_5e_5^T,\quad e_p^\perp = 1 - e_p^T,
$$
where $e_i$ are the standard basis vectors in $\R^5$.  Also recall that $b = (1+|\xi|^2)^{1/2}$. Next, recall that
$$
Y = Y_{h,1} \times Y_{3,0}\times Y_{p,1},\quad Y_{h,n} = L^2_{2D,n}(m)^2\cap\{ \nabla_h\cdot \underline{w}_h = 0\},\, Y_{3 ,n}=  L^2_{2D,n}(m),\,\, Y_{p,n} =  L^2_{2D,n}(m)^2\cap\{\nabla_h\cdot \ul{\Phi} = 0\}
$$
with the standard $L^2_{2D}(m)$ norm.  We first state asymptotics for the strongly continuous semi-group $\re^{\tau\mc{L}}$ generated by the operator $\mc{L}$ on $L^2_{2D}(m)$.

 \begin{Proposition}\label{p:sgest2}
Let $f\in L^2_{2D}(m)$, $m>1,q\in [1,2], \alpha\in \mb{N}^2$, then there exists $C>0$ constant such that 
\beq
|| b^m \p^\alpha \re^{\tau \mc{L}} Q_n f||_{L^2(\R^2)},  \leq C\fr{\re^{-\gamma\tau}}{a(\tau)^{1/q-1/2+|\alpha|/2}}||b^m f||_{L^q(\R^2)}\label{e:sgest2}
\eeq
where $a(\tau) = 1-\re^{-\tau}$ and 
$$
\gamma = \begin{cases}
\fr{m-1-\epsilon}{2}, &\text{if}\quad n+1<m\leq n+2\\
\fr{n+1}{2},\quad &\text{if}\quad m>n+2.
\end{cases}$$
\end{Proposition}
\begin{proof}
See Proposition 4.2.2 of \cite{roussier2003stabilite} or Propositions A.2  and A.5 of \cite{gallay2002invariant}.
\end{proof}
One can then readily conclude the following temporal estimates on $\ul{\mc{S}}_0$.
\begin{Proposition}
Let $f\in Y$, $m>3, q\in [1,2], \alpha\in \mb{N}^2.$ If in addition $b^m f\in L^q$ then there exists a constant $C>0$ such that 
\begin{align}
\| e_h^T\lp(\p^\alpha \ul{\mc{S}}_0(\tau) f\rp)\|_m &\leq C \fr{\re^{-\tau}}{a(\tau)^{1/q - 1/2+|\alpha|/2}}\|b^m e_h^T f\|_{L^q},\label{e:S0est1}\\
\| e_3^T\lp(\p^\alpha \ul{\mc{S}}_0(\tau) f\rp)\|_m &\leq C \fr{\re^{-\tau/2}}{a(\tau)^{1/q - 1/2+|\alpha|/2}} \| b^m e_3^T f\|_{L^q}.\label{e:S0est2}\\
\| e_p^T\lp(\p^\alpha \ul{\mc{S}}_0(\tau) f\rp)\|_m &\leq C \fr{\re^{-\tau}}{a(\tau)^{1/q - 1/2+|\alpha|/2}} \| b^m e_p^T f\|_{L^q}.\label{e:S0est4}
\end{align}
\end{Proposition}

Transforming the nonlinear equation \eqref{eq:subsystem} on the barotropic subspace $\{\tl v = 0\}$ into scaling variables, we obtain an equation in the rotating frame for the $(\underline{w},\underline{\Phi})$.
\begin{align}
\underline{w}_\tau&= \mc{L} \underline{w} - \underline{\mfu}_h\cdot \nabla \underline{w} + \underline{w}_h\cdot \nabla \left(\begin{array}{c}\underline{\mfu}_h \\ 0\end{array}\right)    ,\notag\\
\underline{\Phi}_\tau &= \mc{L} \underline{\phi} - \underline{\mfu}_h\cdot \nabla \underline{\Phi} + \underline{\Phi}\cdot\nabla\underline{\mfu}_h. \label{e:brf}
\end{align}
We use the ansatz
$$
\underline{w} = \underline{w}_{app} + \underline{w}_R,\quad \underline{\Phi} = \underline{\Phi}_{app} + \underline{\Phi}_R,\quad \underline{\mfu} = \underline{\mfu}_{app} + \underline{\mfu}_R,
$$
with the approximate solution 
$$
\underline{w}_{app} = (0,0,A\varphi_0)^T + B_1\re^{-\tau/2}(\p_2 \varphi_0,-\p_1\varphi_0,0)^T,\quad
\underline{\Phi}_{app} = B_2\re^{-\tau/2} (\p_2\varphi_0, -\p_1\varphi_0)^T 
\quad \underline{\mfu}_{app} = (A\underline{\mfu}_{h}^0,B_1 \re^{-\tau/2}\varphi_0)^T,
$$
where $\varphi_0$ is the Gaussian, $\underline{\mfu}_h^0$ is its corresponding velocity profile, and $\underline{w}_{R,h}, \underline{\Phi}_R \in L^2_{2D,1}(m)^2, \underline{w}_{R,3}\in L^2_{2D,0}(m).$ Note that $\underline{w}_{app}$ and $\ul{\Phi}_{app}$ just correspond to the explicit vortex solution discussed in Section \ref{ss:fulldyn}.  Alternatively, it can be obtained from the stationary frame approximate solution by re-writing \eqref{e:approxsol} in the rotating frame. Using the fact that the approximate solution solves \eqref{e:brf} we find that the perturbation of the linear evolution satisfies
\begin{align}
(\underline{W}_R)_\tau = \mc{L}\, \underline{W}_R + \underline{\Lambda}\, \underline{W}_R\label{e:aunsc}
\end{align}
where 
$$
\underline{\Lambda}\, \underline{W}_R:= -\underline{\mfu}_{app,h}\cdot\nabla \underline{W}_R - \underline{\mfu}_{R,h}\cdot \nabla \underline{W}_{app} + \left(\begin{array}{c} \underline{w}_{app,h}\cdot\nabla \underline{\mfu}_{R,h} \\0\\ \underline{\Phi}_{app}\cdot\nabla  \underline{\mfu}_{R,h}\end{array}\right)   +  \left(\begin{array}{c}\underline{w}_{R,h}\cdot\nabla\underline{\mfu}_{app,h} \\0\\\underline{\Phi}_{R}\cdot\nabla\underline{\mfu}_{app,h}\end{array}\right).
$$
We can then split $\underline{\Lambda}$ into three parts,
\begin{align}
\ul{\Lambda} &= \ul{\Lambda_0} + \ul{\Lambda_1} + \ul{\Lambda_2},\notag\\
\ul{\Lambda_0} \underline{W}_R &= 
-\lp[\underline{\mfu}_{app,h}\cdot\nabla \left(\begin{array}{c}0 \\ \ul{w}_{R,3} \\0\end{array}\right) + \underline{\mfu}_{R,h} \cdot\nabla\left(\begin{array}{c}0 \\\ul{ w}_{app,3} \\0\end{array}\right)
\rp]\notag\\
\ul{\Lambda_1} \underline{W}_R &= 
\left(\begin{array}{c}\ul{w}_{R,h}\cdot\nabla \underline{\mfu}_{app,h}\\0 \\ \ul{\Phi}_{R}\cdot\nabla \underline{\mfu}_{app,h}\end{array}\right)  - \underline{\mfu}_{app,h} \cdot\nabla\left(\begin{array}{c}\ul{w}_{R,h} \\0 \\ \Phi_{R}\end{array}\right)\notag\\
\ul{\Lambda_2} \underline{W}_R &= 
\left(\begin{array}{c}\ul{w}_{app,h}\cdot\nabla \underline{\mfu}_{R,h}\\0 \\ \ul{\Phi}_{app}\cdot\nabla \underline{\mfu}_{R,h}\end{array}\right)  - \underline{\mfu}_{R,h} \cdot\nabla\left(\begin{array}{c}\ul{w}_{app,h} \\0 \\ \ul{\Phi}_{app}\end{array}\right),\notag
\end{align}
where $\ul{\Lambda_0}$ is exactly the same as in the 2-D Navier Stokes case discussed in \cite{gallay2005global}, and $\ul{\Lambda_2}$ consists entirely of terms decaying like $\re^{-\tau/2}$. For the rest of the section we shall omit the $R$ subscript.

\subsubsection{Evolution generated by $\ul{L} + \ul{\Lambda_0} + \ul{\Lambda_1}$}

We then use this decomposition to obtain the following proposition which gives temporal decay estimates on the linear evolution generated by $\underline{L} + \underline{\Lambda}_0 + \underline{\Lambda}_1$ defined on $Y$. Recall that $a(\tau) = 1-\re^{-\tau}$.
\begin{Proposition}\label{p:aT1}
The linear operator $\underline{L}+ \underline{\Lambda_0} + \underline{\Lambda_1}$ defines a strongly continuous semi-group $\ul{\mc{T}}_1(\tau)$ on $Y$ which, for $f\in Y$, $|\alpha|\leq1$, $q\in [1,2],$ $m>3$, satisfies the following estimates,
\begin{align}
\|e_h^T\lp(\p^\alpha \ul{\mc{T}}_1(\tau) f\rp)\|_m &\leq \fr{\re^{-\tau}}{a(\tau)^{1/q - 1/2+|\alpha|/2}}\|b^m e_h^T f\|_{L^q},\qquad
\label{T1est1c} \\
\| e_3^T\lp(\p^\alpha \ul{\mc{T}}_1(\tau) f\rp)\|_m &\leq \fr{\re^{-\tau/2}}{a(\tau)^{1/q - 1/2+|\alpha|/2}}\|b^m e_3^Tf\|_{L^q},\qquad
 \label{T1est2c}\\
\|e_p^T\lp(\p^\alpha \ul{\mc{T}}_1(\tau) f\rp)\|_{m} &\leq \fr{\re^{-\tau}}{a(\tau)^{1/q - 1/2+|\alpha|/2}}\|b^{m} e_p^T f\|_{L^q},\qquad
\label{T1est3c} 
\end{align}
\end{Proposition}

We shall split the proof of this result into a series of results. We begin by using Duhamel's formula to write solutions of the evolution equation
\begin{equation}
(W_R)_{\tau} =  \mc{L}\, \underline{W}_R +( \ul{\Lambda}_0 + \ul{\Lambda}_1)\underline{W}_R\ ,
\end{equation}
as
\begin{equation}
\ul{\mc{T}}_1(\tau) \underline{W}_0 = \ul{\mc{S}}_0 (\tau) \underline{W}_0
+ \int_0^{\tau} \ul{\mc{S}}_0(\tau - \sigma) \Lambda_0 \underline{W}_R(\sigma) d\sigma
+ \int_0^{\tau} \ul{\mc{S}}_0(\tau - \sigma) \Lambda_1 \underline{W}_R(\sigma) d\sigma\ .
\end{equation}
Estimating the integral terms as we did the corresponding expressions in Section \ref{s:alg} one 
readily shows that $\ul{\mc{T}}_1(\tau)$ defines a strongly continuous semigroup such that
for any $T>0$, there exists $C_T>0$ such that for any $0 \le \tau \le T$,  and multi-index
$\alpha$ with $|\alpha | =1$, 
\begin{equation}\label{e:existence}
\| \ul{\mc{T}}_1(\tau) \underline{W}_0 \|_Y \le C_T \| \underline{W}_0 \|_Y\ ,\ \ 
\| \partial^{\alpha} \left(  \ul{\mc{T}}_1(\tau) \underline{W}_0 \right)  \|_Y \le \frac{C_T}{a(\tau)^{1/2} } \| \underline{W}_0 \|_Y.
\end{equation}
The more refined estimates of the evolution given in Proposition \ref{p:aT1} will follow from a more detailed analysis of the semigroup.

We begin the estimates of $\ul{\mc{T}}_1(\tau)$ by noting that because of the form of $\ul{\Lambda}_0$ and $\ul{\Lambda}_1$,  it splits into three independent parts which govern the evolution of 
$ \underline{W}_{R,h} = (\underline{W}_{R,1},\underline{W}_{R,2})^T$, $\underline{W}_{R,3}$ and $\underline{W}_{R,p}  = (\underline{W}_{R,4},\underline{W}_{R,5})^T$.  We refer to these three parts of $\ul{\mc{T}}_1$ as
$\ul{\mc{T}}_{1,h}$, $\ul{\mc{T}}_{1,3}$ and $\ul{\mc{T}}_{1,p}$ and will estimate each of them separately.  For later use, we will refer to the pieces of $\mc{L} +( \ul{\Lambda}_0 + \ul{\Lambda}_1)$ which generate each of these semigroups as $\left( \mc{L} +( \ul{\Lambda}_0 + \ul{\Lambda}_1)\right)_h$, and so forth. 

As a first step, note that $\ul{\mc{T}}_{1,3}(\tau)$ is exactly the semigroup studied in Section 4
of  \cite{gallay2005global}, as it is the linearization of the two-dimensional Navier-Stokes equation.
Thus, estimates \eqref{T1est2c} follow from the results of Proposition 4.13 of that work.

To prove the estimates in \eqref{T1est1c} and \eqref{T1est3c} we use methods similar to those in Section 4 of  \cite{gallay2005global}.  We will give the  details for \eqref{T1est1c}, as the case of \eqref{T1est3c} follows in the exact same way. The proof of \eqref{T1est1c} begins by noting that these estimates are exactly those we would obtain
for the horizontal components of the diagonal semigroup $\ul{\mc{S}}_0(\tau)$ (which we will denote
$\ul{\mc{S}}_{0,h}$.)  Thus, we proceed in two steps:
\begin{enumerate}
\item We first show that the difference between $\ul{\mc{S}}_{0,h}(\tau) $ and $\ul{\mc{T}}_{1,h}(\tau)$ is
a compact operator for any $\tau > 0$.  As a consequence, the essential spectral radius of
  $\ul{\mc{T}}_{1,h}(\tau)$ is the same as that of $\ul{\mc{S}}_{0,h}(\tau) $.
 \item The first point implies that the only way that the decay rates of the two semigroups can differ
 is if $\ul{\mc{T}}_{1,h}(\tau)$ has an isolated eigenvalue lying outside the disc which
 gives the essential spectral radius.  Thus, the second step in our analysis is to determine the
 location of isolated eigenvalues of $ \left( \underline{L}\, + \Lambda_0 + \Lambda_1\right)_h$.
 \end{enumerate}

 Thus we shall prove the following two results
\begin{Proposition}\label{p:compact}
 Assume the hypotheses of Proposition \ref{p:aT1}.  Then for any $\tau > 0$, the linear operator ${\mathcal{K}}_h(\tau)  = \ul{\mc{T}}_{1,h}(\tau)
 - \ul{\mc{S}}_{0,h}(\tau)$ is compact.
 \end{Proposition}
\begin{Proposition}\label{p:eigen} Assume that $m > 3$.  Then any eigenvalue, $\lambda$,  of  $ \left( \mc{L}\,  + \ul{\Lambda_0} + \ul{\Lambda_1}\right)_h$ on $Y_{h,1}$ satisfies
\begin{equation}
\Re(\lambda) \le -1.
\end{equation}
\end{Proposition}
Combining these two propositions, the estimates \eqref{T1est1c} follow exactly as the proof of
Proposition 4.13 in \cite{gallay2005global}.    The proof of Proposition \ref{p:compact} follows immediately from Rellich's criterion and the following lemma.
\begin{Lemma}
Assume the hypotheses of Proposition \ref{p:aT1}.  Then for any $T>0$, there
exists a constant $C_T>0$ such that for $0 \le \tau \le T$, 
for any $\tau > 0$, the linear operator ${\mathcal{K}}_h(\tau)  = \ul{\mc{T}}_{1,h}(\tau)
 - \ul{\mc{S}}_{0,h}(\tau)$ satisfies
 \begin{equation}
\| {\mathcal{K}}_h(\tau) \underline{W}_{h} \|_{m+1} \le C_T \|  \underline{W}_{h} \|_{m}\ ,\ \ 
\| \partial_j {\mathcal{K}}_h(\tau) \underline{W}_{h} \|_{m} \le \frac{C_T}{a(\tau)^{1/2}}
 \|  \underline{W}_{h} \|_{m}\ ,\ \ {\mathrm{for}} \ j=1,2\ .
\end{equation}
\end{Lemma}
\begin{proof}
From Duhamel's formula we can write
\begin{equation}\label{e:horizontalHamel}
{\mathcal{K}}_h(\tau) \underline{W}_{h,0} = -\int_0^{\tau} \ul{\mc{S}}_{0,h}(\tau-\sigma) (\underline{W}_h(\sigma)
\cdot \nabla) \underline{\mfu}_{h}^0 d\sigma - \int_0^{\tau} \ul{\mc{S}}_{0,h}(\tau-\sigma) (\underline{\mfu}_{h}^0 \cdot
\nabla) \underline{W}_h(\sigma) d\sigma\ .
\end{equation}
The estimates of the lemma now follow by estimating each of the integral terms.  We first have
\begin{eqnarray}
&& \| \int_0^{\tau} \ul{\mc{S}}_{0,h}(\tau-\sigma) (\underline{\mfu}_{h}^0 \cdot
\nabla) \underline{W}_h(\sigma) \|_{m+1} \le C \int_0^{\tau} e^{- (\tau - \sigma)}
\|  (\underline{\mfu}_{h}^0 \cdot
\nabla) \underline{W}_h(\sigma) \|_{m+1} d\sigma\\
&& \qquad \qquad  \le  C \int_0^{\tau} e^{- (\tau - \sigma)}
\| 
\nabla \underline{W}_h(\sigma) \|_{m} d\sigma
 \le C C_T \int_0^{\tau} \frac{e^{- (\tau - \sigma)}}{a(\sigma)^{1/2}} 
d\sigma \| \underline{W}_{h,0}\|_{m} \\ && \qquad \qquad \le C \| \underline{W}_{h,0}\|_{m} \ .
\end{eqnarray}
Here, the first inequality used our estimates on the semigroup $S_0(\tau)$, the second inequality used the fact that
$\| \underline{\mfu}_{h}^0 f\|_{m+1} \le C \| f \|_m$ due to the decay of $\underline{\mfu}_{h}^0(\xi)$ as $|\xi| \to \infty$,  the third inequality used \eqref{e:existence}, and the
last one the fact that the singularity in $a(\tau)^{-1/2}$ is integrable.

Estimating the $\| \cdot \|_{m+1}$-norm of the first integral term in \eqref{e:horizontalHamel} is
similar, but even easier since derivatives of $\underline{\mfu}_{h}^0(\xi)$ decay more rapidly as $|\xi| \to \infty$ than $\underline{\mfu}_{h}^0(\xi)$ itself, and we leave this estimate to the reader. Now we turn to the second estimate in the lemma, namely the estimate of the derivatives of 
${\mathcal{K}}_h$.  Differentiating the expression in \eqref{e:horizontalHamel} we must estimate
\begin{equation}
\partial_j \left({\mathcal{K}}_h(\tau) \underline{W}_{h,0}\right) =
 -\int_0^{\tau} \partial_j \left(\ul{\mc{S}}_{0,h}(\tau-\sigma) (\underline{W}_j(\sigma)
\cdot \nabla) \underline{\mfu}_{h}^0 \right) d\sigma
 - \int_0^{\tau} \partial_j \left(  \ul{\mc{S}}_{0,h}(\tau-\sigma) (\underline{\mfu}_{h}^0 \cdot
\nabla) \underline{W}_h(\sigma)\right) d\sigma\ .
\end{equation}

Once again, because of the explicit formulas for the derivatives of $\underline{\mfu}_{h}^0(\xi)$, the estimates of the first term are easier than the second, so we focus on estimating the second.
Using the estimates on the semi-group $\ul{\mc{S}}_{0,h}$, plus the estimates in \eqref{e:existence},
we have
\begin{equation}
\|  \int_0^{\tau} \p_j\ul{\mc{S}}_{0,h}(\tau-\sigma) (\underline{\mfu}_{h}^0 \cdot
\nabla) \underline{W}_h(\sigma) \|_{m} \le C C_T \left( \int_0^{\tau} \frac{e^{-(\tau - \sigma)} }{
a(\tau-\sigma)^{1/2} a(\sigma)^{1/2} } d\sigma \right)  \| \underline{W}_{h,0}\|_{m} \ .
\end{equation}
Now, using the fact that
\begin{equation}
\int_0^{\tau} \frac{e^{- (\tau - \sigma)} }{
a(\tau-\sigma)^{1/2} a(\sigma)^{1/2} } d\sigma \le \frac{C}{a(\tau)^{1/2}} \ ,
\end{equation}
the estimate, and the lemma follows.
\end{proof}

We now turn to the proof of Proposition \ref{p:eigen}.  We begin with:
\begin{Lemma}\label{l:aLh} $\lambda$ is an eigenvalue of the operator
$$
L_h \ul{w}_{h}:= \mc{L} \ul{w}_h + \underline{w}_{h}\cdot\nabla \underline{\mfu}_{h}^0 - \underline{\mfu}_{h}^0\cdot\nabla \underline{w}_{h},
$$
posed on $Y_{h,1}$, if and only if $\lambda$ is an eigenvalue of the operator
$$
L_3 \underline{u}_3 := (\mc{L} - 1/2)  \underline{u}_3 - A\underline{\mfu}_{h}^0 \cdot\nabla \underline{u}_3 
$$
on $L^2_{2D,0}(m')$, where $\underline{w}_{R,h} = (\p_2 \underline{u}_3,-\p_1\underline{u}_3)^T$ and $0<m'\leq m -1$. 
\end{Lemma}

\begin{Remark} Note that our estimates of the Biot-Savart kernel imply that the velocity field
$\underline{u}_3 \in L^2_{2D,0}(m')$.
\end{Remark}
\begin{proof}
The proof of the lemma follows by noting that if $W$ and $\lambda$ satisfy
$$
L_h W = \lambda W,
$$
then we set $U = \underline{\mfu}_{3,BS}(W)$ where $\overline{\mfu}_{3,BS}$ is the Biot-Savart mapping \eqref{e:BS3} defined above so that $(\p_2 U, -\p_1U)= W$, and $U\in L^2_{2D,0}(m')$. Indeed our choice of $m,m'$, Proposition B.1 of \cite{gallay2002invariant}, and the generalized Holder inequality gives
 \begin{align}
 \| \overline{\mfu}_{3,BS}(\bar w_h)\|_{m'} &\leq \| b^{m' -m''}\|_{L^p} \|b^{m''} \overline{\mfu}_{3,BS}(\bar w_h)\|_{L^{q'}},\notag\\
 &\leq C \| b^m \bar w_h\|_{L^2},\label{e:BS3a}
 \end{align}
 where $1/2 = 1/p + 1/q'$ with $q'\in (2,\infty)$, chosen so that $m' - m''> 2/p$ (and hence $b^{m'-m''}\in L^p$) and $m'' = m- 2/q',$ (so \cite[Prop. B.1]{gallay2002invariant} holds). Now, we compute the derivatives of $h := L_3 U$ to find
$$
(\p_2 h,-\p_1 h)^T = L_h W = \lambda W.
$$
Applying the Biot-Savart law we then find by linearity that $h = \lambda U$ and thus
$$
L_3 U = \lambda U.
$$
\end{proof}
Proposition \ref{p:eigen} now follows from 
\begin{Lemma}  Let $m'$ be as in Lemma \ref{l:aLh} .  Then any eigenvalue, $\lambda$,
 of the operator
$L_3$ whose eigenfunction lies in the space $L^2_{2D,0}(m')$, satisfies
\begin{equation}
\Re(\lambda) \le - 1 \ .
\end{equation}
\end{Lemma}
\begin{proof}
To prove this lemma, we first note that any eigenfunction of $L_3$ decays as a Gaussian as
$|\xi | \to \infty$.  This is proven in a fashion almost identical to the analogous result in \cite{gallay2005global}.  More precisely, if we write the eigenfunction $\psi$ in polar coordinates we find
that $\psi = \Psi(r) e^{i \ell \theta}$ where $\Psi$ satisfies the ODE
\begin{equation}
\Psi''(r) + \left( \frac{r}{2} + \frac{1}{r} \right) \Psi'(r) + (1-\lambda - \frac{\ell^2}{r^2} - \ri \ell
A \eta(r)) \Psi(r) =  0\ ,
\end{equation}
where $\eta(r) = \frac{1}{2 \pi r^2}(1-e^{-r^2/4})$.  This is almost identical to equation (70) of
\cite{gallay2005global} (in fact, significantly simpler than that case, because the nonlocal
term denoted $\Omega(r)$ in that reference is absent) and using the same sort of ODE estimates
applied there one finds that there exists $\gamma > 0$ such that
\begin{equation}
| \Psi(r)| \le C (1+ r^2)^{\gamma} e^{-r^2/4}\ .
\end{equation}

To localize the eigenvalues of $L_3$ we conjugate the two pieces of the operator with the
Gaussian $\varphi_0^{-1/2}$.  Alternatively, we study the operator on the weighted Hilbert
space
\begin{equation}
X_0 = \{ f\in L^2(\R^2)\,:\, \varphi_0^{-1/2} f \in L^2(\R^2), \int_{\R^2} f(\xi ) d\xi = 0 \},
\end{equation}
 From Lemma 4.7 of \cite{gallay2005global}, we know that 
 ${\mathcal{L}} - 1/2$ is a selfadjoint operator on this space with pure
 point spectrum $-1, -3/2, -2, \dots$.  A direct calculation shows that the operator
 $f \to \underline{\mfu}_{h}^0 \cdot\nabla f$ is anti-symmetric in this space.   Explicitly, we have
\begin{align}
(\tl f, \underline{\mfu}_{h}^0\cdot\nabla f)_X&:= \int_{\R^2} \varphi_0^{-1}\tl f\cdot (\underline{\mfu}_{h}^0\cdot\nabla f) d\xi\notag\\
&=-\int_{\R^2} \varphi_0^{-1} f\cdot (\underline{\mfu}_{h}^0\cdot\nabla \tl f) d\xi
\end{align}
where we have used the fact that $\mathrm{div}\,( \varphi_0^{-1}\underline{\mfu}_{h}^0)=0.$ But now we use the fact that if an anti-symmetric operator is added to a self-adjoint operator with pure point spectrum, the
real part of the spectrum of the resulting operator is less than the largest eigenvalue of the self-adjoint part - in this case $-1$.  (See the proof of Proposition 4.1).
in \cite{gallay2005global}.
\end{proof}

\subsubsection{Evolution generated by $\ul{L} + \ul{\Lambda}$}

We now can find and estimate solutions of the full linear equation \eqref{e:aunsc}, which being non-autonomous, generates an evolutionary family of operators $\ul{\mc{S}}(\tau,\sigma),\, \tau\geq \sigma\geq0$ (in the sense of Pazy \cite[Ch. 5]{pazy2012semigroups}) instead of a semi-group.

\begin{Proposition}\label{p:aSu}
The operator $\underline{L} +\ul{\Lambda}$ defines an evolutionary family of operators $\ul{\mc{S}}(\tau,\sigma)$ which satisfy the following temporal estimates for $q\in [1,2]$, $|\alpha|\leq1,$ $m>3$, $\mu\in(0,1/2)$, $f\in Y$, and $b^m f\in L^q(\R^2)$,
\begin{align}
\|e_3^\perp\lp(\p^\alpha\ul{\mc{S}}(\tau,\sigma)f\rp)\|_{m} &\leq \fr{\re^{-(1/2+\mu)(\tau - \sigma)}}{a(\tau - \sigma)^{1/q - 1/2+|\alpha|/2}}\|b^m  f\|_{L^q} ,\notag\\
\|e_3^T\lp(\p^\alpha\ul{\mc{S}}(\tau,\sigma)f\rp)\|_m &\leq \fr{\re^{-(\tau - \sigma)/2}}{a(\tau - \sigma)^{1/q - 1/2+|\alpha|/2}}\|b^me_3^T f\|_L^q.\qquad
\notag
\end{align}
\end{Proposition}
\begin{proof}
First, the existence and uniqueness of $\ul{\mc{S}}(\tau,\sigma)$ can be obtained using a fixed point argument with the mapping
$$
\ul{W}(\tau)\mapsto \ul{\mc{S}}_0(\tau - \sigma) f+ \int_\sigma^\tau \ul{\mc{S}}_0(\tau - s) \ul{\Lambda} \,\ul{W}(s)ds,
$$
where we have set $\ul{W}(\sigma) = f$. To obtain the temporal estimates we use a Gronwall type argument which uses the fact that the coefficients of $\ul{\Lambda}_2$ decay with rate $\re^{-\tau/2}.$ Due to the uniqueness of the above fixed point, $\ul{\mc{S}}(\tau,\sigma)$ can also be represented by the following fixed point formula
\begin{align}
\ul{W}(\tau) := \ul{\mc{S}}(\tau, \sigma) f &= \ul{\mc{T}}_1(\tau-\sigma) f + \int_\sigma^\tau \ul{\mc{T}}_1(\tau - s) \ul{\Lambda}_2\, \ul{W}(s)ds. 
\end{align}

We then wish to show that
\begin{align}
\Psi(\tau,\sigma) &= \re^{(\tau - \sigma)/2}\Bigg[\re^{\mu(\tau - \sigma)}\lp(\|e_3^\perp \ul{W}(\tau)\|_{m}  + a(\tau - \sigma)^{1/2}\|\nabla(e_3^\perp \ul{W}(\tau))\|_{m}\rp) \notag\\
&+  \| \ul{W}_3(\tau)\|_m + a(\tau - \sigma)^{1/2}\| \nabla ( \ul{W}_3(\tau))\|_m \Bigg]\notag\\
&= \re^{(\tau - \sigma)/2}\Bigg[\re^{\mu(\tau - \sigma)}\lp(\| \ul{W}_h(\tau)\|_{m}  + \| \ul{W}_p(\tau)\|_{m}+ a(\tau - \sigma)^{1/2}(\| \nabla \ul{W}_h(\tau)\|_{m}  + \|\nabla \ul{W}_p(\tau)\|_{m} ) \rp) \notag\\
&+  \|  \ul{W}_3(\tau)\|_m + a(\tau - \sigma)^{1/2}\| \nabla (\ul{W}_3(\tau))\|_m \Bigg]
\end{align}
is bounded for all $\tau\geq \sigma$. Here we weaken the temporal decay rate on the $\ul{W}_h$ and $\ul{W}_p$ components in order to be able to close the Gronwall argument for $\Psi.$ We use Proposition \ref{p:aT1} to obtain
\begin{align}
\| \underline{W}_h(\tau)\|_m &\leq C\re^{-(1/2+\mu)(\tau - \sigma)} \|e_h^T f \|_m \notag\\
&+ C\int_\sigma^\tau \re^{-(1/2+\mu)(\tau-s)}\Big[ (\| \ul{w}_{app,h}(s)\|_{m} + \|b^m \ul{\Phi}_{app,h}(s)\|_{m})\|\nabla\underline{\mfu}_{R,h}(s)\|_{L^2} \notag\\
&\quad+ \| \underline{\mfu}_{R,h}\|_{L^6}( \|\nabla \ul{w}_{app,h}(s)\|_m + \|\nabla \ul{\Phi_{app}}(s)\|_m)  \Big]ds\notag\\
&\leq C\re^{-(1/2+\mu)(\tau - \sigma)} \|e_h^T f  \|_m +C_{B}\int_\sigma^\tau \re^{-(\tau-s)} \re^{-s/2}\|\underline{W}_{3}(s)\|_{L^{3/2}} ds\notag\\
&\leq C\re^{-(1/2+\mu)(\tau - \sigma)} \|e_h^T f\|_m + C_B\int_\sigma^\tau \re^{-(\tau - s)}\re^{-s/2} \| \ul{W}_3(s)\|_m ds\notag\\
&\leq C\re^{-(1/2+\mu)(\tau - \sigma)} \|e_h^T f \|_m + C_B\int_\sigma^\tau \re^{-\tau }\re^{\sigma/2} \Psi(s,\sigma) ds.
\end{align}
Note in the second line we used the estimate $\|\nabla \underline{u}_{R,h}\|_{L^q} \leq C \| \underline{W}_{3}\|_{L^q}$ which holds for all $q\in [1,\infty]$ (see \cite{roussier2003stabilite}) as well as the Biot-Savart estimate \eqref{e:bsbt} to obtain $\| \underline{u}_{h}\|_{L^{6}}\leq C\|\underline{W}_{3}\|_{L^{3/2}}$, while in the third line we used the fact that $L^2_{2D}(m)\hookrightarrow L^q(\R^2)$ for all $q\in [1,2]$ to conclude $ \|\underline{W}_{3}\|_{L^{3/2}}\leq C\| \underline{W}_{3}\|_m$.  Also, $C_B>0$ is a constant dependent on the amplitudes $B_{i}$ for $i = 1,2$.  We can then conclude
\begin{align}
\re^{(1/2+\mu)(\tau - \sigma)} \|\ul{W}_h(\tau)\|_m \leq C \|e_h^T f\|_m + C_B\int_\sigma^\tau \re^{(\mu-1/2)s}\Psi(s,\sigma) ds.
\end{align}
Using the fact that the third component $\bar W_3$ is unaffected by $\ul{\Lambda}_2$, we similarly obtain
\begin{align}
\re^{(\tau-\sigma)/2}\| \ul{W}_3(\tau)\|_m &\leq C \|e_3^T f\|_m,\notag\\
\re^{(1/2+\mu)(\tau - \sigma)} \| \ul{W}_p(\tau)\|_{m} &\leq C \|e_p^T f\|_{m} +C_B\int_\sigma^\tau \re^{(\mu-1/2)s} \Psi(s,\sigma)ds,\notag
\end{align}
as well as the gradient estimates
\begin{align}
a(\tau - \sigma)^{1/2}\re^{(1/2+\mu)(\tau - \sigma)} \|\nabla(\ul{W}_h(\tau))\|_m &\leq C \| e_h^T f\|_m + C_B \int_\sigma^\tau \fr{\re^{(\mu - 1/2)s}a(\tau - \sigma)^{1/2}}{a(\tau - s)^{1/2}}\Psi(s,\sigma)ds,\notag\\
a(\tau - \sigma)^{1/2}\re^{(\tau - \sigma)/2} \|\nabla(\ul{W}_3(\tau))\|_m &\leq C \|e_3^T f\|_m,\notag\\
a(\tau - \sigma)^{1/2}\re^{(\tau - \sigma)/2} \|\nabla(\ul{W}_p(\tau))\|_{m} &\leq C \|e_4^T f\|_{m} 
+ C_B \int_\sigma^\tau \fr{\re^{(\mu - 1/2)s}a(\tau - \sigma)^{1/2}}{a(\tau - s)^{1/2}}\Psi(s,\sigma)ds.\notag\
\end{align}
Combining these together we obtain the inequality
\begin{align}
\Psi(\tau,\sigma) \leq C\|f\|_{m} + C_{A,B}\int_\sigma^\tau \fr{\re^{(\mu - 1/2)s}a(\tau - \sigma)^{1/2}}{a(\tau - s)^{1/2}}\Psi(s,\sigma)ds.
\end{align}
Then, applying Gronwall's inequality, we obtain that 
\begin{align}
\Psi(\tau,\sigma)&\leq C_{A,B}\|f\|_{m} \mathrm{exp}\lp( \int_\sigma^\tau \fr{\re^{(\mu - 1/2)s}a(\tau - \sigma)^{1/2}}{a(\tau - s)^{1/2}}ds \rp),\notag\\
&\leq C_{A,B}\re^{1/(1/2 - \mu)} \|f\|_{m},
\end{align}
which since $\mu<1/2$ implies that $\Psi(\tau,\sigma)$ is bounded for all $\tau\geq \sigma$. We thus obtain the temporal estimates for $q=2$ on each component of $W(\tau)$ by unraveling the definition of $\Psi$ and taking note of the invariance properties of $\underline{L} + \ul{\Lambda}$. A similar argument then gives the desired $L^q$ estimates.
\end{proof}

\subsubsection{Evolution generated by $\bar L + \bar \Lambda$}
We now translate the estimates of Proposition \ref{p:aSu} back into the stationary frame in order to conclude estimates on the linear flow  for $\bar L(\tau) + \bar \Lambda$.  

\begin{Proposition}\label{p:asgest}
Let $m>3$, $\alpha\in \N^2$ with $|\alpha|\leq1$, and $\bar W_0\in Y,$ then the operator $\bar L(\tau) + \bar \Lambda$ generates a evolutionary family of operators, $\bar S(\tau,\sigma)$ for $\tau\geq\sigma\geq0$ which satisfies the following decay estimates
\begin{align}
\|e_3^\perp\lp(\p^\alpha\bar{S}(\tau,\sigma)W_0\rp)\|_{m} &\leq \fr{\re^{-(1/2+\mu)(\tau - \sigma)}}{a(\tau - \sigma)^{1/q - 1/2+|\alpha|/2}}\|b^m W_0\|_{L^q},\notag\\
\|e_3^T\lp(\p^\alpha\bar{S}(\tau,\sigma)W_0\rp)\|_m &\leq \fr{\re^{-(\tau - \sigma)/2}}{a(\tau - \sigma)^{1/q - 1/2+|\alpha|/2}}\|b^m e_3^T W_0\|_{L^q},\notag
\end{align}
\end{Proposition}

\begin{proof}
We set $\bar W(\xi,\sigma) = W_0(\xi)$ as the initial condition of the linear evolution $\bar W(\xi,\tau) = \overline{\mc{S}}(\tau,\sigma) \bar W(\xi,\sigma)$ which solves $\bar W_\tau = (\bar L(\tau) + \bar\Lambda)\bar W$.  In order to use Proposition \ref{p:aSu}, we translate to the corresponding evolution in the rotating frame.  We first set $\ul{w}_3(\xi,\sigma) = W_{0,3}(\xi)$, Since it is untouched by the rotating frame transformation. For the other components, by translating back to unscaled variables, we set
\begin{align}
(\ul{w}_h,\ul{\Phi})(\xi,\sigma) &= (1+s)( \ul{\omega}_h,\ul{\Theta})(x,s)\notag\\
&=(1+s)\re^{-s\Gamma J_2}( \bar{\omega}_h,\bar{\Theta})(x,s)\notag\\
&= \re^{-(\re^\sigma - 1)\Gamma J_2}( \bar{w}_h,\bar{\Phi})(\xi,\sigma).\notag\\
&=\re^{-(\re^\sigma - 1)\Gamma J_2}( W_{0,h},W_{0,p})(\xi,\sigma),\notag
\end{align}
with $J_2$ as in \eqref{e:tJ}. We then let $\underline{W}(\tau):= \ul{\mc{S}}(\tau,\sigma) \ul{W}(\sigma)$, be the solution of \eqref{e:aunsc} with initial data $\underline{W}(\sigma)$ as defined above. Note $\overline{w}_3(\tau):=\underline{w}_3(\tau)$ as the third vorticity component is not affected by the rotating frame.    Next, defining $R(\tau)  = \re^{(\re^\tau - 1)\Gamma J_2}$, we have
\begin{align}
\overline{w}_h(\tau)&= \lp[R(\tau)\left(\begin{array}{c}\underline{w}_h(\tau) \\ \underline{\Phi}(\tau)\end{array}\right)\rp]_{1,2}\notag\\
&= c(\tau) \ul{w}_h(\tau) - s(\tau)\ul{\Phi}(\tau)\notag\\
\overline{\Phi}(\tau)&= \lp[R(\tau)\left(\begin{array}{c}\underline{w}_h(\tau) \\ \underline{\Phi}(\tau)\end{array}\right)\rp]_{3,4}\notag\\
&= s(\tau)\ul{w}_h(\tau) + c(\tau)\ul{\Phi}(\tau)\notag
\end{align}
where the subscripts at the end of the first and third lines denote the component of the vector and we recall that $c(\tau) := \cos(\Gamma(\re^{\tau}-1)), \, s(\tau) := \sin(\Gamma(\re^{\tau}-1))$. Next we use the estimates in the rotating frame to estimate different components of $\overline{W}$.  Using the estimates of Proposition \ref{p:aSu} we find 
\begin{align}
\|\bar w_h(\tau)\|_m &\leq  C\| \underline{w}_h(\tau)\|_m + |s(\tau)|\cdot \|\Phi(\tau)\|_m,\notag\\
&\leq C\re^{-(1/2+\mu)(\tau - \sigma)} \| \ul{W}(\sigma)\|_{m}\notag\\
&\leq C\re^{-(1/2+\mu)(\tau - \sigma)} \| W_0\|_{m}\\
\|\bar \Phi(\tau)\|_{m} &\leq C(\|\underline{w}_h(\tau)\|_{m} + \|\underline{\Phi}(\tau)\|_{m})\notag\\
&\leq C \re^{-(\mu+1/2)(\tau - \sigma)} \| W_0\|_{m} ,
\end{align}
since the rotational transformation does not affect the norm of $W_0$. The estimates for the gradients, as well as $L^q$ data, follow in a similar way.  
\end{proof}

\subsection{Baroclinic evolution}\label{ss:Abc}

\subsubsection{Unperturbed evolution of $\tl L(\tau)$}
First we estimate the unscaled velocity $\tl v = (\tl u, \tl \theta)^T$ under the linear evolution, $\tl v_t = M\tl v:=\Delta \tl v + \mb{P}J_{\Omega,\Gamma}\mb{P}\tl v$, in algebraically weighted spaces.  One readily finds that this equation generates a $C^0$-semigroup which we denote by $\re^{ M t}$.  The approach used in \cite[Prop. 4.4.2]{roussier2003stabilite}
 can be used to readily find
\begin{Proposition}\label{p:bcvel}
Let $m\geq0$, $\alpha\in \mb{N}^3$, $q\in [1,2],$ and $f\in (1-Q) L^2(m)^4$. If also $b^m f\in L^q(\mb{D})$ then 
\begin{align}
\|\p^\alpha \re^{ M (t-s)} f\|_{m} \leq \fr{C \re^{-4\pi^2(t - s)}}{(t-s)^{1/q - 1/2 + |\alpha_h|/2}a(t - s)^{(1/q - 1/2)/2 + \alpha_3|/2}}\| b^m f \|_{L^q(\mb{D})}.
\end{align}
\end{Proposition}
To transfer these results to the vorticity formulation, we need to show that the unscaled Biot-Savart law, which maps $\tl \omega$ to $\p_j\tl u$ for $j = 1,2,3$ is a bounded operator on the $b^m$-weighted $L^q$ spaces. We view this operator as a matrix Fourier multiplier and aim to apply the theory of weighted $L^q$ multipliers developed over the past several decades; see \cite{garcia1985weighted} for a review.  In particular, a weight function $\rho\geq0$ with $\rho\in L^1_{loc}(\R^n)$ is said to be Muckenhaupt class $A_q(\R^n)$ for some $q\in(1,\infty)$ if for all bounded cubes $Q$ with faces parallel to coordinate axes, 
\begin{align}
\sup_Q \lp(\fr{1}{|Q|} \int_Q \rho(x) dx\rp)\lp( \fr{1}{|Q|} \int_Q \rho(x)^{-1/(q-1)} dx  \rp)^{q-1}<\infty
\end{align}
where $|Q|$ denotes the Lebesque measure of $Q$. For our weight $b^m$ considered in $\R^2$ we first have
\begin{Lemma}\label{l:mhr2}
The weight $b^m:= (1+|x_h|^2)^{m/2}$ is of class $A_q(\R^2)$ for all $q>1$ and $0\leq m\leq 2(q-1)$.
\end{Lemma}
\begin{proof}
See \cite[Lem. 2.3 (v)]{farwig1997weighted}.
\end{proof}
Following \cite{sauer2015weighted} we can then define a Muckenhaupt class of weights, $A_q(\mb{D})$, for our spatial domain $\mb{D} = \R^2 \times \mb{T}$ in the same way as above; see \cite[pg. 338]{sauer2015weighted}. From this we then can conclude the following lemma:
\begin{Lemma}\label{l:mhb}
For all $q>1$ and $0\leq m<2(q-1)$, the weight $b^m:=(1+|x_h|^2)^{m/2}$ is of class $A_q(\mb{D})$.
\end{Lemma}
\begin{proof}
Let $Q = I_1\times I_2\times I_3$, for some bounded intervals $I_1,I_2\subset \R, I_3\subset \mb{T}$. Then we use the fact that $b^m$ does not depend on $x_3$ and Lemma \ref{l:mhr2} to obtain
\begin{align}
\sup_Q \lp(\fr{1}{|Q|} \int_Q b^m dx\rp)&\lp( \fr{1}{|Q|} \int_Q b^{-m/(q-1)} dx  \rp)^{q-1} = \notag\\
\sup_Q& \lp(\fr{1}{|I_1\times I_2|} \int_Q b^m dx_h\rp)\lp( \fr{1}{|I_1\times I_2|} \int_Q b^{-m/(q-1)} dx_h  \rp)^{q-1} \leq C 
\end{align}
for some positive constant $C>0$.
\end{proof}
For this class of weights, the results of \cite{sauer2015weighted} can be used to show that $\p_j\tl u$, for $j = 1,2,3$, are bounded operators in the $b^m$ weighted $L^q$ space. Namely we obtain the following:
\begin{Lemma}\label{l:bclem}
Let $\tl \omega\in (1-Q)L^2(m)$ satisfy $b^m \tl w\in L^q(\mb{D})$ for some $m\geq0$ and $q\in [3/2,\infty)$. Also let $\tilde{u}$ be determined by the Biot-Savart law given in \eqref{e:pbbs} above.  Then there exists a $C>0$ such that for for $j = 1,2,3$,
\begin{align}
\|b^m\p_j\tilde{u}\|_{L^q(\mb{D})}\leq C\|b^m\tilde \omega\|_{L^q(\mb{D})}.
\end{align}
\end{Lemma}
\begin{proof}
First we find observe that $\p_j \tilde{u}$, when defined on $\R^3$, is obtained from $\tilde{\omega}$ via a Fourier matrix multiplier $ B_j(k)$ with components of the form $\fr{k_ik_j}{|k|^2},\,\, k\in \R^3$. This implies that $ B_j(k)$ is bounded for $k\in \R^3$, is smooth away from the origin, and satisfies the Mikhlin condition $|k|^\beta |\p_{k}^\beta B_j(k)| \leq C$ for some $C>0$. Thus $ B_j(k)$ defines a bounded operator on $L^q(\R^3)$. 

Then, following \cite[Prop. 6]{sauer2015weighted}, we take a smooth bump function $\chi(k_3)$, with $\chi(0)=1$ and $\mathrm{supp}\{\chi\} \subset (-1/2,1/2)$, and define a smoothed operator $\tl B_j(k):= (1-\chi(k_3)) B_j(k).$ Observe that $\tl B_j(k) \equiv B_j(k)$ for $k\in \R^2\times \{\mb{Z}\diagdown\{0\}\}$. Then, since $\mb{D}$ is a locally compact abelian group (so that the Fourier transform can be defined on it), Lemma \ref{l:mhb} and the results of \cite[Prop. 4, Rmk. 5]{sauer2015weighted} (see also \cite{andersen2009restriction}) can be used to show that $B_j(k)$, defined on $k\in \R^2\times \{\mb{Z}\diagdown\{0\}\}$, defines a bounded operator on the weighted space $L^q(\ell):=\{f\in L^q(\mb{D})\,|\, \|b^\ell f \|_{L^q(\mb{D})}<\infty\}$ for all $0<\ell<2(q-1)$. We note that the derivatives $\p_{k_h}^\beta B_j(k)$ are also bounded multipliers, smooth away from the origin so that they also define bounded operators on $L^q(\ell)$. 

Next, letting $\beta\in \mb{N}^2$ denote a multi-index and $m_1\in \mb{N}$, it can readily be found that there exists constants $C,C'>0$ such that
\begin{align}\label{e:eqxh}
C'\sum_{|\beta|\leq m_1} |x_h^\beta f(x)|\leq (1+|x_h|^2)^{m_1/2}|f(x)| \leq C\sum_{|\beta|\leq m_1} |x_h^\beta f(x)|
\end{align}
Denoting $\mc{F}$ as the Fourier transform on $\mb{D}$, $m_1 =\lfloor m\rfloor $ the greatest integer below $m$ we then have
\begin{align}
\|b^m \p_j \tilde{u}\|_{L^q(\mb{D})} &=\|b^{m_1} \p_j \tilde{u}\|_{L^q(m-m_1)} \notag\\
&\leq C \sum_{|\beta|\leq m_1} \| x_h^\beta \p_j\tilde{u}\|_{L^q(m -m_1)}\notag\\
&\leq C \sum_{|\beta|\leq m_1} \| \mc{F}^{-1}\lp[ \p_{k_h}^\beta ( B_j(k) \widehat{\tilde{\omega}}(k)  )  \rp]\|_{L^q(m-m_1)}\notag\\
&\leq C \sum_{|\beta|\leq m_1} \|\mc{F}^{-1} \lp[ \p_{k_h}^\beta \widehat{\tilde{\omega}}(k)\rp]\|_{L^q(m-m_1)}\notag\\
&= C \sum_{|\beta|\leq m_1} \| x_{h}^\beta \tl\omega \|_{L^q(m-m_1)}\notag\\
&\leq C \|b^{m_1}\tl \omega\|_{L^q(m-m_1)} = C \|b^m\tl \omega\|_{L^q(\mb{D})}.\notag
\end{align}
Here we have used equivalencies \eqref{e:eqxh} in the second and last lines and the $L^q(m-m_1)$ boundedness of the multipliers $\p_{k_h}^\beta B_j$, for all $m-m_1<1<2(q-1)$, in the fourth line. 
\end{proof}

Next let $\tl W(\tau)$ be the solution of the linear equation  \eqref{e:linv2} in vorticity formulation. It can readily be found that this formulation also defines an evolutionary family of operators $\tl S_0(\tau,\sigma)$ on $L^2(m)^3\times H^1(m).$ Also let $\tilde{U}:= (\tilde{\mfu},\tilde\phi)^T$ be its corresponding velocity profile.  We can then obtain
\begin{Proposition}\label{p:sgest3}
Let $m\geq0$, $\alpha\in \mb{N}^3$, $q\in [1,2],$ and $f\in (1-Q) L^2(m)^3\times H^1(m)$. If also $b^m f\in L^q(\mb{D})$ then, for any $\delta\in (0,4\pi^2)$,
\begin{align}
\|\p^\alpha\tilde{S}_0(\tau,\sigma)f \|_{*,m} \leq \fr{C \re^{-(4\pi^2 - \delta)(\re^\tau - \re^\sigma)}}{a(\tau - \sigma)^{1/q - 1/2 + |\alpha_h|/2} a(\re^\tau - \re^\sigma)^{(1/q - 1/2)/2 + |\alpha_3|/2}} \lp( \|b^me_4^\perp f\|_{L^q(\mb{D})} + \|b^m \nabla e_4^T f\|_{L^q(\mb{D})}\rp)\end{align}
\end{Proposition}
\begin{proof}
Denote $\tl W(\tau) = \tilde{S}_0(\tau,\sigma)f$ and set $\tl W(\sigma) = f$. Recall the norm has the form $\| \tl W(\tau)\|^2_{*,m} = \|\tl w\|_m^2 + \|\tl \phi \|_{H^1(m)}^2.$  Furthermore, it readily seen that 
\begin{align}
\|\tl W(\tau)\|_{*,m} &\leq C\lp( \|\nabla\tilde{\mfu}\|_m + \|\tilde\phi\|_{H^1(m)}\rp)\notag\\
&\leq C\lp( \|\nabla_h\tilde{\mfu}\|_m + \|\p_3\tilde{\mfu}\|_m + \|\nabla_h\tilde\phi\|_{L^2(m)} + \|\p_3\tilde\phi\|_{L^2(m)} + \|\tilde\phi\|_{L^2(m)}  \rp)
\end{align}
We shall then estimate each of these pieces by relating them to their unscaled components which we can estimate using Proposition \ref{p:bcvel}.  We find, setting $\tl U  = (\tilde{\mfu}_1,\tilde{\mfu}_2,\tilde{\mfu}_3,\tl\phi)^T$ and using the fact that $\nabla$ and $M$ commute, the estimate
\begin{align}
\|\nabla_{\xi} \tl U(\tau)\|_m^2  &= \int_{\mb{D}} \lp| (1+|\fr{x_h}{\sqrt{1+t}}|^2)^{m/2}(1+t)^{1/2}\nabla_{x_h} \tl U(\fr{x_h}{\sqrt{1+t}},x_3,\log(1+t))\rp|^2 \fr{dx}{1+t}\notag\\
&= \int_{\mb{D}}\lp| (1+|\fr{x_h}{\sqrt{1+t}}|^2)^{m/2}(1+t)\nabla_{x_h} \tl v(x,t)\rp|^2 \fr{dx}{1+t}\notag\\
&= (1+t)\int_{\mb{D}}\lp| (1+|\fr{x_h}{\sqrt{1+t}}|^2)^{m/2} \re^{(t-s)M} \nabla_{x_h} \tl v(x,s)\rp|^2 dx\notag\\
&\leq \fr{C \re^{-8\pi^2(t - s)} (1+t)}{(t-s)^{2(1/q - 1/2) }a(t - s)^{(1/q - 1/2) }}
\lp(\int_{\mb{D}} |(1+|\fr{x}{\sqrt{1+s}}|^2)^{m/2} \nabla_{x_h} \tl v(x,s)|^q dx  \rp)^{2/q}\notag\\
&\leq \fr{C \re^{-8\pi^2(t - s)} (1+t)}{(t-s)^{2(1/q - 1/2 )}a(t - s)^{(1/q - 1/2) }}
\Bigg[\lp(\int_{\mb{D}} |(1+|\fr{x}{\sqrt{1+s}}|^2)^{m/2} \tl \omega(x,s)|^q  dx \rp)^{2/q}\notag\\
&\quad+ \lp(\int_{\mb{D}} |(1+|\fr{x}{\sqrt{1+s}}|^2)^{m/2} \nabla_{x_h} \tl \theta(x,s)|^q  dx \rp)^{2/q}\Bigg]
\notag
\end{align}
where $\nabla_{\xi} = (\p_{\xi_1},\p_{\xi_2})^T, \nabla_{x_h} = (\p_{x_1},\p_{x_2})^T$, and the last inequality was obtained using Lemma \ref{l:bclem}.
Changing back to scaled variables we then obtain for some small $\delta>0,$
\begin{align}
\|\nabla_h \tl U(\tau)\|_m^2&\leq 
C\fr{\re^{-8\pi^2(\re^\tau-\re^\sigma)}\re^{\tau - \sigma}\re^{2(1/2 - 1/q)(\tau - \sigma)}}{a(\tau - \sigma)^{2(1/q - 1/2)}a(\re^\tau - \re^\sigma)^{1/q - 1/2}}\lp( \|b^m\tilde{w}\|_{L^q(\mb{D})}^2 + \|b^m \nabla_{x_h} \tilde{\phi}\|_{L^q(\mb{D})}^2  \rp)\notag\\
&\leq C\fr{\re^{-(8\pi^2 - \delta)(\re^\tau - \re^\sigma)}}{a(\tau - \sigma)^{2(1/q - 1/2)}a(\re^\tau - \re^\sigma)^{1/q - 1/2}}\lp( \| b^m e_4^\perp f\|_{L^q(\mb{D})}^2 + \|b^m \nabla e_4^Tf\|_{L^q(\mb{D})}^2\rp).
\end{align}
The estimates for $\|\p_3 \tl U(\tau)\|_m$, $\|\tl \phi(\tau)\|_m$ and $|\alpha|>0$ follow in the same way.
\end{proof}

\subsubsection{Perturbed evolution of $\tl L(\tau) + \tl \Lambda$}
We now wish to prove Proposition \ref{p:btev}.  To begin we construct the evolutionary family of operators, which we shall denote as $\tl S(\tau,\sigma)$, generated by $\tl L(\tau) + \tl \Lambda$.  This can be done by studying the integral formulation, 
\begin{align}
\tl W(\tau) &= \tl S_0(\tau,\sigma) \tl W_0 + \int_\sigma^\tau  \tl S_0(\tau,s) (1-Q)\lp[  \overline{\mfu}_{app,h}\cdot\nabla_h\tl W + \tilde{\mfu}_h\cdot\nabla_h\bar W_{app} - \tl w_h\cdot\nabla_h\left(\begin{array}{c}\overline{\mfu}_{app}\\0\end{array}\right)   - \bar w_{app,h}\cdot\nabla_h \left(\begin{array}{c}\tilde{\mfu}\\0\end{array}\right) \rp]d\sigma\notag\\
&\qquad+\int_\sigma^\tau  \re^{\sigma/2}\tl S_0(\tau,s) (1-Q)\lp[  \overline{\mfu}_{app,3} \cdot\p_3\tl W  - \bar w_{app,3}\cdot\,\p_3 \left(\begin{array}{c}\tilde{\mfu}\\0\end{array}\right) \rp]d\sigma.
\end{align}
Existence and uniqueness can then be obtained via a fixed-point argument similar to ones in previous sections. To obtain the temporal estimates we use another Gronwall type argument. That is we show that the following quantity is bounded uniformly for all $\tau\geq \sigma\geq0$.  For ease of notation we set $c = 4\pi^2 -\delta,\, \tl c = c - \delta'$ with $\delta'>0$ to be chosen so that $\tl c>0$,
\begin{align}
\tl \Psi(\tau,\sigma)&= \re^{\tl c(\re^{\tau}- \re^{\sigma})} \Bigg( \| \tl W(\tau)\|_{*,m} + a(\tau - \sigma)^{1/2} \| \nabla_h \tl W(\tau)\|_{*,m} + a(\re^\tau - \re^\sigma)^{1/2} \|\p_3 \tl W(\tau)\|_{*,m} \Bigg).    
\end{align}  

Using estimates similar to those used in the proof of Theorem \ref{t:alg-as}, as well as the fact that $\|\nabla \bar\phi\|_m\leq C \|\overline{\Phi}\|_m$, we then estimate with $q = 3/2$,
\begin{align}
\re^{\tl c(\re^\tau-\re^\sigma)}\| \tl W(\tau)\|_{*,m}&\leq C\re^{-\delta(\re^\tau - \re^\sigma)}\|\tl W_0\|_m \notag\\
&+ C_{A,B}\int_\sigma^\tau \fr{\re^{(\tl c - c)(\re^\tau-\re^s)}}{a(\tau - s)^{1/q-1/2}a(\re^{\tau} - \re^s)^{(1/q - 1/2)/2}} \Bigg( \|\nabla_h \tl W\|_{*,m} + \| \tl W\|_{*,m} + e^{-s/2}\|\tl W\|_{*,m}   
\Bigg)ds\notag\\
&+ C_{A,B}\int_\sigma^\tau \fr{\re^{(\tl c - c)(\re^\tau-\re^s)}}{a(\tau - s)^{1/q-1/2}a(\re^{\tau} - \re^s)^{(1/q - 1/2)/2}}\re^{s/2}\Bigg( \re^{-s/2} \|\p_3 \tl W\|_{*,m} + \|\tl W\|_{*,m}    \Bigg)\notag\\
&\leq C\|\tl W_0\|_{*,m} + C_{A,B}\int_\sigma^\tau \re^{-\delta'(\re^\tau-\re^s)}g(\tau,s,\sigma) \tl\Psi(s,\sigma)ds,
\end{align}
where $C_{A,B}$ is a constant dependent on the approximate solution amplitudes $A, B_1,B_2$ and
$$
g(\tau,s,\sigma) = \fr{1+\re^{s/2}+ a(s-\sigma)^{-1/2} + a(\re^s-\re^\sigma)^{-1/2}}{a(\tau - s)^{1/q-1/2}a(\re^{\tau} - \re^s)^{(1/q - 1/2)/2}}.
$$

In a similar manner we find
\begin{align}
a(\tau - \sigma)^{1/2} \re^{\tl c(\re^{\tau}-\re^\sigma)} \|\nabla_h \tl W(\tau)\|_{*,m} &\leq C\re^{-\delta(\re^\tau - \re^\sigma)}\|\tl W_0\|_{*,m} +
C_{A,B} \int_\sigma^\tau \re^{(\tl c - c)(\re^\tau-\re^s)}\fr{a(\tau-\sigma)^{1/2}}{a(\tau-s)^{1/2}}g(\tau,s,\sigma)\tl\Psi(s,\sigma)ds\notag\\
a(\re^\tau - \re^\sigma)^{1/2} \re^{\tl c(\re^{\tau}-\re^\sigma)} \|\p_3 \tl W(\tau)\|_{*,m} &\leq C\re^{-\delta(\re^\tau - \re^\sigma)}\|\tl W_0\|_{*,m} + 
 C_{A,B}\int_\sigma^\tau \re^{(\tl c - c)(\re^\tau-\re^s)}\fr{a(\re^\tau-\re^\sigma)^{1/2}}{a(\re^{\tau}-\re^s)^{1/2}}g(\tau,s,\sigma)\tl\Psi(\tau,s,\sigma)ds.
\end{align}
Combing these all together we obtain 
\begin{align}
\tl\Psi(\tau,\sigma) &\leq C \re^{-\delta(\re^\tau - \re^\sigma)}\|\tl W_0\|_{*,m} + C_{A,B} \int_\sigma^\tau \re^{-\delta'(\re^\tau-\re^s)} \tl g(\tau,s,\sigma)\tl\Psi(\tau,\sigma)  ds,
\end{align}
with 
$$
\tl g(\tau,s,\sigma) = g(\tau,s,\sigma)\lp(1+\fr{a(\tau-\sigma)^{1/2}}{a(\tau-s)^{1/2}}+ \fr{a(\re^\tau-\re^\sigma)^{1/2}}{a(\re^{\tau}-\re^s)^{1/2}}    \rp).
$$
Then applying Gronwall's inequality we obtain
\begin{align}
\tl\Psi(\tau,\sigma)&\leq C_{A,B}\|\tl W_0\|_{*,m} \mathrm{exp}\lp[\int_\sigma^\tau \re^{-\delta'(\re^\tau-\re^s)}\tl g(\tau,s,\sigma) ds    \rp].\notag\\
&\leq  \tl C_{A,B}\|\tl W_0\|_{*,m}  \re^{\tl C}.
\end{align}
where $\tl C>0$ is a constant dependent on $\delta'$ and $q$. The integral inside the exponential can be bounded, uniformly in $\tau\geq\sigma$, using the argument of \cite[Prop. 4.5.2]{roussier2003stabilite} and   using the fact that the singularities at $s = \tau$ have the form $a(\tau - s)^{1/q}$ or $a(\re^\tau - \re^s)^{1/q}$ and those at $s = \sigma$ take the form $a(s - \sigma)^{1/2}$ or $a(\re^s-\re^\sigma)^{1/2}$.

\bibliography{Stable-Strat}

\begin{thebibliography}{10}

\bibitem{andersen2009restriction}
{\sc K.~Andersen and P.~Mohanty}, {\em Restriction and extension of {F}ourier
  multipliers between weighted ${L}^p$ spaces on $\mathbb{R}^n$ and
  $\mathbb{T}^n$}, Proceedings of the American Mathematical Society, 137
  (2009), pp.~1689--1697.

\bibitem{babin1998nonlinear}
{\sc A.~Babin, A.~Mahalov, and B.~Nicolaenko}, {\em On nonlinear baroclinic
  waves and adjustment of pancake dynamics}, Theoretical and Computational
  Fluid Dynamics, 11 (1998), pp.~215--235.

\bibitem{babin1999regularity}
\leavevmode\vrule height 2pt depth -1.6pt width 23pt, {\em On the regularity of
  three-dimensional rotating {E}uler--{B}oussinesq equations}, Mathematical
  Models and Methods in Applied Sciences, 9 (1999), pp.~1089--1121.

\bibitem{babin2002fast}
\leavevmode\vrule height 2pt depth -1.6pt width 23pt, {\em Fast singular
  oscillating limits of stably-stratified 3d {E}uler and {N}avier--{S}tokes
  equations and ageostrophic wave fronts}, Large-scale atmosphere-ocean
  dynamics, 1 (2002), pp.~126--201.

\bibitem{boffetta2012two}
{\sc G.~Boffetta and R.~E. Ecke}, {\em Two-dimensional turbulence}, Annual
  Review of Fluid Mechanics, 44 (2012), pp.~427--451.

\bibitem{cao2007global}
{\sc C.~Cao and E.~S. Titi}, {\em Global well-posedness of the
  three-dimensional viscous primitive equations of large scale ocean and
  atmosphere dynamics}, Annals of Mathematics,  (2007), pp.~245--267.

\bibitem{carlen1996optimal}
{\sc E.~A. Carlen and M.~Loss}, {\em Optimal smoothing and decay estimates for
  viscously damped conservation laws, with applications to the 2-d
  {N}avier-{S}tokes equation}, Duke Math. J, 81 (1996), pp.~135--157.

\bibitem{charney1948scale}
{\sc J.~Charney}, {\em On the scale of atmospheric motions}, Geofysiske
  publikasjoner, Gr{\o}ndahl \& s{\o}ns boktr., I kommission hos Cammermeyer,
  1948.

\bibitem{charve2005convergence}
{\sc F.~Charve}, {\em Convergence of weak solutions for the primitive system of
  the quasigeostrophic equations}, Asymptotic Analysis, 42 (2005),
  pp.~173--209.

\bibitem{charve2005global}
\leavevmode\vrule height 2pt depth -1.6pt width 23pt, {\em Global
  well-posedness and asymptotics for a geophysical fluid system},
  Communications in Partial Differential Equations, 29 (2005), pp.~1919--1940.

\bibitem{charve2006asymptotics}
\leavevmode\vrule height 2pt depth -1.6pt width 23pt, {\em Asymptotics and
  vortex patches for the quasigeostrophic approximation}, Journal de
  Math{\'e}matiques Pures et Appliqu{\'e}es, 85 (2006), pp.~493--539.

\bibitem{charve2008global}
\leavevmode\vrule height 2pt depth -1.6pt width 23pt, {\em Global
  well-posedness for the primitive equations with less regular initial data},
  Annales de la Facult{\'e} des Sciences de Toulouse, Math{\'e}matiques
  S{\'e}rie 6 (2008), pp.~221--238.

\bibitem{charve2011global}
{\sc F.~Charve, V.-S. Ngo, et~al.}, {\em Global existence for the primitive
  equations with small anisotropic viscosity}, Revista Matem{\'a}tica
  Iberoamericana, 27 (2011), pp.~1--38.

\bibitem{chemin2006mathematical}
{\sc J.-Y. Chemin, B.~Desjardins, I.~Gallagher, and E.~Grenier}, {\em
  Mathematical geophysics: An introduction to rotating fluids and the
  {N}avier-{S}tokes equations}, vol.~32, Oxford University Press on Demand,
  2006.

\bibitem{farwig1997weighted}
{\sc R.~Farwig and H.~Sohr}, {\em Weighted {$L^q$}-theory for the {S}tokes
  resolvent in exterior domains}, Journal of the Mathematical Society of Japan,
  49 (1997), pp.~251--288.

\bibitem{favier2014inverse}
{\sc B.~Favier, L.~Silvers, and M.~Proctor}, {\em Inverse cascade and symmetry
  breaking in rapidly rotating {B}oussinesq convection}, Physics of Fluids, 26
  (2014), p.~096605.

\bibitem{gallay2009global}
{\sc T.~Gallay and V.~Roussier-Michon}, {\em Global existence and long-time
  asymptotics for rotating fluids in a $3{D}$ layer}, Journal of Mathematical
  Analysis and Applications, 360 (2009), pp.~14--34.

\bibitem{gallay2002invariant}
{\sc T.~Gallay and C.~E. Wayne}, {\em Invariant manifolds and the long-time
  asymptotics of the {N}avier-{S}tokes and vorticity equations on ${R}^2$},
  Archive for Rational Mechanics and Analysis, 163 (2002), pp.~209--258.

\bibitem{gallay2005global}
\leavevmode\vrule height 2pt depth -1.6pt width 23pt, {\em Global stability of
  vortex solutions of the two-dimensional {N}avier-{S}tokes equation},
  Communications in Mathematical Physics, 255 (2005), pp.~97--129.

\bibitem{garcia1985weighted}
{\sc J.~Garc{\'\i}a-Cuerva and J.-L.~R. de~Francia}, {\em Weighted norm
  inequalities and related topics}, vol.~104, Elsevier, 1985.

\bibitem{gill1982atmosphere}
{\sc A.~E. Gill}, {\em Atmosphere-ocean dynamics}, vol.~30, Academic Press,
  1982.

\bibitem{greenleaf1981principal}
{\sc A.~Greenleaf}, {\em Principal curvature and harmonic-analysis}, Indiana
  University Mathematics Journal, 30 (1981), pp.~519--537.

\bibitem{guillen2001anisotropic}
{\sc F.~Guill{\'e}n-Gonz{\'a}lez, N.~Masmoudi, M.~Rodr{\'\i}guez-Bellido,
  et~al.}, {\em Anisotropic estimates and strong solutions of the primitive
  equations}, Differential and Integral Equations-Athens, 14 (2001),
  pp.~1381--1408.

\bibitem{herring_metais_1989}
{\sc J.~R. Herring and O.~M\'{e}tais}, {\em Numerical experiments in forced
  stably stratified turbulence}, Journal of Fluid Mechanics, 202 (1989),
  pp.~97--115.

\bibitem{holm1989lyapunov}
{\sc D.~D. Holm and B.~Long}, {\em Lyapunov stability of ideal stratified fluid
  equilibria in hydrostatic balance}, Nonlinearity, 2 (1989), p.~23.

\bibitem{hsia2007stratified}
{\sc C.-H. Hsia, T.~Ma, and S.~Wang}, {\em Stratified rotating {B}oussinesq
  equations in geophysical fluid dynamics: Dynamic bifurcation and periodic
  solutions}, Journal of {M}athematical {P}hysics, 48 (2007), p.~065602.

\bibitem{hu2003primitive}
{\sc C.~Hu, R.~Temam, and M.~Ziane}, {\em The primitive equations on the large
  scale ocean under the small depth hypothesis}, Discrete and Continuous
  Dynamical Systems, 9 (2003), pp.~97--132.

\bibitem{iwabuchi2016global}
{\sc T.~Iwabuchi, A.~Mahalov, and R.~Takada}, {\em Global solutions for the
  incompressible rotating stably stratified fluids}, Mathematische Nachrichten,
  290 (2017), pp.~613--631.

\bibitem{julien2012statistical}
{\sc K.~Julien, A.~Rubio, I.~Grooms, and E.~Knobloch}, {\em Statistical and
  physical balances in low {R}ossby number {R}ayleigh--{B}{\'e}nard
  convection}, Geophysical \& Astrophysical Fluid Dynamics, 106 (2012),
  pp.~392--428.

\bibitem{knobloch2015spatial}
{\sc E.~Knobloch}, {\em Spatial localization in dissipative systems}, Annual
  Review of Condensed Matter Physics, 6 (2015), pp.~325--359.

\bibitem{koba2012global}
{\sc A.~Koba, H.~Mahalov and T.~Yoneda}, {\em Global well-posedness of the
  rotating {N}avier-{S}tokes-{B}oussinesq equations with stratification
  effects}, Adv. Math. Sci. Appl., 22 (2012), pp.~61--90.

\bibitem{kraichnan1980two}
{\sc R.~H. Kraichnan and D.~Montgomery}, {\em Two-dimensional turbulence},
  Reports on Progress in Physics, 43 (1980), p.~547.

\bibitem{li2007some}
{\sc J.~Li and S.~Wang}, {\em Some mathematical and numerical issues in
  geophysical fluid dynamics and climate dynamics}, arXiv preprint
  arXiv:0711.1886,  (2007).

\bibitem{lions1992new}
{\sc J.-L. Lions, R.~Temam, and S.~Wang}, {\em New formulations of the
  primitive equations of atmosphere and applications}, Nonlinearity, 5 (1992),
  p.~237.

\bibitem{lions1992equations}
{\sc J.-L. Lions, R.~Temam, and S.~Wang}, {\em On the equations of the
  large-scale ocean}, Nonlinearity, 5 (1992), p.~1007.

\bibitem{lorenz1986existence}
{\sc E.~N. Lorenz}, {\em On the existence of a slow manifold}, Journal of the
  Atmospheric Sciences, 43 (1986), pp.~1547--1558.

\bibitem{lorenz1987nonexistence}
{\sc E.~N. Lorenz and V.~Krishnamurthy}, {\em On the nonexistence of a slow
  manifold}, Journal of the Atmospheric Sciences, 44 (1987), pp.~2940--2950.

\bibitem{majda2003introduction}
{\sc A.~Majda}, {\em Introduction to PDEs and Waves for the Atmosphere and
  Ocean}, vol.~9, American Mathematical Soc., 2003.

\bibitem{rademacher2017}
{\sc M.~Oliver, G.~Badin, C.~Franzke, and J.~Rademacher}, {\em {M}ulti-scale
  methods for geophysical flows}, preprint, 2017.

\bibitem{pazy2012semigroups}
{\sc A.~Pazy}, {\em Semigroups of linear operators and applications to partial
  differential equations}, vol.~44, Springer Science \& Business Media, 2012.

\bibitem{pedlosky2013geophysical}
{\sc J.~Pedlosky}, {\em Geophysical fluid dynamics}, Springer-Verlag, 1979.

\bibitem{petcu2005renormalization}
{\sc M.~Petcu, R.~Temam, and D.~Wirosoetisno}, {\em Renormalization group
  method applied to the primitive equations}, Journal of Differential
  Equations, 208 (2005), pp.~215--257.

\bibitem{riley2000fluid}
{\sc J.~J. Riley and M.-P. Lelong}, {\em Fluid motions in the presence of
  strong stable stratification}, Annual Review of Fluid Mechanics, 32 (2000),
  pp.~613--657.

\bibitem{roussier2003stabilite}
{\sc V.~Roussier-Michon}, {\em Sur la stabilit{\'e} des ondes sph{\'e}riques et
  le mouvement d'un fluide entre deux plaques infinies}, PhD thesis, Paris 11,
  2003.

\bibitem{rubio2014upscale}
{\sc A.~M. Rubio, K.~Julien, E.~Knobloch, and J.~B. Weiss}, {\em Upscale energy
  transfer in three-dimensional rapidly rotating turbulent convection},
  Physical Review Letters, 112 (2014), p.~144501.

\bibitem{sauer2015weighted}
{\sc J.~Sauer}, {\em Weighted resolvent estimates for the spatially periodic
  {S}tokes equations}, Annali Dell'Universita'Di Ferrara, 61 (2015),
  pp.~333--354.

\bibitem{sprague2006numerical}
{\sc M.~Sprague, K.~Julien, E.~Knobloch, and J.~Werne}, {\em Numerical
  simulation of an asymptotically reduced system for rotationally constrained
  convection}, Journal of Fluid Mechanics, 551 (2006), pp.~141--174.

\bibitem{temam2010stability}
{\sc R.~Temam and D.~Wirosoetisno}, {\em Stability of the slow manifold in the
  primitive equations}, SIAM Journal on Mathematical Analysis, 42 (2010),
  pp.~427--458.

\bibitem{temam2011slow}
{\sc R.~Temam and D.~Wirosoetisno}, {\em Slow manifolds and invariant sets of
  the primitive equations}, Journal of the Atmospheric Sciences, 68 (2011),
  pp.~675--682.

\bibitem{temam2005some}
{\sc R.~Temam and M.~Ziane}, {\em Some mathematical problems in geophysical
  fluid dynamics}, Handbook of Mathematical Fluid Dynamics, 3 (2005),
  pp.~535--658.

\end{thebibliography}
\bibliographystyle{siam}

\end{document}